\documentclass[11pt, a4paper]{article}

\usepackage[english]{babel}
\usepackage{dutchcal}
\usepackage[T1]{fontenc}

\usepackage{amsmath}
\usepackage{amssymb}
\usepackage{amsthm}

\usepackage{tikz}
\usepackage{tikz-cd}

\usepackage{hyperref}
\usepackage[capitalize]{cleveref}
\usepackage{footmisc}
\usepackage{bbm}
\usepackage{mathrsfs}
\usepackage{extarrows}
\usepackage{enumitem}
\usepackage{stmaryrd}
\usepackage{microtype}

\newcommand{\llpar}{\llparenthesis}
\newcommand{\rrpar}{\rrparenthesis}

\newcommand{\ab}{\mathrm{ab}}

\newcommand{\aff}{\mathrm{aff}}

\newcommand{\KM}{\mathrm{KM}}
\newcommand{\Ell}{\mathrm{Ell}}

\newcommand{\Tori}{\mathrm{Tori}}

\newcommand{\et}{\mathrm{\acute{e}t}}
\newcommand{\fin}{\mathrm{fin}}

\newcommand{\gl}{\mathrm{gl}}

\DeclareMathOperator{\res}{res}
\newcommand{\ori}{\mathrm{or}}

\newcommand{\pre}{\mathrm{pre}}
\newcommand{\st}{\textup{st}}

\newcommand{\Tate}{\mathrm{Tate}}

\DeclareMathOperator{\tr}{tr}

\DeclareMathOperator{\KU}{KU}

\newcommand{\gKU}{\mathbf{KU}}
\DeclareMathOperator{\KO}{KO}
\newcommand{\gKO}{\mathbf{KO}}

\newcommand{\Sph}{\mathbf{S}}
\newcommand{\gTMF}{\mathbf{TMF}}

\DeclareMathOperator{\TMF}{TMF}

\DeclareMathOperator{\Ab}{Ab}

\DeclareMathOperator{\Aff}{Aff}

\DeclareMathOperator{\CAlg}{CAlg}

\DeclareMathOperator{\Cat}{Cat}

\DeclareMathOperator{\Fun}{Fun}

\DeclareMathOperator{\Glo}{Glo}

\DeclareMathOperator{\Mod}{Mod}

\newcommand{\PrLst}{\mathrm{Pr}^L_{\mathrm{st}}}

\DeclareMathOperator{\QCoh}{QCoh}

\DeclareMathOperator{\Ind}{Ind}

\DeclareMathOperator{\Span}{Span}
\DeclareMathOperator{\Spec}{Spec}

\DeclareMathOperator{\Shv}{Shv}
\newcommand{\Stk}{{\mathrm{Stk}}}

\DeclareMathOperator{\Sp}{Sp}

\DeclareMathOperator{\SpDMnc}{SpDM^{nc}}

\DeclareMathOperator{\SpSch}{SpSch}
\newcommand{\Spc}{\mathcal{S}}
\newcommand{\spaces}{\mathcal{S}}

\DeclareMathOperator{\Tot}{Tot}

\newcommand{\Fin}{\mathcal{F}\mathrm{in}}

\DeclareMathOperator{\Aut}{Aut}
\DeclareMathOperator{\colim}{colim}

\newcommand{\id}{\mathrm{id}}
\DeclareMathOperator{\Map}{Map}
\DeclareMathOperator{\map}{map}

\DeclareMathOperator{\Hom}{Hom}

\newcommand{\op}{\mathrm{op}}

\newcommand{\A}{\mathsf{A}}

\newcommand{\B}{\mathsf{B}}

\newcommand{\BT}{\mathbf{BT}}

\newcommand{\cC}{\mathcal{C}}

\newcommand{\D}{\mathsf{D}}
\newcommand{\calD}{\mathcal{D}}
\newcommand{\E}{\mathbf{E}}
\newcommand{\sfE}{\mathsf{E}}

\newcommand{\calF}{\mathcal{F}}

\newcommand{\G}{\mathbf{G}}
\newcommand{\calG}{\mathcal{G}}

\newcommand{\h}{\mathrm{h}}

\newcommand{\M}{\mathsf{M}}
\newcommand{\sfM}{\mathsf{M}}

\renewcommand{\O}{\mathcal{O}}

\renewcommand{\P}{\mathbf{P}}

\newcommand{\calP}{\mathcal{P}}
\newcommand{\bfP}{\mathbf{P}}
\newcommand{\Q}{\mathbf{Q}}

\newcommand{\R}{\mathsf{R}}

\newcommand{\sfS}{\mathsf{S}}
\newcommand{\sfR}{\mathsf{R}}

\newcommand{\T}{\mathrm{T}}

\newcommand{\Z}{\mathbf{Z}}

\newcommand{\calQ}{\mathcal{Q}}

\newcommand{\calC}{\mathcal{C}}
\newcommand{\calK}{\mathcal{K}}
\newcommand{\Cc}{\mathcal{C}}
\newcommand{\calE}{\mathcal{E}}
\newcommand{\bfG}{\mathbf{G}}
\newcommand{\bfE}{\mathbf{E}}

\newcommand{\bbZ}{\mathbb{Z}}
\newcommand{\bfZ}{\mathbf{Z}}
\newcommand{\calL}{\mathcal{L}}
\newcommand{\bfQ}{\mathbf{Q}}

\newcommand{\bbC}{\mathbb{C}}

\newcommand{\Th}{\operatorname{Th}}
\newcommand{\Psh}{\mathrm{Psh}}

\def\hyp{{\hbox{-}}}

\newcommand{\sfZ}{\mathsf{Z}}
\newcommand{\sfX}{\mathsf{X}}
\newcommand{\sfY}{\mathsf{Y}}
\newcommand{\calO}{\mathcal{O}}

\newcommand{\sfA}{\mathsf{A}}

\newcommand{\bfB}{\mathbf{B}}
\newcommand{\bfR}{\mathbf{R}}

\newcommand{\nc}{\mathrm{nc}}
\newcommand{\ev}{\mathrm{ev}}

\newcommand{\dcat}{\mathcal{D}}
\newcommand{\ccat}{\mathcal{C}}

\newcommand{\bbone}{\mathbbm{1}}

\newcommand{\jdiagram}{\mathsf{J}}

\newcommand{\spherespectrum}{\mathbf{S}}

\newcommand{\abglobalspaces}{\Spc^{\gl}_{\ab}}

\newcommand{\spectra}{{\mathrm{Sp}}}
\newcommand{\PshT}{{\calP}}

\newcommand{\einfty}{\mathbf{E}_\infty}
\newcommand{\dual}[1]{\widehat{#1}}

\newcommand{\pdivaff}{\mathbf{P}\mathrm{Div}_{\mathrm{aff}}}
\newcommand{\pdivpreaff}{\mathbf{P}\mathrm{Div}_{\mathrm{aff}}^{\mathrm{pre}}}
\newcommand{\pdivoraff}{\mathbf{P}\mathrm{Div}_{\mathrm{aff}}^{\mathrm{or}}}

\newcommand{\pdiv}{\mathbf{P}\mathrm{Div}}

\newcommand{\pdivorstk}{\mathbf{P}\mathrm{Div}^{\mathrm{or}}}
\newcommand{\pdivstk}{\mathbf{P}\mathrm{Div}}
\newcommand{\pdivprestk}{\mathbf{P}\mathrm{Div}^{\pre}}

\newcommand{\frakp}{\mathfrak{p}}

\newcommand{\fmodule}{\mathcal{F}}

\newcommand{\BH}{\mathbf{B}H}
\newcommand{\BK}{\mathbf{B}K}
\newcommand{\BG}{\mathbf{B}G}

\newcommand{\sfN}{\mathsf{N}}

\newcommand{\image}{\operatorname{im}}

\newcommand{\Top}{\mathrm{Top}}

\newcommand{\zSpec}{\operatorname{\mathsf{Spec}}}
\newcommand{\cart}{\mathrm{cart}}
\newcommand{\height}{\operatorname{ht}}

\newcommand{\ol}[1]{\overline{#1}}

\newcommand{\bs}{{-}}
\newcommand{\ul}[1]{\underline{#1}}

\newcommand{\al}{\alpha}

\newcommand{\Ga}{\Gamma}

\newcommand{\Qloc}{\calQ}

\DeclareMathOperator{\Ar}{Ar}

\newcommand{\all}{\mathrm{all}}

\newcommand{\End}{\operatorname{End}}

\newcommand{\CMon}{\operatorname{CMon}}
\newcommand{\Mon}{\CMon}

\newcommand{\Inj}{\mathrm{Inj}}
\newcommand{\LocSys}{\mathrm{LocSys}}
\newcommand{\PrL}{\mathrm{Pr}^{\mathrm{L}}}

\newcommand{\awayfromS}{{S{\perp}\pi}}

\usepackage{mathtools}

\newtheorem*{theorem*}{Theorem}

\theoremstyle{plain}\numberwithin{equation}{subsection}
\newtheorem{theorem}[equation]{Theorem}
\Crefname{theorem}{{Th}.\!\!}{{Ths}.\!\!}
\newtheorem{theoremalph}{Theorem}

\Crefname{theoremalph}{{Th}.\!\!}{{Ths}.\!\!}
\newtheorem{coralph}[theoremalph]{Corollary}

\Crefname{coralph}{{Cor}.\!\!}{{Cors}.\!\!}

\Crefname{defalph}{{Df}.\!\!}{{Dfs}.\!\!}

\Crefname{conjalph}{{Conj}.\!\!}{{Conjs}.\!\!}

\Crefname{problem}{{Prb}.\!\!}{{Prbs}.\!\!}
\newtheorem{prop}[equation]{Proposition}
\Crefname{prop}{{Pr}.\!\!}{{Prs}.\!\!}
\newtheorem{lemma}[equation]{Lemma}
\Crefname{lemma}{{Lm}.\!\!}{{Lms}.\!\!}
\newtheorem{cor}[equation]{Corollary}
\Crefname{cor}{{Cor}.\!\!}{{Cors}.\!\!}

\Crefname{conjecture}{{Conj}.\!\!}{{Conjs}.\!\!}

\theoremstyle{definition}\numberwithin{equation}{subsection}
\newtheorem{mydef}[equation]{Definition}
\Crefname{mydef}{{Df}.\!\!}{{Dfs}.\!\!}

\Crefname{defn}{{Df}.\!\!}{{Dfs}.\!\!}

\Crefname{recall}{{Rcl}.\!\!}{{Rcls}.\!\!}
\newtheorem{construction}[equation]{Construction}
\Crefname{construction}{{Con}.\!\!}{{Cons}.\!\!}
\newtheorem{ass}[equation]{Assumption}
\Crefname{ass}{{As}.\!\!}{{As}.\!\!}
\newtheorem{notation}[equation]{Notation}
\Crefname{notation}{{Nt}.\!\!}{{Nts}.\!\!}

\Crefname{situation}{{St}.\!\!}{{Sts}.\!\!}

\theoremstyle{remark}\numberwithin{equation}{subsection}
\newtheorem{example}[equation]{Example}
\Crefname{example}{{Ex}.\!\!}{{Exs}.\!\!}
\newtheorem{ex}[equation]{Example}
\Crefname{ex}{{Ex}.\!\!}{{Exs}.\!\!}

\Crefname{nonexample}{{NonEx}.\!\!}{{NonEx}.\!\!}

\Crefname{claim}{{Clm}.\!\!}{{Clms}.\!\!}
\newtheorem{remark}[equation]{Remark}
\Crefname{remark}{{Rmk}.\!\!}{{Rmks}.\!\!}
\newtheorem{rmk}[equation]{Remark}
\Crefname{rmk}{{Rmk}.\!\!}{{Rmks}.\!\!}

\Crefname{idea}{{Id}.\!\!}{{Ids}.\!\!}

\Crefname{warn}{{Warn}.\!\!}{{Warns}.\!\!}

\Crefname{question}{{Qn}.\!\!}{{Qns}.\!\!}

\Crefname{figure}{{Fig.}\!\!}{{Figs.}\!\!}
\Crefname{footnote}{{Fn.}\!\!}{{Fn.}\!\!}

\Crefname{part}{{\textsection}\!\!}{{\textsection}\!\!}
\Crefname{chapter}{{\textsection}\!\!}{{\textsection}\!\!}
\Crefname{section}{{\textsection}\!\!}{{\textsection}\!\!}
\Crefname{subsection}{{\textsection}\!\!}{{\textsection}\!\!}
\Crefname{appendix}{{\textsection}\!\!}{{\textsection}\!\!}

\usepackage{parskip}

\makeatletter
\def\thm@space@setup{%
  \thm@preskip=\parskip \thm@postskip=0pt
}
\makeatother

\begingroup
    \makeatletter
    \@for\theoremstyle:=mydef,remark,plain,theorem,defn\do{%
        \expandafter\g@addto@macro\csname th@\theoremstyle\endcsname{%
            \addtolength\thm@preskip\parskip
            }%
        }
\endgroup

\usetikzlibrary{spath3}
\tikzset{between/.style n args={2}{/tikz/spath/at end path construction={
    \tikzset{spath/split at keep middle={current}{#1}{#2}}
}}}

\begin{document}
\title{Ambidextrous global spectra and tempered cohomology}
\author{William Balderrama\footnote{\url{williamb@math.uni-bonn.de}}, Jack Morgan Davies\footnote{\url{davies@uni-wuppertal.de}}, and Sil Linskens\footnote{\url{sil.linskens@mathematiks.uni-regensburg.de}}}
\date{March 19, 2026}
\maketitle

\begin{abstract}
We introduce generalizations of global equivariant spectra which encode globally equivariant cohomology theories equipped with additional \emph{transfers}, such as the deflation maps present in equivariant topological $K$-theory. We call these \emph{$\mathcal{Q}$-ambidextrous global spectra}, where $\mathcal{Q}$ is a parameter encoding which additional transfers one allows.

As our main example, we prove that the tempered cohomology theory associated with an oriented $\bfP$-divisible group, constructed by Lurie, is represented by a \emph{$\pi$-ambidextrous} global $\bfE_\infty$ ring spectrum, encoding transfers along all relatively $\pi$-finite maps of global spaces. This is established by means of a general parametrized decategorification process, perhaps of independent interest, that produces $\mathcal{Q}$-ambidextrous global spectra from suitable global families of stable $\infty$-categories. By allowing $\mathcal{Q}$ to vary, we are able to coherently encode the fact that non-invertible morphisms of oriented $\bfP$-divisible groups induce maps of tempered theories that only commute with certain transfers.

With these $\pi$-ambidextrous enhancements in hand, we explore the fundamental properties of tempered theories as equivariant stable homotopy types. We construct a well-behaved \emph{$F$-global homology theory} for any $\pi$-finite space $F$, with good base change properties. Taking $F = \bfB H$ for a finite group $H$, this establishes general base change results for the geometric fixed points of tempered theories. We use this to compute the $H$-geometric fixed points of tempered theories, showing that they vanish for $H$ nonabelian and admit a simple algebro-geometric model when $H$ is abelian, with identifiable blueshift properties.
\end{abstract}

\setcounter{tocdepth}{2}
\tableofcontents

\section{Introduction}
Let $G$ be a finite group. Atiyah and Segal \cite{equivktheory,atiyahsegal} introduced the \emph{$G$-equivariant complex $K$-theory} of a compact $G$-space $X$:
\[
KU_G^0(X) = \raisebox{.6em}{$\left\{\begin{array}{cc}\text{$G$-equivariant vector} \\
\text{bundles over $X$}\end{array}\right\}$} \,\diagup\, \raisebox{-.4em}{\text{stable equivalence.}}
\]
In \cite{ec3}, Lurie generalized this drastically by introducing \emph{tempered cohomology theories}, which share the essential properties of equivariant complex $K$-theory (such as Atiyah--Segal \cite{atiyahsegal} completion theorems and a Hopkins--Kuhn--Ravenel \cite{hkr} character theory), associated with \emph{oriented $\bfP$-divisible groups} over $\bfE_\infty$ rings. Examples include equivariant complex $K$-theory itself, arising from the torsion subgroup of the multiplicative group scheme over $KU$, and higher chromatic analogues such as \emph{equivariant elliptic cohomology theories}, arising from the torsion subgroups of oriented elliptic curves.

This article continues the study of tempered cohomology theories. One of our primary goals is to capture additional structure present in tempered cohomology. To motivate this, consider the classical example of equivariant complex $K$-theory. The $G$-equivariant complex $K$-theory of a $G$-space $X$ is built out of its values over the cells $G/H_+ \wedge S^n$. The fundamental computation here identifies the cohomology of the $0$-cells
\[
KU_G^0(G/H) \cong RU(H)
\]
with the complex representation ring of $H$. As the group $G$ varies, the representation rings $RU(G)$ carry a significant amount of algebraic structure: associated with any subgroup inclusion $i\colon H \subset G$, one has \emph{restriction} $i^\ast$ and \emph{transfer maps} $i_!$
\[
i^\ast\colon RU(G) \to RU(H),\qquad i_! = \bbC[G]\otimes_{\bbC[H]}(\bs)\colon RU(H)\to RU(G),
\]
and associated with any surjective homomorphism $p\colon G \to K$ one has \emph{inflation} $p^\ast$ and \emph{deflation} $p_!$ maps
\begin{equation}\label{eq:deflations}
p^\ast\colon RU(K) \to RU(G),\qquad p_! = \bbC[G]\otimes_{\bbC[H]}(\bs) \colon RU(G) \to RU(K).
\end{equation}
This algebraic structure on representation rings is reflected in the higher structure present on equivariant complex $K$-theory. For example, the fact that $G$-equivariant complex $K$-theory is represented by a \emph{genuine $G$-spectrum}
\[
KU_G \in \spectra_G
\]
encodes the presence of restrictions and transfers between subgroups of $G$. As the group $G$ varies, these $G$-spectra assemble into a \emph{global equivariant spectrum}
\[
\gKU \in \spectra^\gl
\]
in the sense of Schwede \cite{s}, defining a cohomology theory on \emph{global spaces} (or \emph{orbispaces}). This additionally encodes the presence of inflations. These refinements are available for all tempered theories: In \cite{gepner2024global2ringsgenuinerefinements}, Gepner, Pol, and the third-named author proved that the tempered cohomology theory associated with an oriented $\bfP$-divisible group is represented by a global $\bfE_\infty$ ring, i.e.\ by an object of $\CAlg(\spectra^\gl_\ab)$. For $K$-theory, this additional multiplicative structure refines the tensor product of representations and its compatibilities with restrictions, transfers, and inflations.

Notably missing from this picture are the deflations, or transfers along surjective group homomorphisms, which are present on representation rings. In fact, Lurie's \emph{tempered ambidexterity theorem} \cite[Th.1.1.21]{ec3} shows that tempered theories admit deflations, and much more: they admit transfers along all relatively $\pi$-finite maps of global spaces. These are integral analogues of the $K(n)$-local transfers arising from Hopkins--Lurie's $K(n)$-local ambidexterity theorem \cite{ambi}.

\subsection{Tempered cohomology theories as \texorpdfstring{$\pi$}{pi}-ambidextrous global spectra}\label{ssec:intro_sheaf}
To capture these additional transfers, we introduce a category $\Sp^\gl_\pi$ of \emph{$\pi$-ambidextrous global spectra}. This precisely encodes global cohomology theories equipped with transfers for all relatively $\pi$-finite morphisms of global spaces (see \Cref{ex:variousQ_intro} below). We then show that Lurie's tempered cohomology theory $\A_\bfG^{(\bs)}$, associated with an oriented $\bfP$-divisible group $\bfG$ over an $\bfE_\infty$ ring $\A$, admits a \emph{canonical} refinement to a $\pi$-ambidextrous $\bfE_\infty$ ring $\ul{\A}_\bfG \in \CAlg(\Sp^\gl_\pi)$. Moreover, this construction is natural in $\A$, allowing us to carry it out for families of oriented $\bfP$-divisible groups over arbitrary base stacks (including nonconnective spectral Deligne--Mumford stacks and more, see \cref{ssec:stacks}). This leads to the following, which is a special case of \Cref{main:functorialityinpdivisiblegroups} below.

\begin{theorem}\label{main:pisheaf}
Let $\bfG$ be an oriented $\bfP$-divisible group over a stack $\sfM$. Then $\bfG$ endows the structure sheaf $\calO_\sfM$ of $\bfE_\infty$ rings on $\sfM$ with a lift to a sheaf $\O_\bfG$ of $\pi$-ambidextrous global $\bfE_\infty$ rings:
\begin{center}\begin{tikzcd}
&\CAlg(\spectra^\gl_\pi)\ar[d]\\
\Aff_{/\sfM}^\op\ar[r,"\calO_\sfM"]\ar[ur,dashed,"\ul{\calO}_\bfG"]&{\CAlg.}
\end{tikzcd}\end{center}
In particular, the ring of global sections $\Gamma(\calO_\sfM)$ of $\sfM$ is the underlying spectrum of a $\pi$-ambidextrous global $\bfE_\infty$ ring $\Gamma(\ul{\calO}_\bfG)\in \CAlg(\spectra^\gl_\pi)$.
\end{theorem}

If $\sfM =\Spec \A$ is affine, then $\ul{\A}_\bfG = \Ga(\ul{\O}_{\bfG})$ is a refinement of Lurie's tempered cohomology theory $\A_\bfG^{(\bs)}$ to a $\pi$-ambidextrous global $\E_\infty$ ring. We note that this special case of the theorem above is essentially contained in the work of Ben-Moshe \cite{benmoshe_tempered}.

\begin{ex}\label{ex:KU_intro}
The torsion subgroup $\mu_{\bfP^\infty}\subset\bfG_m$ of the multiplicative group scheme canonically admits the structure of an oriented $\bfP$-divisible group over $\KU$ \cite[Th.6.5.1]{ec2}. By \cite[Th.4.1.2]{ec3}, the tempered cohomology theory associated with $\mu_{\bfP^\infty}$ is equivalent to equivariant complex $K$-theory. Thus we obtain a $\pi$-ambidextrous equivariant complex $K$-theory spectrum
\[
\gKU \coloneqq \Ga(\O_{\mu_{\bfP^\infty}/\Spec \KU}) \in \CAlg(\Sp^\gl_\pi)
\]
encoding all the structure discussed above, including deflations.

To illustrate this extra deflationary structure, let $G$ be a finite group and $X$ be a $G$-space. Given a normal subgroup $N\triangleleft G$ which acts properly discontinuously on $X$, the covering map $X \to X/N$ extends to a relatively $1$-finite map
\[
p\colon X//G \to (X/N)//(G/N)
\]
of global spaces. One thus obtains a transfer map
\[
p_!\colon KU_G^0(X) \to KU_{G/N}^0(X/N)
\]
in equivariant $K$-theory, sending the class of a $G$-equivariant vector bundle $\xi \to X$ to the class of the $G/N$-equivariant vector bundle $p_!(\xi)$ for which the fiber $(p_!\xi)_{Nx}$ is naturally equivalent to the space of coinvariants $\left(\bigoplus_{n\in N}\xi_{nx}\right)_{N}$ for $x\in X$.
\end{ex}

The extra flexibility of \cref{main:pisheaf} being applicable to oriented $\bfP$-divisible groups over arbitrary base stacks allows us to extend this example to equivariant \emph{real} $K$-theory.

\begin{ex}\label{ex:KO_intro}
As will be explained in greater detail in \cite[\textsection A]{geometricnorms}, the quotient stack $\Spec\KU / C_2$ for the action of complex conjugation on $\KU$ is equivalent to a \emph{moduli stack of oriented tori} $\M_\Tori^\ori$. The oriented $\bfP$-divisible group $\mu_{\bfP^\infty}$ over $\KU$ descends to $\M_\Tori^\ori$, and applying \cref{main:pisheaf} we thus obtain a $\pi$-ambidextrous equivariant real $K$-theory spectrum
\[\gKO \coloneqq \Ga(\O_{\mu_{\bfP^\infty}/\M_\Tori^\ori}) \in \CAlg(\spectra_\pi^\gl),\]
satisfying $\gKO\simeq\gKU^{\h C_2}$.
\end{ex}

We emphasize that, even for these classical examples, the structure we produce is new: to our knowledge, coherent deflations have not been previously constructed in equivariant $K$-theory, despite being a well-understood construction at the level of representation rings.

We obtain even more examples by considering $\bfP$-divisible groups associated with elliptic cohomology theories.

\begin{ex}\label{ex:TMF_intro}
The \emph{moduli stack of oriented elliptic curves} $\M_\Ell^\ori$ of Goerss--Hopkins--Miller, although our usage here crucially uses Lurie's interpretation in \cite[\textsection7]{ec2}, carries the universal oriented elliptic curve $\calE$. The torsion subgroup $\calE[\bfP^\infty]\subset\calE$ is an oriented $\bfP$-divisible group over $\M_\Ell^\ori$, and applying \cref{main:pisheaf} yields a $\pi$-ambidextrous global $\bfE_\infty$ ring of topological modular forms
\[\gTMF \coloneqq \Ga(\O_{\calE[\P^\infty]}) \in \CAlg(\spectra_\pi^\gl).\]
This is a $\pi$-ambidextrous global refinement of Hopkins' $\E_\infty$ ring $\TMF$ of topological modular forms, which  by \cite{elltempcomp} represents the same global cohomology theory (for finite abelian groups) as the equivariant elliptic cohomology theories constructed by Gepner--Meier \cite{davidandlennart} and  Gepner--Linskens--Pol \cite{gepner2024global2ringsgenuinerefinements}.
\end{ex}

These examples, along with others and their relationships, are further discussed in \Cref{ssec:piambi_tempered_cohomology}.

Although the $\pi$-ambidextrous structure of $\ul{\O}_\bfG$ given in \Cref{main:pisheaf} is new, one of the main focuses of this article is on the subtle \emph{functoriality} of $\ul{\O}_\bfG$ in general \emph{non-invertible morphisms} of oriented $\bfP$-divisible groups. To state this compatibility requires variants of the notion of a $\pi$-ambidextrous global spectrum. We will use this opportunity to detail some of our key definitions.

\subsection{Functoriality as \texorpdfstring{$\calQ$}{Q}-ambidextrous global spectra}\label{ssec:mainresults}

Write $\spaces$ for the category of spaces. 

\begin{mydef}
The \emph{global orbit category} for finite abelian groups is the full subcategory
\[
\Glo_\ab\subset\spaces
\]
of spaces spanned by the classifying spaces of finite abelian groups. Write
\[
\Spc^\gl_\ab \coloneqq \Psh(\Glo_\ab)
\]
for the category of \emph{$\ab$-global spaces}. Given a finite abelian group $H$, we write $\bfB H \in \Glo_\ab \subset \spaces^\gl_\ab$ for the corresponding representable $\ab$-global space.
\end{mydef}

Global spaces, also called \emph{orbispaces}, are a model for globally equivariant unstable homotopy theory; see the works of  Gepner--Henriques \cite{gepnerhenriques}, Lurie \cite[\textsection 3]{ec3}, and Schwede \cite{schwedeorbispaces}. A \emph{global cohomology theory} is a limit preserving functor
\[
(\abglobalspaces)^\op \to \spectra.
\]

Informally, one can think of a ``genuine global cohomology theory''  as a global cohomology theory equipped with \emph{transfers}. Classically, one considers transfers along faithful maps of groupoids, leading to the category $\Sp^{\gl}$ of global spectra as defined by Schwede \cite{s}. However, as we have seen, there are many interesting examples which admit substantially more transfers. To capture these examples, we will parametrize different varieties of genuine global cohomology theories by certain subcategories $\calQ\subset\abglobalspaces$ which we call \emph{inductible}; see \Cref{def:qambiglobalspectrum} for the precise definition. Given an inductible subcategory $\calQ\subset\abglobalspaces$, there is a span category $\Span_\calQ(\abglobalspaces)$ of global spaces, whose backwards maps are arbitrary but whose forward maps are required to be locally in $\calQ$.  

\begin{mydef}\label{maindef:ambidextrousglobalspectra}
A \emph{$\calQ$-ambidextrous spectrum} $E$ is a functor
\[
E\colon \Span_\calQ(\abglobalspaces)\to\spectra
\]
whose restriction to $(\abglobalspaces)^\op \subset \Span_\calQ(\abglobalspaces)$ is a global cohomology theory. Write
\[
\Sp^\gl_\calQ\subset\Fun(\Span_\calQ(\abglobalspaces)^\op,\spectra)
\]
for the category of $\calQ$-ambidextrous global spectra. This category is stable (\Cref{lm:qglobalspectra_stablepresentable}) and presentable (\cref{lem:Qcom_localization}), and supports a symmetric monoidal structure whose commutative monoids are exactly those $\calQ$-ambidextrous spectra $E\colon \Span_\calQ(\abglobalspaces)\to\spectra$ equipped with a lax symmetric monoidal structure (\cref{prop:qcommutativeglobalringsaslaxfunctors}).
\end{mydef}

\begin{ex}\label{ex:variousQ_intro}
As $\calQ$ varies, the category of $\calQ$-ambidextrous spectra captures many interesting variants of global spectra:
\begin{enumerate}
\item Suppose that $\calQ\subset\abglobalspaces$ consists of the isomorphisms. Then $\Sp^\gl_{\calQ}$ is simply the $\infty$-category of global cohomology theories.
\item Suppose that $\calQ = \Glo_\ab^f\subset\Glo_\ab$ is the subcategory spanned by the faithful maps. Then, as mentioned, $\Sp^\gl_{\calQ}$ is equivalent to Schwede's $\infty$-category of global spectra $\Sp^\gl$ with respect to the family of finite abelian groups by the main result of \cite{lenzmackey}.
\item Suppose that $\calQ = \Glo_\ab$. Then $\Sp^\gl_\calQ$ is a variant of $\Sp^\gl_\ab$ whose objects have the additional structure of deflations, transfers along arbitrary group homomorphisms. This $\infty$-category has not appeared in the literature, but has nonetheless been hinted at; see for example \cite[p.1331]{schwedeglobalalgebraicktheory}.
\item Suppose that $\calQ = \spaces_\pi\subset\spaces$ is the category of $\pi$-finite spaces viewed as a full subcategory of $\abglobalspaces$ by the cofree functor $\spaces\to\abglobalspaces$. Then $\Sp^\gl_\pi$ is the category of \emph{$\pi$-ambidextrous global spectra} which \cref{main:pisheaf} shows is the natural home for tempered cohomology theories. By definition, the objects of $\Sp^\gl_\pi$ come equipped with coherently compatible restrictions and transfer maps along all relatively $\pi$-finite maps of orbispaces $\sfX \to \sfY$.
\end{enumerate}
\end{ex}

The subtlety in encoding the functoriality of tempered cohomology theories in morphisms of $\bfP$-divisible groups arises from the fact that not all natural transformations between equivariant cohomology theories commute with the additional transfers that are present. Of particular note for us, the natural transformation $\A_{\bfG'}^{(\bs)} \to \A_{\bfG}^{(\bs)}$ of tempered cohomology theories associated with a homomorphism $\bfG \to \bfG'$ of oriented $\bfP$-divisible groups over $\A$ does \emph{not} generally refine to a map $\ul{\A}_{\bfG'} \to \ul{\A}_\bfG$ of $\pi$-ambidextrous spectra. However, in examples, it does commute with some large class of transfers. By allowing $\calQ$ to vary, we can coherently encode the sense in which these natural transformations generally only commute with a subset of all possible transfers.

To make this precise, for each set of prime numbers $S$, we define in \cref{ssec:pcovings} a wide subcategory $\awayfromS\subset\pi$ of $\pi$-finite spaces whose morphisms, in a strong sense, do not interact with $p$-primary information for $p\in S$. This leads to the category $\Sp^\gl_{\awayfromS}$ of $\pi$-ambidextrous global spectra away from $S$. Let $\pdivorstk_{\nmid S}\subset\pdivorstk$ be the wide subcategory of oriented $\bfP$-divisible groups whose morphisms are required to induce fiberwise isomorphisms on $p$-local components for all primes $p\notin S$; see \Cref{df:Cartesian_away_from_S}. Our first main result is the following refinement of \Cref{main:pisheaf}.

\begin{theoremalph}[\Cref{thm:temperedglobalspectrum}]\label{main:functorialityinpdivisiblegroups}
    For any set of prime numbers $S$, there is a functor
    \[(\pdivorstk_{S\nmid})^\op \to \CAlg(\Sp^\gl_{\awayfromS}), \qquad \bfG \mapsto \Ga(\ul{\O}_\bfG)\]
    refining the assignment of \Cref{main:pisheaf}.
\end{theoremalph}

\begin{ex}\label{intro_ex:adams_on_KU}
The torsion subgroup $\mu_{\bfP^\infty}\subset\bfG_m$ of the multiplicative group scheme satisfies
\[
\mu_{\bfP^\infty}(KU[1/S])_{\geq 1} \simeq \Sigma \bfQ/\bfZ[1/S]
\]
for any set $S$ of nonzero integers, and for an integer $\ell$ the Adams operation $\psi^{\ell}\colon \KU \to \KU[1/\ell]$ induces multiplication by $\ell$ on this; when $S = \emptyset$ this follows from \cite[Ex.6.5.7]{balderrama2023algebraic} after passing to torsion, and the same method of calculation applies in general.
In particular, $\psi^{\ell}$ extends to a morphism
\[
(\psi^{\ell},[\ell])\colon (\KU,\mu_{\bfP^\infty}) \to (\KU[1/\ell],\mu_{\bfP^\infty})
\]
of oriented $\bfP$-divisible groups, i.e.\ a morphism in $\pdivorstk$, where $[\ell]$ is the $\ell$-fold multiplication map of $\mu_{\bfP^\infty}$. Alternately, one may construct $\psi^\ell$ as a morphism in $\pdivorstk$ by viewing $\KU$ as an orientation classifier for the formal multiplicative group \cite[\textsection 6.5]{ec2}, and twisting the tautological orientation of $\KU[1/\ell]$ by $[\ell]$; see the discussion in \Cref{ex:AO_on_KO,ex:AO_on_TMF} and \cite[\textsection 6]{globaltate}. This morphism lies in $\pdivorstk_{S\nmid}$ for any $S$ containing all prime divisors of $\ell$, and therefore induces a map
\[
\psi^{\ell}\colon \gKU_{\awayfromS} \to \gKU[1/\ell]_{\awayfromS}
\]
of $\pi$-ambidextrous $\bfE_\infty$ rings away from $S$. This is \emph{not} a map of $\pi$-ambidextrous spectra unless $\ell = \pm 1$ (see \cref{prop:pconverse}). As a $\pi$-ambidextrous spectrum away from $S$ admits an underlying $G$-spectrum provided the primes in $S$ do not divide the order of $G$ (see \cref{restricted_G-spectrum}), this recovers a classical theorem of Hirata--Kono \cite[Th.3.1]{hiratakono} that the Adams operation $\psi^\ell\colon \KU \to \KU[1/\ell]$ lifts to a morphism $\KU_G\to\KU_G[1/\ell]$ of $G$-spectra provided $\ell$ is coprime to the order of $G$.
\end{ex}

Similar examples involving Adams operations on $\KO$ and $\TMF$ are discussed in \Cref{ex:AO_on_KO,ex:AO_on_TMF}, respectively.

\subsection{Parametrized category theory and tempered local systems}

The proof of \Cref{main:functorialityinpdivisiblegroups} takes up the bulk of the article. We proceed by establishing a general decategorification construction in the setting of \emph{parametrized category theory}, which we then apply to the theory of \emph{tempered local systems} over oriented $\bfP$-divisible groups introduced by Lurie. Both of these steps are nontrivial, and constitute much of the core technical work of this paper, taking place in \cref{sec:parametrized} and \cref{sec:temperedlocalsystems} respectively. This work is largely carried out for general small categories $T$ in place of $\Glo_\ab$, but we restrict ourselves to this special situation for the introduction; see \textsection\ref{subsec:globalsemi}-\ref{ssec:monoidalparametrized} for the general and more precise definitions.

\begin{mydef}[{Informal}]
An \emph{$\ab$-global category} is a limit preserving functor
\[
\ccat\colon (\abglobalspaces)^\op\to\Cat.
\]
Given an inductible subcategory $\calQ\subset\abglobalspaces$, an $\ab$-global category $\ccat$ is said to be \emph{$\calQ$-stable} if the following conditions are satisfied:
\begin{enumerate}
\item $\ccat(\sfX)$ is stable for all $\sfX\in\abglobalspaces$.
\item For every map $f\colon \sfX\to\sfY$ which is locally in $\calQ$, the restriction $f^\ast\colon \ccat(\sfY)\to\ccat(\sfX)$ sits in a string of adjunctions $f_! \dashv f^\ast \dashv f_\ast$ satisfying a Beck--Chevalley condition.
\item For every truncated map $f\colon \sfX\to\sfY$ which is locally in $\calQ$, a certain norm map
\[
\mathrm{Nm}_f\colon f_! \to f^\ast
\]
is an equivalence.
\end{enumerate}
These assemble into a category $\Cat(\Glo_\ab)_{\calQ\textup{-st}}$ of $\calQ$-stable categories and \emph{$\calQ$-exact functors}: natural transformations $\ccat\to\dcat$ whose components are exact functors of stable categories and which satisfy a base change condition with respect to the adjunctions of (2). Moreover, $\Cat(\Glo_\ab)_{\calQ\textup{-st}}$ admits a natural symmetric monoidal structure leading to a category $\Cat(\Glo_\ab)_{\calQ\textup{-st}}^\otimes$ of \emph{symmetric monoidal $\calQ$-stable categories}.
\end{mydef}

Given an ordinary symmetric monoidal stable category $\ccat$, the endomorphism spectrum $\End_\ccat(\bbone)$ of its monoidal unit carries a canonical $\bfE_\infty$ structure by \cite[\textsection4.8]{ha}. Our work in \cref{sec:parametrized} culminates in the following parametrized refinement of this construction.

\begin{theoremalph}[{\cref{thm:generaldecategorification}}]\label{introthm:globalsections}
The assignment sending a symmetric monoidal stable category $\ccat$ the $\bfE_\infty$ ring $\End_{\ccat}(\bbone)$ of endomorphisms of its unit refines to a limit preserving functor
\[
\End_{(\bs)}(\bbone)\colon \Cat(\Glo_\ab)^\otimes_{\calQ\textup{-st}} \to \CAlg(\Sp^\gl_\calQ)
\]
satisfying $\End_\ccat(\bbone)^\sfX = \End_{\ccat(\sfX)}(\bbone)$ for $\sfX\in\abglobalspaces$.
\end{theoremalph}

This theorem extends the parametrized decategorification process from \cite[\textsection3]{benmoshe_tempered}. Its proof is significantly more involved than the non-parametrized counterpart, and goes through several intermediate constructions that may be of more general interest in parametrized category theory. As these constructions are nevertheless technical, we refer the reader to \Cref{sec:parametrized} for more information.

To apply \cref{introthm:globalsections}, we need examples of $\calQ$-stable $\ab$-global categories. Such examples arise from the theory of \emph{tempered local systems} introduced and studied by Lurie in \cite{ec3}. Given an oriented $\bfP$-divisible group $\bfG$ and an $\ab$-global space $\sfX$, write $\LocSys_\bfG(\sfX)$ for the category of $\bfG$-tempered local systems on $\sfX$; we review the relevant definitions in \cref{ssec:recollectionofpdivs}. As we prove in \cref{ssec:functorialityofTLS}, this construction is natural in both $\bfG$ and $\sfX$, and extends to a functor
\[
\LocSys_{(\bs)}\colon (\pdivorstk)^\op \to \Cat(\Glo_\ab).
\]
When restricted to the subcategory of $\pdivorstk$ spanned by oriented $\bfP$-divisible groups over an affine base and fiberwise isomorphisms, this was established by Ben-Moshe in \cite{benmoshe_tempered}; we extend his construction to allow arbitrary morphisms of oriented $\bfP$-divisible groups. The main theorem of \cref{sec:temperedlocalsystems} is the further refinement to symmetric monoidal $\awayfromS$-stable categories.

\begin{theoremalph}\label{introthm:functoriality}
For any set $S$ of prime numbers, the assignment to an oriented $\bfP$-divisible group $\bfG$ its $\ab$-global category of $\bfG$-tempered local systems lifts to a functor
\[
\LocSys_{(\bs)}\colon (\pdivorstk_{S\nmid})^\op \to \Cat(\Glo_\ab)^\otimes_{\awayfromS\textup{-st}}.
\]
\end{theoremalph}

\cref{main:functorialityinpdivisiblegroups} is then simply the composite of \cref{introthm:globalsections} and \cref{introthm:functoriality}.

Taking $S = \emptyset$, \cref{introthm:functoriality} asserts that for a fixed $\bfP$-divisible group $\bfG$, the $\ab$-global category $\LocSys_\bfG$ admits in a canonical way the structure of a symmetric monoidal $\pi$-stable category, i.e.\ one has
\[
\LocSys_\bfG \in \Cat(\Glo_\ab)^\otimes_{\pi-\st}.
\]
This is a reformulation of Lurie's \emph{tempered ambidexterity theorem} \cite[Th.7.2.10]{ec3} (see \cref{thm:temperedambidexterity}). The extra work needed to prove \cref{introthm:functoriality} stems from the observation that this lift is \emph{not} natural in the $\bfP$-divisible group $\bfG$: in general, a morphism $\alpha\colon \bfG'\to\bfG$ of $\bfP$-divisible groups will induce a functor $\LocSys_{\bfG}\to\LocSys_{\bfG'}$ which is \emph{not} $\pi$-exact, but, as we prove, is $\awayfromS$-exact for $S$ the set of primes at which $\alpha$ is not a fiberwise isomorphism.

\subsection{Tempered homology and geometric fixed points}

Having constructed $\pi$-ambidextrous spectra from oriented $\bfP$-divisible groups, we then turn our attention to studying the equivariant stable homotopy type of their associated tempered theories. Our first contribution is to construct a sensible and general form of tempered \emph{homology} associated with a $\bfP$-divisible group.

Let $\bfG$ be an oriented $\bfP$-divisible group over a stack $\sfM$. If $F$ is a $\pi$-finite space, then the sheaf in \cref{main:pisheaf}, after evaluation on $F$, defines a \emph{quasi-coherent} sheaf of $\bfE_\infty$ rings
\[
\calO_\bfG^F \in \CAlg(\QCoh(\sfM))
\]
on $\sfM$ (\Cref{pr:quasicoherent_on_picompactspaces}). Write
\[
\bfG(F) \in \Stk_{/\sfM}^\aff
\]
for its relative spectrum, a stack which is affine over $\sfM$. 

\begin{theoremalph}[{\Cref{constr:temperedhomology}}]\label{thm:introtemperedhomology}
There is a lax symmetric monoidal and colimit preserving \emph{$F$-global tempered homology functor}
\[
\O_\bfG^F[\bs]\colon \abglobalspaces{}_{/F}\to\QCoh(\bfG(F)),
\]
sending a map of $\pi$-finite spaces $p\colon K \to K'$ over $F$ to the $\calO_\bfG^F$-linear transfer $p_!\colon \calO_\bfG^K \to \calO_\bfG^{K'}$. 
\end{theoremalph}

Taking $F = \bfB H$ for a finite group $H$, the category $\spaces_H$ of $H$-spaces sits as a subcategory of the category $\abglobalspaces{}_{/\bfB H}$ of $H$-global spaces which is closed under colimits and products. Thus \cref{thm:introtemperedhomology} specialises to provide a $\bfG$-tempered $H$-equivariant homology functor
\[
\calO_\bfG^{\bfB H}[\bs//H]\colon \spaces_H \to \QCoh(\bfG(\bfB H)).
\]
As suggested by Lurie in the context of elliptic cohomology \cite[Rmk.3.14]{lurieecsurvey}, tempered homology is in many ways better behaved than tempered cohomology, especially when applied to non-compact $H$-spaces. In particular, it satisfies the following general base change property.

\begin{theoremalph}[{\Cref{thm:homologybasechange}}]\label{intro:basechangeforhomology}
Let $\bfG$ be an oriented $\bfP$-divisible group over a stack $\sfM$. Then, for any map $f\colon \sfN\to\sfM$ of stacks, the induced map $\tilde{f}\colon (f^\ast\bfG)(F) \to \bfG(F)$ satisfies
\[
\tilde{f}^\ast \left(\calO_\bfG^F[\sfX]\right)\simeq \calO_{f^\ast\bfG}^F[\sfX]
\]
for any $\pi$-finite space $F$ and global space $\sfX \in \abglobalspaces{}_{/F}$.
\end{theoremalph}

After passing to reduced homology, our construction of $\bfB H$-global homology allows us to define the \emph{$H$-geometric fixed points} for a finite group $H$
\[
\calO_\bfG^{\Phi H} = \widetilde{\calO}_\bfG^{\bfB H}[\widetilde{E}\calP_H//H] \in \CAlg(\QCoh(\bfG(\bfB H)))
\]
of $\bfG$, with associated stack
\[
\Phi^H \bfG \in \Stk^\aff_{/\bfG(\bfB H)}.
\]
The base change theorem (\Cref{intro:basechangeforhomology}) for $F$-global tempered homology then specializes to base change for geometric fixed points with immediate consequences. For example, we show that for any $\bfP$-divisible group $\bfG$ and any finite \emph{nonabelian} group $H$, the $H$-geometric fixed points stack
\begin{equation}\label{eq:intro_vanishing}\Phi^H\bfG = \emptyset\end{equation}
vanishes whenever $H$ is nonabelian; see \Cref{pr:nonabelian_vanishing}, and also \Cref{cor:tomdieckesque} for an application.

We round out our discussion of geometric fixed points by giving a modular interpretation of $\Phi^H\bfG$ when $H$ is abelian. By construction, if $H$ is an abelian group, then we may regard
\[
\bfG(\bfB H) \simeq \bfG[\dual{H}] \simeq \Hom(\dual{H},\bfG)
\]
as a stack of homomorphisms $\dual{H} \to \bfG$. For $H = \bfZ/(n)$, this is equivalent to the $n$-torsion subgroup of $\bfG$. Each inclusion of subgroups $L \to H$ induces a closed immersion
\[\Hom(\dual{L},\bfG) \to \Hom(\dual{H}, \bfG),\]
and we define the \emph{stack of injections}
\[\Inj(\dual{H}, \bfG) \coloneq \bfG[\dual{H}]\setminus\bigcup_{L\subsetneq H}\bfG[\dual{L}] \]
as the open substack of $\Hom(\dual{H}, \bfG)$ away from these closed immersions for all proper subgroups $L \subsetneq H$; see \Cref{def:injstack}. Our final result is the following.

\begin{theoremalph}\label{main:geometricmodel}
Let $H$ be a finite abelian group and $\bfG$ be an oriented $\bfP$-divisible group. Then there is a (necessarily unique) equivalence
\[
\Phi^H\bfG\simeq \Inj(\dual{H},\bfG)
\]
of open substacks of $\bfG(\bfB H)\simeq\bfG[\dual{H}]$.
\end{theoremalph}

This can be regarded as an integral and geometric refinement of the computation of the geometric fixed points of the Borel theories associated with Morava $E$-theory \cite[\textsection 3]{bhnnns}. It implies a corresponding \emph{blueshift} phenomenon for the geometric fixed points of tempered theories.

\begin{coralph}[{\Cref{cor:blueshift}}]\label{maintheorem:blueshift}
    Let $\bfG$ be an oriented $\bfP$-divisible group. Fix a prime $p$, and suppose that the $p$-divisible group $\bfG_{(p)}$ is of height $\leq h$. Let $H$ be a finite abelian group, and set $n = \log_p |H/pH|$. Then the stack $\Phi^H\bfG$ has chromatic height at most $h-n$ at the prime $p$. In particular, $\Phi^H\bfG = \emptyset$ if $n > h$.
\end{coralph}

Most of the results above are stated and proved in the geometric setting of stacks and quasi-coherent sheaves, as this is the setting where base change \cref{intro:basechangeforhomology} holds in full generality. One may apply the lax symmetric monoidal global sections functor $\Ga\colon \QCoh(\sfM) \to \Mod_{\Ga(\calO_\sfM)}$ to obtain statements about $\E_\infty$ rings and spectra, only some care is needed as in general $\Ga$ need not preserve colimits; for example, $\Gamma(\Phi^H\bfG)\not\simeq \Gamma(\ul{\calO}_\bfG)^{\Phi H}$ in general.

In practice, the work of Mathew--Meier \cite{akhilandlennart} and our previous work \cite[Ths.A \& B]{reconstruction} shows that the stacks which are of interest in chromatic homotopy theory are frequently \emph{$0$-semiaffine}, meaning that $\Ga\colon \QCoh(\sfM) \to \Mod_{\Ga(\calO_\sfM)}$ does in fact preserve colimits. This includes all of the examples given above in \Cref{ex:KU_intro,ex:KO_intro,ex:TMF_intro} and below in \Cref{ssec:piambi_tempered_cohomology}. Under this assumption, all of the preceding theorems behave well with respect to global sections. In particular, if $\bfG$ is an oriented $\bfP$-divisible group over a $0$-semiaffine stack $\sfM$, then the global sections of $\Phi^H\bfG$ recover the $H$-geometric fixed points of the $\pi$-ambidextrous global $\bfE_\infty$ ring $\Gamma(\ul{\calO}_\bfG)$:
\[
\Gamma(\calO_{\Phi^H\bfG})\simeq \Gamma(\ul{\calO}_\bfG)^{\Phi H};
\]
see \Cref{rmk:0semiaffine}.

\begin{ex}
The main theorem of \cite{akhilandlennart} states that the stack $\sfM_\Ell^\ori$ is $0$-affine. By the vanishing statements of (\ref{eq:intro_vanishing}) and \Cref{maintheorem:blueshift}, we then see that equivariant $\TMF$ satisfies
\[
\gTMF^{\Phi H}\simeq \begin{cases}\Gamma(\calO_{\Inj(\dual{H},\calE)}),&H\text{ abelian and requiring at most $2$ generators},\\0&\text{ otherwise};\end{cases}
\]
see \Cref{ex:geoFP_TMF}. After inverting the order of $H$, the stack $\Inj(\dual{H},\calE)$ becomes \'etale over $\sfM_\Ell^\ori$, and may be identified with the classical moduli stack of \emph{level $\dual{H}$-structures} on $\calE$:
\[\gTMF^{\Phi C_n\times C_n} \simeq \TMF(n), \qquad \gTMF^{\Phi C_n}[1/n] \simeq \TMF_1(n).\]
This implies that $\gTMF^{\Phi H} \neq 0$ if and only if $H$ is a finite abelian group requiring at most $2$ generators, thus computing the \emph{derived defect base} of $\gTMF$ à la \cite{derivedindandres}. Considering the  stack of injections more closely, we see that $\gTMF^{\Phi C_n}$ is an integral refinement of $\TMF_1(n)$, in the sense that it has chromatic height exactly $1$ at all primes $p$ dividing $n$; compare this to $\TMF_1(n)$ which has height $0$ at such primes. This corresponds to the fact that if an elliptic curve $E$ admits a subgroup isomorphic to the constant group scheme $\ul{C}_p$, then its formal height is at most $1$, and that ordinary elliptic curves over complete local rings provide examples with formal height exactly $1$.
\end{ex}

These results provide sound foundations for the further study of tempered cohomology theories as equivariant spectra. The base change results for the geometric fixed points of tempered cohomology theories have already been used by the second- and third-named authors \cite[Th.7.14]{globaltate} to study global Tate $K$-theory. In forthcoming work \cite{geometricnorms}, we rely on the moduli interpretation of geometric fixed points \Cref{main:geometricmodel} to construct \emph{geometric norms} on $\gTMF$ and an ultra-commutative structure on $\gTMF \otimes\Q$.

\subsection*{Outline}
This article is comprised of four sections; more detailed introductions can be found at the beginning of each section.

\begin{itemize}
    \item In \Cref{sec:globalhomotopytheory}, we introduce the category of \emph{$\calQ$-ambidextrous monoids} in a general category $\calC$, denoted by $\Mon^\gl_\calQ(\calC)$, which specializes to $\Sp^\gl_\calQ$ from \Cref{maindef:ambidextrousglobalspectra} for $\calC = \Sp$ and $\calQ\subset\abglobalspaces$, and establish its basic categorical properties.
    \item In \Cref{sec:parametrized}, we recall and prove the necessary concepts and statements from parametrized category theory, including $\calQ$-stability. This culminates with our proof of the general decategorification result \Cref{introthm:globalsections}.
    \item In \Cref{sec:temperedlocalsystems}, we show that Lurie's categories of tempered local systems $\LocSys_\bfG$ define $\pi$-stable $\Glo_\ab$-categories and discuss which morphisms of oriented $\bfP$-divisible groups induce morphisms of $\calQ$-stable $\Glo_\ab$-categories for various $\calQ$. In particular, this contains a proof of \Cref{introthm:functoriality}.
    \item In \Cref{sec:thetemperedglobalspectrum}, we combine the previous two sections to show that Lurie's tempered cohomology theories are represented by $\pi$-ambidextrous global spectra and discuss the equivariant stable homotopy theory of these examples.
\end{itemize}

The reader willing to take \Cref{main:functorialityinpdivisiblegroups}, and hence also \Cref{introthm:globalsections,introthm:functoriality}, for granted, is welcome to read only \Cref{sec:globalhomotopytheory} for the definition of $\Sp_\calQ^\gl$ and \Cref{sec:thetemperedglobalspectrum} for the applications to the equivariant stable homotopy theory of tempered cohomology theories.

\subsection*{Acknowledgements}
SL would like to thank Bastiaan Cnossen for helpful discussions concerning \Cref{sec:parametrized}, as well as him and Tobias Lenz for conversations during the preparation of \cite{CLLSpans}. SL also thanks Shay Ben-Moshe for illuminating discussions regarding decategorification and transfers in tempered theories. The overlap of the work here with \cite{benmoshe_tempered} is a consequence of the generosity with which he shared his ideas.

JMD was supported by the DFG-funded research 
training group GRK 2240: Algebro-Geometric Methods in Algebra, 
Arithmetic and Topology. SL was an associate member of the SFB: Higher invariants. Both JMD and SL would like to thank the Isaac Newton Institute for Mathematical Sciences, Cambridge, for support and hospitality during the programme \emph{Equivariant homotopy theory in context} where work on this paper was undertaken. This work was supported by EPSRC grant no EP/K032208/1.

\section{Ambidextrous global spectra}\label{sec:globalhomotopytheory}
In this section, we review \Cref{maindef:ambidextrousglobalspectra} and the fundamental properties of $\Sp_\calQ^\gl$. In fact, we will proceed slightly more generally, by replacing $\Glo_{\ab}$ by an arbitrary small category $T$ in the definition. First, we introduce the notion of an \emph{inductible} class $\calQ$ of maps between presheaves on $T$, and then define \emph{$\calQ$-ambidextrous monoids} in a general category $\calC$ with limits in \Cref{ssec:qcommtuativemonoids}, denoted by $\Mon_\calQ(\calC)$. This specializes to $\Sp^\gl_\calQ$ for $T = \Glo_{\ab}$ and $\calC = \Sp$. Then we study the basic categorical properties of this construction. In \Cref{ssec:symmetricmonoidalstructures}, we prove that $\Mon_\calQ(\calC)$ inherits a canonical symmetric monoidal structure from $\calC$, and give two equivalent characterizations of this structure. Finally, in \Cref{ssec:homologytheories} we construct homology theories out of $\calQ$-ambidextrous monoids, generalizing the classical construction of equivariant homology theories.

For this section we fix a small category $T$, and denote by $\PshT$ the category $\Psh(T) = \Fun(T^{\op},\Spc)$ of presheaves on $T$. We will identify $T$ with a subcategory of $\PshT$ via the Yoneda embedding.

\subsection{\texorpdfstring{$\calQ$}{Q}-commutative monoids}\label{ssec:qcommtuativemonoids}

\begin{mydef}\label{df:preinductable}
A small subcategory $\calQ\subset\PshT$ is \emph{inductible} if:
	\begin{enumerate}
    \item $\calQ$ contains all the objects of $T$.
	\item A map $f\colon X\to Y$ with $Y\in \calQ$ is in $\calQ$ if and only if $f\times_Y A$ is for all maps $A \to Y$ for $A\in T$.
	\item $\calQ$ is closed under diagonals: if $X \to Y$ is in $\calQ$ then $X \to X\times_Y X$ is in $\calQ$.
	\item Every morphism in $\calQ$ with target in $T$ is truncated.
\end{enumerate}
\end{mydef}

\begin{mydef}\label{def:local_class_generated}
Given an inductible subcategory $\calQ\subset\PshT$, the \emph{local class}
\[
\bar{\calQ}\subset \PshT
\]
generated by $\calQ$ is the wide subcategory of $\PshT$ spanned by those morphisms $X \to Y$ which are locally in $\calQ$, in the sense that $X\times_Y A \to A$ is in $\calQ$ for all $A \to Y$ with $A\in T$.
\end{mydef}

The local class $\bar{\Qloc}$ is closed under base change by construction, allowing us to make the following definition.

\begin{mydef}
Given an inductible subcategory $\Qloc$,  we write $\Span_{{\Qloc}}(\PshT)$ for the category of spans in $\PshT$ with arbitrary backwards map and with forwards map in $\bar{\Qloc}$.
\end{mydef}

Thus the objects of $\Span_{{\Qloc}}(\PshT)$ are the objects of $\PshT$, but with mapping spaces
\[
\Map_{\Span_{{\Qloc}}(\PshT)}(X,Y) = \Biggl\{\begin{tikzcd}[cramped, sep=small]
  & \arrow{dl} W \arrow[dr, "\in {\bar{\Qloc}}"] & \\ 
X &                                      & Y
\end{tikzcd}\Biggr\};
\]
composition is given by pullback, i.e.\
\[
\left(X \leftarrow W \rightarrow Y \right)\circ\left(Y\leftarrow V \rightarrow Z\right) = \left(X \leftarrow W\times_Y V \rightarrow Z\right).
\]
See \cite{HHLNa} for a formal definition.

\begin{mydef}\label{def:qambiglobalspectrum}
	Let $\Cc$ be a category with small limits. A \emph{$\calQ$-ambidextrous monoid in $\Cc$} is a functor 
	\[
	E\colon \Span_{{\Qloc}}(\PshT)\to \Cc
	\]
	whose restriction along $\PshT^{\op}\hookrightarrow \Span_{{\Qloc}}(\PshT)$ preserves limits. Such objects assemble into a category
	\[
	\Mon_\Qloc(\Cc) \subset \Fun(\Span_{{\Qloc}}(\PshT),\Cc).
	\]
\end{mydef}

\begin{mydef}
A $\calQ$-ambidextrous monoid in $\Sp$ is called a \emph{$\calQ$-ambidextrous spectrum}, and we write
\[\Sp_\calQ \coloneqq \Mon_\calQ(\Sp) \]
for the category of $\calQ$-ambidextrous spectra. More generally, we will verify in \cref{lm:qglobalspectra_stablepresentable} that $\Mon_\calQ(\calC)$ is stable whenever $\calC$ is. Therefore, in this case, we will denote $\Mon_\calQ(\calC)$ by $\Sp_{\calQ}(\calC)$ and refer to an object therein as a $\calQ$-ambidextrous spectrum (object) in $\calC$.
\end{mydef}

\begin{remark}
Note that the category $T$ is implicit in the notations above by the choice $\calQ\subset \Psh(T)$. When $T=\Glo_{\ab}$, we will refer to objects in $\Sp^{\gl}_{\calQ} \coloneqq \Sp_{\calQ}$ as $\calQ$-\emph{ambidextrous global spectra}, as in the introduction.
\end{remark}

\begin{rmk}
We will verify in \cref{lem:Qcom_localization} below that $\Mon_\Qloc(\ccat)$ is presentable, resp., stable, whenever $\ccat$ is presentable, resp., stable. It then formally follows that
\[
\Mon_\calQ(\ccat)\simeq \Mon_\calQ(\spaces)\otimes\ccat,
\]
for any presentable category $\ccat$, and
\[
\Sp_\calQ(\ccat)\simeq {\Sp_\calQ}\otimes {\ccat}
\]
for any stable presentable category $\ccat$.
\end{rmk}

Other than the globally equivariant example that is the main focus of this article, the language of $\calQ$-ambidextrous spectra also makes sense in traditional $G$-equivariant stable homotopy theory. Denote by $\Fin_G$ the category of finite $G$-sets for a finite group $G$. 

\begin{example}\label{ex:GequivSHT}
Let $G$ be a finite group. Potentially the first example of $T$ and $\calQ$ for which the definition above was considered is $(T, \calQ) = (\mathrm{Orb}_G, \Fin_G)$, the category of transitive $G$-sets and finite $G$-sets respectively. We will follow the literature in referring to objects of $\Mon_{\Fin_G}(\calC)\simeq \Fun^\times(\Span(\Fin_G),\ccat)$ as \emph{$G$-commutative monoids in $\calC$} and denote them $\Mon_G(\calC)$.

Importantly, the Guillou--May theorem identifies $\Sp_{\Fin_G}$ with the category $\Sp_G$ of genuine $G$-spectra, see for example \cite[Cor.9.16]{CLLEquiv}. In particular, we note that when $\calC$ is stable, $\Mon_G(\calC)$ is equivalent to the tensor product $\calC \otimes {\Sp_G}$ of $\calC$ with (genuine) $G$-spectra. Therefore, in this case, we may refer to $\Mon_G(\calC)$ as the category of $G$-spectra in $\calC$. Similarly, we will refer to objects in $\CAlg(\Mon_G(\calC))$ as $G$-equivariant $\E_\infty$-rings in $\calC$.
\end{example}

\begin{rmk}\label{rmk:restrictionbetweenCMonQQprime}
If $\calQ\subset\calQ'$, then the restriction along the induced functor $\Span_\calQ(\PshT)\to\Span_{\calQ'}(\PshT)$ defines a conservative and limit preserving restriction functor
\[
\res\colon \Mon_{\calQ'}(\Cc) \to \Mon_\calQ(\Cc). 
\]
\end{rmk}

Let us unpack the structure encoded by a $\calQ$-ambidextrous monoid.

\begin{notation}
	Let $E$ be a $\calQ$-ambidextrous monoid in $\Cc$. Associated with every map $f\colon X\to Y$ in $\PshT$ is a \emph{restriction} map
	\[
	\res_f\coloneqq E(Y\leftarrow X = X)\colon E(Y) \to E(X),
	\]
	and associated with every map $q\colon Y\to Z$ in $\Qloc$ is a \emph{transfer} map
	\[
	\tr_q\coloneqq E(Y=Y\to Z) \colon E(Y)\to E(Z).
	\]
	The composition law in $\Span_{\calQ}(\PshT)$ tells us immediately that 
	\[
	\res_f \res_g \simeq \res_{gf}\quad\textup{and}\quad \tr_q \tr_p\simeq \tr_{qp}.
	\]
    Compositions of transfers and restrictions are controlled by the structure of pullbacks along maps in $\calQ$: given a map $f\colon X \to Z$ in $\PshT$ and a map $q\colon Y \to Z$ in $\calQ$, we may form the pullback 
	\[\begin{tikzcd}
		{X\times_Y Z} & X \\
		Y & {Z.}
		\arrow["{q'}", from=1-1, to=1-2]
		\arrow["{f'}"', from=1-1, to=2-1]
		\arrow["\lrcorner"{anchor=center, pos=0.125}, draw=none, from=1-1, to=2-2]
		\arrow["f", from=1-2, to=2-2]
		\arrow["q", from=2-1, to=2-2]
	\end{tikzcd}\]
    Contained in the structure of $E$ is a homotopy witnessing $\res_f \tr_q \simeq \tr_{q'} \res_{f'}$.
    \end{notation}
    
\begin{example}
    In the case that $T = \Glo_{\ab}$  and $f\colon \bfB H\to \bfB G, q\colon \bfB K\to \bfB G$ are faithful maps of orbits, we may compute that 
	\[
	\bfB H \times_{\bfB G} \bfB K \simeq \coprod_{[g] \in H\backslash G/ K} \bfB(H^g\cap K).
	\]
    So the relationship between restriction and transfer above specializes in this case to the usual double coset formula relating restriction and transfers along injective group homomorphisms familiar from equivariant homotopy theory and representation theory:
	\[
	\res_H^G \tr_K^G \simeq \sum_{[g] \in H \backslash G / K} \tr^H_{H\cap {}^{g}K} \circ c_g^* \circ \res^K_{H^{g}\cap K}
	\]
    for any two subgroups $H,K\subset G$.
\end{example}

Now we detail examples of inductible subcategories of $\abglobalspaces$ that play a role in this article.

\begin{example}
Let $\calG$ denote the category of finite groupoids, and denote by $\calO\subset\calG$ the wide subcategory spanned by the faithful morphisms of groupoids with abelian isotropy. The restricted Yoneda embedding along $\Glo_{\ab}\subset \calG$ realizes $\calG$ as a full subcategory of ab-global spaces. An easy argument shows that in this way $\calO$ and $\calG$ both define inductible subcategories of $\abglobalspaces$. By \cite[Th.A]{lenzmackey}, the category $\Sp_\calO^\gl$ of $\calO$-ambidextrous $\ab$-global spectra is equivalent to (the underlying $\infty$-category of) Schwede's category of global spectra $\Sp^\gl_\ab$ for the family of finite abelian groups. 
\end{example}

\begin{example}
From the inductible subcategory $\calG$, we obtain a category $\Sp^\gl_\calG$ of $\calG$-ambidextrous global spectra. By construction, $\calG$-ambidextrous global spectra admit \emph{deflations}, i.e.\ transfers along surjective group homomorphisms, and there is a forgetful functor $\Sp^\gl_\calG \to \Sp^\gl_\ab$ which forgets these deflations.
\end{example}

The next examples require the following definition.

\begin{mydef}\label{df:rightadjoint}
The evaluation functor 
\[
\ev_{\ast}\colon \Spc^{\gl}_{\ab} \to \Spc
\]
admits a fully faithful right adjoint denoted by $\bfR$. As $\Glo_{\ab}$ is a full subcategory of $\Spc$, $\bfR$ is computed by the restricted Yoneda embedding $\bfR(X)(\bfB A) \simeq \Map_{\Spc}(BA,X)$. 
\end{mydef}

\begin{example}
Let $\spaces_\pi\subset\spaces$ denote the full subcategory of $\pi$-finite spaces: spaces with finitely many path components, each of which has finitely many non-zero homotopy groups, all of which are finite. The essential image $\pi = \bfR(\spaces_\pi)\subset\abglobalspaces$ is inductible, and so generates a local class $\ol{\pi}\subset\abglobalspaces$. We call the maps in $\ol{\pi}$ \emph{relatively $\pi$-finite} maps of $\ab$-global spaces, following \cite[Def.7.2.4]{ec3}. As justified by \Cref{main:pisheaf}, the category $\Sp^{\gl}_{\pi}$ of $\pi$-amibdextrous ab-global spectra encodes the maximal structure on the tempered cohomology theories associated with oriented $\bfP$-divisible groups which can be obtained from Lurie's tempered ambidexterity theorem and our decategorification results.
\end{example}

\begin{example}
Let $S$ be a set of primes. In \cref{ssec:pcovings}, we will identify an inductible subcategory $\awayfromS\subset\pi\subset\abglobalspaces$ consisting of relatively $\pi$-finite maps of $\ab$-global spaces which, in a strong sense, do not interact with $p$-primary information for $p\in S$. The resulting category $\Sp^\gl_{\awayfromS}$ encodes ab-global cohomology theories with transfers only along relatively $\pi$-finite maps that do not interact with $S$. This class will play a key role in our discussion of the naturality of the $\pi$-ambidextrous $\ab$-global spectrum associated with an oriented $\bfP$-divisible group.
\end{example}

Finally, we mention the following trivial example.

\begin{example}
The collection of isomorphisms $\simeq \coloneq T^{\simeq}$ always forms an inductible subcategory of $\PshT$. The resulting category $\Sp^T_\simeq$ of $\simeq$-ambidextrous spectra is equivalent to the category of limit preserving functors $\PshT^\op\to\Sp$, i.e.\ to the category of (naïve) cohomology theories on $\PshT$. When $T= \Glo_{\ab}$, these are the (naive) ab-global cohomology theories of the introduction.
\end{example}

We now come to some basic categorical facts about $\Mon_\calQ(\cC)$.

\begin{lemma}\label{lem:Qcom_localization}
Let $\calQ$ be an inductible subcategory of $\PshT$, and denote by $\calD$ the full subcategory of $\PshT$ spanned by objects in $\calQ$. Then, for every $\Cc\in \PrL$, the restriction functor
\[
\Mon_\Qloc(\Cc) \subset \Fun(\Span_{{\calQ}}(\PshT),\Cc) \to \Fun(\Span_{\calQ}(\calD), \calC)
\]
is fully faithful and witnesses $\Mon_\Qloc(\Cc)$ as an accessible Bousfield localization of the functor category $\Fun(\Span_{\calQ}(\calD),\Cc)$. In particular, $\Mon_\Qloc(\Cc)$ is presentable.
\end{lemma}

\begin{proof}
Write $j\colon T\subset \calD$ and $i\colon \calD\subset \PshT$ for the inclusions, and consider the following diagram
\[\begin{tikzcd}
	{\Fun(\Span_{{\Qloc}}(\PshT),\Cc)} & {\Fun(\Span_{\Qloc}(\calD),\Cc)} \\
	{\Fun((\PshT)^{\op},\Cc)} & {\Fun(\calD^{\op},\Cc)} & {\Fun(T^{\op},\Cc).}
	\arrow[""{name=0, anchor=center, inner sep=0}, "{\Span(i)^*}", shift left=2, from=1-1, to=1-2]
	\arrow["{\mathrm{res}}"', from=1-1, to=2-1]
	\arrow[""{name=1, anchor=center, inner sep=0}, "{\Span(i)_*}", shift left=2, hook', from=1-2, to=1-1]
	\arrow["{\mathrm{res}}", from=1-2, to=2-2]
	\arrow[""{name=2, anchor=center, inner sep=0}, "{i^*}", shift left=2, from=2-1, to=2-2]
	\arrow[""{name=3, anchor=center, inner sep=0}, "{i_*}", shift left=2, hook', from=2-2, to=2-1]
	\arrow["{j_*}", hook', from=2-3, to=2-2]
	\arrow["\dashv"{anchor=center, rotate=-90}, draw=none, from=0, to=1]
	\arrow["\dashv"{anchor=center, rotate=-90}, draw=none, from=2, to=3]
\end{tikzcd}\]
We note that the right Kan extension along $i$ exists by the dual of \cite[Th.5.1.5.6.]{htt}. On the other hand, the right Kan extension along $\Span(i)$ exists and the Beck--Chevalley transformation $\res \Span(i)_*\Rightarrow i_* \res$ is an equivalence by \cite[Pr.C.18]{normsmotivic}. Now we claim that a functor $E\colon \Span_{{\Qloc}}(\PshT) \to \calC$ is in $\Mon_\Qloc(\calC)$ if and only if the following condition is satisfied:
\begin{enumerate}
\item[($\ast$)] $E$ is isomorphic to $\Span(i)_\ast E'$ for a functor $E'\colon \Span_{\calQ}(\calD)\to\calC$, and $\res E'$ is right Kan extended from $T$.
\end{enumerate} 

To prove this claim, note that by definition $E$ is in $\Mon_\Qloc(\ccat)$ if and only of $\res E$ preserves limits, and by the dual of \cite[Th.5.1.5.6]{htt} this holds if and only if $\res E$ is right Kan extended along $ij$, or equivalently, if and only if the unit maps $\res E \to i_\ast i^\ast \res E \to  (ij)_\ast (ij)^\ast \res E$ are equivalences. By commutativity of the above diagram, we have $i_\ast i^\ast E\simeq \res \Span(i)_\ast \Span(i)^\ast E$, and so this holds if and only if the unit $E \to \Span(i)_\ast \Span(i)^\ast E$ is an equivalence and $E' = \Span(i)^\ast E$ has the property that $\res E' = i^\ast E$ is right Kan extended from $T$, as claimed.

To show this is an accessible Bousfield localization, we begin by observing that clearly $j_*$ is an accessible Bousfield localization, since it is a right adjoint and commutes with $\kappa$-filtered colimits provided that $\kappa$ is larger then the slice $(T^{\op})_{/X.}$ for all objects $X\in \calD$ (recall that $\calD$ is small). So by \cite[Pr.5.5.4.2]{htt}, a functor is in the essential image of $j_*$ if and only if it is $S$-local for a set $S$ of morphisms in $\Fun(\calD^{\op},\calC)$. But then by definition $E\colon \Span_\calQ(\calD)\to \calC$ is in $\Mon_\Qloc(\calC)$ if and only if it is $\iota_!(S)$-local, where $\iota_!$ is left adjoint to restriction along $\iota\colon \calD^{\op}\to \Span_{\calQ}(\calD)$. Applying the cited proposition again, we conclude that $\Mon_\Qloc\subset \Fun(\Span_{\calQ}(\calD),\calC)$ is an accessible Bousfield localization. 

For presentability, note that $\Span_{\calQ}(\calD)$ is now small and so $\Fun(\Span_{\calQ}(\calD),\calC)$ is presentable.
\end{proof}

\begin{lemma}\label{lm:qglobalspectra_stablepresentable}
    If $\calC$ is stable, then $\Mon_\Qloc(\calC)$ is stable. In particular, $\Sp_\calQ$ is stable. 
\end{lemma}

\begin{proof}
	By definition, $\Mon_\Qloc(\calC)$ sits in a Cartesian diagram
	\[\begin{tikzcd}
		{\Mon_\Qloc(\calC)} & {\Fun(\Span_{\calQ}(\PshT),\calC)} \\
		{\Fun^{\mathrm{R}}(\PshT^{\op},\calC)} & {\Fun(\PshT^{\op},\calC).}
		\arrow[from=1-1, to=1-2]
		\arrow[from=1-1, to=2-1]
		\arrow[from=1-2, to=2-2]
		\arrow[""{name=0, anchor=center, inner sep=0}, from=2-1, to=2-2]
		\arrow["\lrcorner"{anchor=center, pos=0.125}, draw=none, from=1-1, to=0]
	\end{tikzcd}\]
	If $\calC$ is stable, then this is a diagram of stable categories and exact functors, and therefore the pullback is again stable.
\end{proof}

\subsection{As a symmetric monoidal category}\label{ssec:symmetricmonoidalstructures}
Let $\ccat$ be a presentably symmetric monoidal category and $\calQ\subset\PshT$ be an inductible subcategory. Our next goal is to introduce a symmetric monoidal structure on $\Mon_\calQ(\ccat)$. 
To begin, we note $\Span_\calQ(\PshT)$ is symmetric monoidal with respect to the Cartesian product in $\PshT$ (see \cite[Ex.3.2.2]{CLLR24Bispan}). We note this is not the categorical Cartesian product in $\Span_\calQ(\PshT)$. One would like to say that
\[
\Mon_\calQ(\calC)\subset\Fun(\Span_\calQ(\PshT),\ccat)
\]
inherits the structure of a symmetric monoidal category by localizing Day convolution. However, some care is needed to account for the fact that $\Span_\calQ(\PshT)$ is a large category. Our work in this section will show that this is nonetheless essentially the case. We proceed as follows.

\begin{construction}
Let $\calD\subset\PshT$ be the full subcategory spanned by the objects of $\calQ$. As $\calD$ is closed under products, there is an induced symmetric monoidal structure $\Span_\Qloc(\calD)^\otimes$ on $\Span_\Qloc(\calD)$. As $\Span_\Qloc(\calD)$ is small, we may then equip
\[
\Fun(\Span_\calQ(\calD),\calC)
\]
with the Day convolution symmetric monoidal structure, see \cite[\textsection2.2.6]{ha}.
\end{construction}

\begin{prop}
The inclusion 
\[
\Mon_\Qloc(\calC) \hookrightarrow \Fun(\Span_{\calQ}(\calD),\calC)
\]
is the right adjoint in a symmetric monoidal Bousfield localization.
\end{prop}

\begin{proof}
By \cite[Pr.3.3.4]{hinich}, it suffices to prove that the maps sent to equivalences by the localization 
\[
\Fun(\Span_{\calQ}(\calD),\calC) \to \Mon_\Qloc(\calC) 
\]
are closed under the tensor product. In fact, since the tensor product commutes with colimits in both variables separately, it suffices to prove that the tensor product of maps in the generating class $\iota_!(S)$ from the proof of \cref{lem:Qcom_localization} are again sent to equivalences. Equip $\calD^{\op}$ with the coCartesian monoidal structure, so that $\calD^{\op} \to \Span_{\calQ}(\calD)$ is a strong symmetric monoidal functor. By \cite[Pr.3.6]{benmosheschlank}, the functor 
\[\iota_!\colon \Fun(\calD^{\op},\calC)\to \Fun(\Span_{\calQ}(\calD),\calC) 
\] 
is strong symmetric monoidal if we equip the source and target with the Day convolution symmetric monoidal structure. Using this fact, a simple calculation shows that it suffices to prove that the tensor product of maps in $S$ is again sent to an equivalence by the localization 
\[
\Fun(\calD^{\op},\calC)\to \Fun(T^{\op},\calC).
\]
This is clear, as the Day convolution symmetric monoidal structure on $\Fun(\calD^{\op},\calC)$ is equivalent to the pointwise symmetric monoidal structure \cite[Example 3.14.]{denissilluca} and the localization is given by restricting along the inclusion $T^{\op} \subset \calD^{\op}$.
\end{proof}

\begin{cor}
$\Mon_\calQ(\calC)$ admits a unique symmetric monoidal structure for which the localization
\[
\Fun(\Span_{\calQ}(\calD),\calC) \to \Mon_\Qloc(\calC)
\]
is strong symmetric monoidal.
\qed
\end{cor}

This allows us to make the following definition.

\begin{mydef}\label{df:qambiabglobal_rings}
A \emph{$\calQ$-ambidextrous $\E_\infty$ ring} in $\calC$ is an object of $\CAlg(\Mon_\calQ(\calC))$. For $\calC = \Sp$, we refer to the objects of $\CAlg(\Sp_\calQ)$ as \emph{$\calQ$-ambidextrous $\E_\infty$ rings}.
\end{mydef}


Given two symmetric monoidal categories $\jdiagram$ and $\ccat$, commutative monoids for the Day convolution monoidal structure on $\Fun(\jdiagram,\ccat)$ (when it exists) are exactly the lax symmetric monoidal functors $\jdiagram\to\ccat$. Our next goal is to derive a similar description of $\calQ$-ambidextrous $\ab$-global $\E_\infty$ rings. This will be an essential ingredient to our decategorification theorems in the next section. We require some preparation, which we undertake now.

\begin{mydef}
Let $F\colon \Cc\to \calD$ be a lax symmetric monoidal functor and let $X,Y\in \calD$. Write
\begin{equation}\label{eq:otimes_on_slices}
\otimes\colon \calC_{X/}\times \calC_{Y/}\to \calC_{X \otimes Y/}
\end{equation}
for the functor which sends a pair $\{(A, X\to F(A)),(B,X\to F(B))\}$ to $(A\otimes B, X\otimes Y\to F(A)\otimes F(B)\to F(A\otimes B))$. We say $F\colon \Cc\to \calD$ is \emph{extendable} if the functor above is initial for all $X,Y\in \calD$.
\end{mydef}

\begin{lemma}
Let $F\colon \Cc\to \calD$ be an extendable lax symmetric monoidal functor of small symmetric monoidal categories. Then, for all $\calE\in \CAlg(\PrL)$, the restriction functor 
\[
\Fun(\calD,\calE)\to \Fun(\calC,\calE)
\]
is strong symmetric monoidal with respect to Day convolution. In particular, right Kan extension along $F$ is canonically lax symmetric monoidal.
\end{lemma}

\begin{proof}
Recall that restriction along $F$ is lax symmetric monoidal with respect to Day convolution, giving a transformation
\[
F^*(I)\otimes^{\textup{Day}} F^*(J)\to F^*(I\otimes^{\textup{Day}} J)
\]
for all $I,J\in \Fun(\calD,\calE)$. Evaluating the transformation above on $C\in \calC$ and applying the formula for the tensor product in Day convolution, we find that the map above is homotopic to the map
\[
\colim_{\otimes \downarrow
 C} IF(A)\otimes IF(B) \to \colim_{\otimes\downarrow F(C)} I(X)\otimes I(Y)
\]
induced on colimits by the functor $\bar{F}\colon {\otimes} \downarrow C \to {\otimes} \downarrow F(C)$ induced by $F$. Applying Quillen's theorem A, the functor $\bar{F}$ is final if and only if the category $( \otimes \downarrow C)_{(X,Y)/}$ is contractible for all $(X,Y)\in {\otimes} \downarrow F(C)$. However, these categories are easily seen to be equivalent to the under-slices of the functor (\ref{eq:otimes_on_slices}). Another application of Quillen's theorem A implies that these are contractible exactly when (\ref{eq:otimes_on_slices}) is initial, proving the statement.
\end{proof}

For the following statement, we equip $\Span_{{\calQ}}(\PshT)$ with the symmetric monoidal structure induced by the Cartesian monoidal structure of $\PshT$.

\begin{lemma}
The symmetric monoidal inclusion $\Span_{\calQ}(\calD)\subset \Span_{{\calQ}}(\PshT)$ is extendable.
\end{lemma}

\begin{proof}
Fix $X,Y\in \PshT$ and consider the commutative square 
\[\begin{tikzcd}
	{\Span_{\calQ}(\calD)_{X/}\times \Span_{\calQ}(\calD)_{Y/}} & {\Span_{\calQ}(\calD)_{X\times Y/}} \\
	{(\calD_{/X})^{\op}\times (\calD_{Y/})^{\op}} & {(\calD_{/X\times Y})^{\op}.}
	\arrow["\otimes", from=1-1, to=1-2]
	\arrow[from=2-1, to=1-1]
	\arrow["\times", from=2-1, to=2-2]
	\arrow[from=2-2, to=1-2]
\end{tikzcd}\]
As observed in the proof of \cite[Pr.C.18]{normsmotivic}, the vertical functors are both initial. However, the bottom horizontal functor is also initial, since it admits a right adjoint. Therefore, the top functor is initial by right cancellability of initial functors.
\end{proof}

Given a presentable category $\calC$, let $\Fun(\Span_{{\calQ}}(\PshT),\calC)^{\mathrm{Day}}$ denote the Day convolution operad, see \cite[Con.2.2.6.7]{ha}. Due to size issues, this is not a symmetric monoidal category (and may in fact fail to be locally small). Nevertheless, we have the following result.

\begin{prop}\label{prop:cmon_on_big_spans}
There is a fully faithful inclusion of categories \[\CAlg(\Mon_\calQ(\calC))\subset \CAlg(\Fun(\Span_{\calQ}(\PshT),\calC)^{\mathrm{Day}}),\] with essential image those objects whose underlying functor is right Kan extended from $\Span_{\calQ}(\calD)$.
\end{prop}

\begin{proof}
Let $\widehat{\spaces}$ be a universe of large spaces with respect to which $\ccat$ is small, and let $\widehat{\ccat} \coloneqq \Fun^{\text{small}-R}(\ccat^\op,\widehat{\spaces})$ be the category of functors $\ccat^\op \to \widehat{\spaces}$ which preserve small limits, so that $\widehat{\ccat}$ is presentable with respect to $\widehat{\spaces}$. Then $\Fun(\Span_{{\calQ}}(\PshT),\widehat{\calC})$ is now a symmetric monoidal category (in this larger universe), and so the previous two lemmas combine to give a fully faithful and lax symmetric monoidal functor
\[
\mathrm{Ran}\colon \Fun(\Span_{{\calQ}}(\PshT),\widehat{\calC}) \to \Fun(\Span_{{\calQ}}(\PshT),\widehat{\calC}).
\]
Therefore, we obtain a fully faithful inclusion
\[ \Mon_\calQ(\calC) \subset \CAlg(\Fun(\Span_{\calQ}(\calD),\widehat{\calC})) \subset \CAlg(\Fun(\Span_{{\calQ}}(\PshT),\widehat{\calC})).\]
As $\calC\to \widehat{\calC}$ commutes with all small limits, this factors through 
\[\CAlg(\Fun(\Span_{{\calQ}}(\PshT),\calC)^{\mathrm{Day}})\subset \CAlg(\Fun(\Span_{{\calQ}}(\PshT),\widehat{\calC})),\] providing the necessary fully faithful inclusion. Its essential image is easily characterized by unwinding the definitions.
\end{proof}

\begin{prop}\label{prop:qcommutativeglobalringsaslaxfunctors} 
Let $\calC\in \CAlg(\PrL)$. There is a fully faithful inclusion 
\[
\CAlg(\Mon_\calQ(\calC)) \subset \Fun^{\otimes\textup{-lax}}(\Span_{{\calQ}}(\PshT),\calC)
\]
which identifies the category of $\calQ$-ambidextrous $\E_\infty$ rings in $\calC$ with the subcategory of lax symmetric monoidal functors $\Span_{{\calQ}}(\PshT)\to \calC$ whose underlying functor is a $\calQ$-ambidextrous monoid in $\calC$.
\end{prop}

\begin{proof}
The universal property of Day convolution, see \cite[Ex.2.2.6.9]{ha}, gives a commutative diagram
\[\begin{tikzcd}
	{\CAlg(\Fun(\Span_{{\calQ}}(\PshT),\calC)^{\mathrm{Day}})} && {\Fun^{\otimes\textup{-lax}}(\Span_{{\calQ}}(\PshT),\calC)} \\
	& {\Fun(\Span_{{\calQ}}(\PshT),\calC).}
	\arrow["\simeq", from=1-1, to=1-3]
	\arrow["{\mathrm{fgt}}"', from=1-1, to=2-2]
	\arrow["{\mathrm{fgt}}", from=1-3, to=2-2]
\end{tikzcd}\]
Combining this with \cref{prop:cmon_on_big_spans} gives the result.
\end{proof}

We will typically think of a $\calQ$-ambidextrous $\E_\infty$ ring in $\Cc$ through the lens of the previous proposition. For example, with this perspective, we can highlight some of the important additional structure contained in a $\calQ$-ambidextrous $\E_\infty$ ring.

\begin{rmk}
	Suppose $E\in \CAlg(\Mon_\calQ(\Cc))$ is a $\calQ$-ambidextrous $\E_\infty$ ring in $\Cc$, viewed as a lax symmetric monoidal functor
	\[
	\Span_{{\calQ}}(\PshT)\to \Cc	
    \]
	whose restriction to $\PshT^{\op}$ preserves limits.
	\begin{enumerate}
		\item Consider $\PshT$, equipped with the Cartesian symmetric monoidal structure. Then the inclusion $\PshT^{\op} \subset \Span_{{\calQ}}(\PshT)$ is strong monoidal. Precomposing $E$ by this functor, we obtain a lax monoidal functor 
		$\PshT^{\op} \to \Cc.$ 
        Note that $\PshT^{\op}$ is coCartesian symmetric monoidal, and so by the universal property of coCartesian symmetric monoidal categories, the functor above is equivalent to a functor $\PshT^{\op} \to \CAlg(\Cc)$. In particular, $E(X)$ is an $\E_\infty$ ring for all $X\in \PshT$ and $\mathrm{res}_f\colon E(Y)\to E(X)$ is a map of $\E_\infty$ rings for all $f\colon X\to Y$.
		\item Consider a map $q\colon Y\to X$ in $\Qloc$. As explained in \cref{prop:rel_lift_to_mod} below, the span $Y= Y \to X$ is a map of $X$-modules in $\Span_{{\calQ}}(\PshT)$, where $X$ is a commutative algebra in $\Span_{{\calQ}}(\PshT)$ via the previous point and $Y$ is a $X$-module via $\res_q$. This structure is preserved by the lax symmetric monoidal functor $E$, and so we conclude that the map 
	   \[
	   \mathrm{tr}_q\colon E(Y) \to E(X)
	   \]
	is canonically a map of $E(X)$-modules. We may think of this as an analog of \emph{Frobenius reciprocity}.
\end{enumerate}
\end{rmk}

Let us explain the claim made in point 2 of the remark above. For later use, we will even record a coherent form of the statement.

\begin{prop}\label{prop:lifts_to_modules}
Let $\Cc\in \CAlg(\PrL)$ and $E$ be a $\calQ$-ambidextrous $\E_\infty$ ring in $\Cc$. Then there is a lift
\begin{center}\begin{tikzcd}
&\Mod_{E(\ast)}(\Cc)\ar[d]\\
\Span_{{\calQ}}(\PshT)\ar[r,"E"]\ar[ur,dashed]&\Cc
\end{tikzcd}
\end{center}
of $E$ to a $\calQ$-ambidextrous $\E_\infty$ ring in $\Mod_{E(\ast)}(\calC)$.
\end{prop}

\begin{proof}
The point $\ast$ is the unit of $\Span_{{\calQ}}(\PshT)$, and so we can define the desired lift of $E$ as the composite
\[
\Span_{{\calQ}}(\PshT)\simeq \Mod_{\ast}\Span_{{\calQ}}(\PshT) \xrightarrow{E} \Mod_{E(\ast)}(\Cc).\qedhere
\]
\end{proof}

We can also obtain a version of this statement relative to any $X\in \PshT$ via the following proposition. In the following proposition, we denote by $\Span_{{\calQ}}(\PshT{ }_{/X})$ the category of spans in $\PshT{}_{/X}$ of the form 
\[\begin{tikzcd}
	Y & W & Z\\
	& X.
	\arrow[from=1-1, to=2-2]
	\arrow[from=1-2, to=1-1]
	\arrow["{\in\bar{\calQ}}", from=1-2, to=1-3]
	\arrow[from=1-2, to=2-2]
	\arrow[from=1-3, to=2-2]
\end{tikzcd}\]
This is again symmetric monoidal via pullback in $\PshT_{/X}$. 

\begin{prop}\label{prop:rel_lift_to_mod}
Let $X\in \PshT$. Then the restriction 
\[
\Span_{{\calQ}}(\PshT_{/X})\to \Span_{{\calQ}}(\PshT)
\]
is canonically lax symmetric monoidal. In particular, given $E\in \CAlg(\Mon_\Qloc(\calC))$ there exists a diagram
\[
\begin{tikzcd}
	{\Span_{{\Qloc}}(\PshT{}_{/X})} & {\Mod_{E(X)}(\calC)} \\
    \Mod_X({\Span_{{\Qloc}}(\PshT)})&{\Mod_{E(X)}(\calC)}
    \\
	{\Span_{{\Qloc}}(\PshT)} & \calC
	\arrow["{E_{|X}}", from=1-1, to=1-2]
	\arrow[from=1-1, to=2-1]
    \arrow[from=2-1, to=3-1]
	\arrow["\mathrm{fgt}", from=2-2, to=3-2]
	\arrow[from=1-2, to=2-2, equals]
	\arrow["E", from=3-1, to=3-2]
    \arrow[from=2-1,to=2-2]
\end{tikzcd}\]
of lax symmetric monoidal functors.
\end{prop}

\begin{proof}
Consider the adjunction 
\[
\mathrm{fgt} \colon \PshT{}_{/X} \rightleftarrows \PshT \colon {-\times X}, 
\]
and note that $-\times X$ is strong symmetric monoidal. By an application of \cite[Cor.C.21]{normsmotivic}, this induces an adjunction
\[
\Span(-\times X)\colon \Span_{{\calQ}}(\PshT) \rightleftarrows \Span_{{\calQ}}(\PshT{}_{/X})\colon \Span(\mathrm{fgt}).
\]
However $\Span(-\times X)$ is again strong monoidal by definition, and so we conclude that $\Span(\mathrm{fgt})$ is canonically lax symmetric monoidal. The ``in particular'' statement follows by applying the argument of the previous proposition to the lax symmetric monoidal composite
\[
\Span_{{\calQ}}(\PshT{}_{/X})\to \Span_{{\Qloc}}(\PshT)\xrightarrow{E} \calC. \qedhere
\]
\end{proof}

\subsection{Homology theories associated with ambidextrous \texorpdfstring{$\E_\infty$}{E-infty} rings}\label{ssec:homologytheories}

We now explain how to define ``homology theories'' associated with a $\calQ$-ambidextrous $\E_\infty$ ring $E$ in $\calC$. Informally, our definition of $E$-homology $\D E(\bs)$ will proceed by requiring it to agree with $E$-cohomology on \emph{representables}, $\D E(A) \coloneqq E(A)$ for $A\in T$. However, the functoriality of this assignment will be in transfers, rather than restrictions. This assignment is then Kan extended to a homology theory on all presheaves on $T$. For this definition to be sensible it is clearly necessary for $T\subset \calQ$, and so we make the following:

\begin{ass}
Throughout this subsection, the inductible subcategory $\calQ\subset\PshT$ contains the full subcategory $T\subset\PshT$ spanned by the representables.
\end{ass}

For later use we will also make our definition relative to any object $Y\in \calQ$.

\begin{ex}
A key example is the inductible subcategory $\pi\subset \abglobalspaces$. Applying the constructions of this section to tempered cohomology will give the \emph{tempered homology theories} studied in \Cref{ssec:temperedHOMOLOGY}. An important non-example is $\calO\subset \abglobalspaces$. In particular, we do not associate a homology theory to global spectra in the sense of Schwede.
\end{ex} 

\begin{remark}
Observe that left Kan extension defines an analog of the suspension spectrum functor $\Sigma^\infty_+\colon \PshT \to \Sp_\calQ$. Given $E\in \Sp_\calQ$, a natural guess for the definition of $E$-homology, generalizing the definition in the case of equivariant and ordinary spectra, is $X\mapsto \Sigma^\infty_+X\otimes E.$ Our definition will generally be different from this, and is instead geared to ensure good duality properties. An important exception is the case $T, \calQ= \mathrm{Orb}_G$, where our definition does recover the usual definition of $G$-equivariant homology, see \Cref{rmk:ord_eq_homology}.
\end{remark}

\begin{mydef}\label{def:homology}
Let $E\colon \Span_{\calQ}(\PshT)\to \calC$ be a $\calQ$-ambidextrous $\E_\infty$ ring in $\calC$ and let $Y\in \calQ$. Consider the lax symmetric monoidal functor
\[
E_{|Y}\colon \Span_{\calQ}(\PshT_{/Y})\to \Mod_{E(Y)}\calC.
\]
Note that there is an inclusion
\[
\calQ_{/Y}\subset \Span_{\calQ}(\Spc^{\gl}_\ab{}_{/F})
\]
of symmetric monoidal categories. By restricting $E_{|F}$ along this inclusion, we obtain a lax symmetric monoidal functor
\[
\calQ_{/F} \to \Span_{\calQ}(\PshT_{/F}) \to \Mod_{E(F)} \Cc
\]
courtesy of \Cref{prop:lifts_to_modules}. Taking the left Kan extension of this functor along the inclusion 
\[
\calQ_{/F} \subset \PshT_{/F}
\]
yields a functor
\[
\D E\colon \PshT_{/Y} \to \Mod_{E(F)}\calC,
\]
which we call \emph{$E$-homology}. Since left Kan extension along a strong monoidal functor is again strong monoidal on Day convolution \cite[Cor.3.8]{benmosheschlank}, and so induces a functor on commutative algebras, we find that $\D E$ is again canonically lax symmetric monoidal. 
\end{mydef}

The construction above is not generally well-behaved; the following proposition explains when it is.

\begin{prop}\label{prop:homology_underlying}
Let $E$ be a $\calQ$-ambidextrous $\E_\infty$ ring in $\calC$ and let $Y\in \calQ$. If the composite 
\[
\calQ\to \Span_{{\Qloc}}(\PshT)\xrightarrow{E} \calC 
\]
is left Kan extended from $T$, then $E$-homology 
\[
\D E\colon \PshT_{/Y} \to \Mod_{E(Y)}\calC
\]
is colimit preserving.
\end{prop}

\begin{remark}
We will see in \Cref{prop:homology_colim} that this condition is satisfied for the $\calQ$-ambidextrous $\E_\infty$ rings in $\calC$ which are obtained by decategorification, i.e.~from \Cref{thm:generaldecategorification}.
\end{remark}

\begin{proof}[Proof of Proposition~\ref{prop:homology_underlying}]
To prove this statement, we may forget $E(Y)$-module structures. Then the result follows immediately from the fact that a functor $\PshT_{/Y}\to \calC$ is colimit preserving if and only if it is left Kan extended from $T_{/Y}$.
\end{proof}

Let us make this homology construction explicit in the case of the $G$-equivariant stable homotopy theory of \Cref{ex:GequivSHT}.

\begin{mydef}\label{def:ghomology}
Let $\calC$ be a stable presentable symmetric monoidal category and let $E\in \CAlg(\Mon_G(\calC))$. The $G$-homology theory associated with $E$ is the lax symmetric monoidal and colimit preserving functor
\[
E[\bs]\colon \spaces_G \to \Mod_{E(\ast)}\calC
\]
obtained from the composite
\[
\Fin_G\to \Span(\Fin_G) \xrightarrow{E} \Mod_{E(\ast)}\calC
\]
by left Kan extension along the inclusion $\Fin_G\subset \Spc_G$.
\end{mydef}

\begin{remark}\label{rmk:ord_eq_homology}
When $\calC = \Sp$, our definition of $G$-homology agrees with the classical definition. This follows from the fact that the functor $\Span(\Fin_G)\to \Sp^G$ sending $G/H$ to $\Sigma^{\infty}_+ G/H$ is strong symmetric monoidal, combined with the fact that duality in $\Span(\Fin_G)$ is obtained by reflecting spans.
\end{remark}

\begin{construction}\label{restricted_G-spectrum}
There is a canonical fully faithful and strong symmetric monoidal functor
\[
(\bs)//G\colon \Fin_G \to \abglobalspaces{}_{/\BG},\qquad G/H \mapsto (G/H)//G\simeq \mathbf{B} H \to (G/G)//G \simeq \BG.
\]
Thus for any inductible subcategory $\calQ\subset\abglobalspaces$ which contains faithful morphisms between classifying spaces of subgroups of $G$, restriction along $\Span(\Fin_G)\to\Span_{\calQ}(\abglobalspaces{}_{/\BG})$ defines a lax symmetric monoidal and limit preserving functor
\[
\Mon_{\calQ}(\calC) \to \Mon_G(\calC),\qquad E \mapsto E_G.
\]
Moreover, by construction, the diagram
\begin{center}\begin{tikzcd}
\spaces_G\ar[r,"{E_G[\bs]}"]\ar[d,tail,"(\bs)//G"]&\Mod_{E_G(\ast)}\ccat\ar[d,"\simeq"]\\
\abglobalspaces{}_{/\bfB G}\ar[r,"\mathsf{D}E"]&\Mod_{E(\bfB G)}\ccat
\end{tikzcd}\end{center}
commutes, where $E_G[\bs]$ and $\mathsf{D}E$ are the $G$-equivariant homology theories constructed in \cref{def:homology} and \cref{def:ghomology}, and the right equivalence identifies $E(\bfB G)\simeq E_G(\ast)$.
\end{construction}

\section{Parametrized category theory and decategorification}\label{sec:parametrized}
The previous section equips us with a well-behaved category of $\calQ$-ambidextrous $\E_\infty$ rings, which encodes a multiplicative cohomology theories equipped with a coherent collection of transfers for maps that are locally in $\calQ$. In this section we will prove \cref{thm:generaldecategorification}, which produces examples of such objects by way of a \emph{decategorification} process. For motivation, recall the existence of the functor
\[\CAlg \to \CAlg(\PrLst), \qquad A \mapsto \Mod_A\]
sending an $\E_\infty$ ring to its category of modules. This functor has a right adjoint, a decategorification functor, sending $\Cc\in \CAlg(\PrLst)$ to the $\E_\infty$ ring of endomorphisms of the unit in $\Cc$. Our construction of $\calQ$-ambidextrous $\E_\infty$ rings, specifically their values and functoriality in restriction, will similarly proceed by passing to the endomorphisms of the unit in a suitable diagram $\calC(\bs)$ of presentable stable categories, i.e.~a parametrized category. 

The construction of coherent transfer maps will be inspired by that of \cite{harpazambi}. Namely, Harpaz observes that coherent transfer maps in $K(n)$-local cohomology theories are provided by the \emph{higher} semiadditivity of the $K(n)$-local category, in the sense of \cite{ambi}. Similarly, we will build the coherent transfers encoded in a $\calQ$-ambidextrous $\E_\infty$ ring after assuming the so-called \emph{$\calQ$-parametrized semiadditivity} of $\calC(\bs)$. This notion was developed in \cite{CLLSpans} and the reader may benefit from consulting \emph{ibid} throughout this section.

To declutter our notation, and to assist in future applications, we again work in a generic parametrized setting. To this end, we again fix an arbitrary small category $T$. Recall we denote by $\calP$ the category of presheaves on $T$.

We begin in \Cref{subsec:globalsemi} by recalling notions of semiadditivity and stability in parametrized higher category theory, and define variants depending on an appropriate subcategory $\calQ \subset \calP$. In \Cref{ssec:monoidalparametrized}, we discuss parametrized monoidality and its interaction with semiadditivity. In particular, we prove that the universal $\calQ$-semiadditive category is an idempotent algebra in a suitable category. This is enough to equip the unit of a symmetric monoidal $\calQ$-semiadditive category with the structure of a $\calQ$-ambidextrous $\E_\infty$-ring, but only ``internally" to the parametrized category.

To go from this to \Cref{thm:generaldecategorification} requires further work. In  \Cref{ssec:unstraighteningparametrized} we observe that the unstraightenings of a monoidal parametrized category is itself monoidal, and give a new presentation for this. In a sense, this allows us to compute the externalization of the structure obtained before. In \Cref{ssec:parametrizedglobalsections}, we produce a parametrized version of the global sections functor associated with a monoidal stable category. Finally, in \Cref{ssec:decategoricification} we put all these pieces together to prove our general decategorification statement \Cref{thm:generaldecategorification}.

\subsection{Semiadditive and stable parametrized categories}\label{subsec:globalsemi}
We begin with the basic notions of parametrized category theory. 

\begin{mydef}
A \emph{$T$-category} $\calC$ is a limit preserving functor $\calC\colon \PshT^{\op}\to \Cat$. This defines a category 
\[
\Cat(T) \subset \Fun(\PshT^{\op},\Cat),
\]
whose morphisms we call \emph{functors} of $T$-categories.
\end{mydef}

\begin{rmk}
Clearly, right Kan extension and restriction define an equivalence
\[
\Cat(T)\simeq\Fun(T^\op,\Cat).
\]
In particular, $\Cat(T)$ is presentable.
\end{rmk}

For the purposes of this article, we will be most interested in the case $T = \Glo_{\ab}$. We call $\Glo_{\ab}$-categories \emph{ab-global categories}.

\begin{mydef}
Given a $T$-category $\calC$ we define its \emph{underlying category} $\calC(\ast)$ to be the value of $\calC$ on $\ast\in \PshT$. Since $\calC$ is limit preserving, we obtain an equivalence $\calC(\ast) \simeq \lim_{T^{\op}} \calC$. 
\end{mydef}

\begin{remark}
The category $\Cat(T)$ is naturally $\Cat$-enriched. Given $\ab$-global categories $\ccat$ and $\dcat$, we write $\Fun(\Cc,\calD)$ for the category of functors from $\Cc$ to $\calD$.
\end{remark}

\begin{mydef}
We define the $T$-category of \emph{$T$-spaces}:
\[
\ul{\Spc}_T \colon \PshT^{\op} \to \Cat, \quad X\mapsto \PshT_{/X},
\]
where functoriality is given by pullback. This is a $T$-category by descent for the $\infty$-topos $\PshT$  
\end{mydef}

\begin{mydef}\label{def:Q_cocomplete}
Let $\calQ\subset \PshT$ be an inductible subcategory. A $T$-category $\calC$
	\begin{enumerate}
		\item Has \emph{$\calQ$-colimits} if the restriction $p^*\colon \calC(Y)\to \calC(X)$ along any map $p\colon X\to Y$ in $\bar{\calQ}$ admits a left adjoint $p_!\colon \calC(X)\to \calC(Y)$ and those left adjoints satisfy base change, meaning that every square 
		\[\begin{tikzcd}
			{\calC(Y)} & {\calC(Z)} \\
			{\calC(X)} & {\calC(X\times_YZ),}
			\arrow["{p^*}", from=1-1, to=1-2]
			\arrow["{f^*}"', from=1-1, to=2-1]
			\arrow["{\mathrm{pr}_2^*}", from=1-2, to=2-2]
			\arrow["{\mathrm{pr}_1^*}", from=2-1, to=2-2]
		\end{tikzcd}\]    
		induced by a pullback square in $\PshT$ such that $p$ is in $\bar{\calQ}$, is vertically left adjointable, i.e.\ the natural transformation $\mathrm{pr}_{1!}\circ \mathrm{pr}_2^\ast \to f^\ast \circ p_!$ is an equivalence.
		\item Has \emph{$\calQ$-limits} if $\calC^{\op}\coloneqq (\bs)^{\op}\circ \calC$ has $\calQ$-colimits. If $\Cc$ has $\calQ$-limits, we will denote the right adjoint of restriction $p^*\colon \Cc(Y)\to \Cc(X)$ along a map $p\colon X\to Y$ in $\bar{\calQ}$ by $p_*\colon \Cc(X)\to \Cc(Y)$.
		\item Has \emph{fiberwise $\mathcal{K}$-colimits} for a collection of categories $\mathcal{K}$ if $\calC\colon T^{\op} \to \Cat$ factors through the subcategory $\Cat^{\calK\textup{-L}}\subset\Cat$ of categories with $\mathcal{K}$-colimits and functors which preserve $\mathcal{K}$-colimits,
		\item Has \emph{fiberwise $\mathcal{K}$-limits} if $\calC^{\op}$ has fiberwise $\mathcal{K}$-colimits, in other words $\calC$ is required to factor through $\Cat^{\calK\textup{-R}}$.
		\item Is \emph{$\calQ$-semiadditive} if it admits $\calQ$-colimits and $\calQ$-limits and if for every \emph{truncated} map $f\colon X\to Y$ in $\bar{\Qloc}$ the canonical norm map
	\[
	\mathrm{Nm}_f\colon f_!\to f_*,
	\]
    from the left to the right adjoint of $f$, defined inductively on the degree of truncation of $f$, is an equivalence.
	\end{enumerate}
\end{mydef}

\begin{remark}
For details on the definition of the norm maps $\mathrm{Nm}_f$, originally defined by Hopkins--Lurie \cite{ambi}, we direct the reader to \cite[Con.3.3]{CLLSpans}.
\end{remark}

\begin{mydef}\label{def:qstable}
	Suppose $\calC$ and $\calD$ are two $T$-categories with $\calQ$-colimits. A functor $F\colon \calC\to \calD$ preserves $\calQ$-colimits if for every map $f\colon X\to Y$ in $\bar{\Qloc}$, the square
	\[\begin{tikzcd}
		{\calC(Y)} & {\calD(Y)} \\
		{\calC(X)} & {\calD(X)}
		\arrow["F_Y", from=1-1, to=1-2]
		\arrow["{f^*}"', from=1-1, to=2-1]
		\arrow["{f^*}", from=1-2, to=2-2]
		\arrow["F_X", from=2-1, to=2-2]
	\end{tikzcd}\]
	is vertically left adjointable. In this case, there is a canonical equivalence $F_Y f_!\simeq f_! F_X$. One defines $\calQ$-limit preserving functors between $T$-categories with $\calQ$-limits dually. We say $F$ preserve fiberwise $\calK$-(co)limits if each functor $F(X)\colon \calC(X)\to \calD(X)$ preserves $\calK$-(co)limits. 
\end{mydef}

For functors between $\calQ$-semiadditive $T$-categories we have the following pleasant self-duality. 

\begin{lemma}[{\cite[Cor.3.17]{CLLSpans}}]\label{lem:QexactFunctor}
	Suppose $\calC$ and $\calD$ are $\calQ$-semiadditive. Then a functor $F\colon \calC\to \calD$ preserves $\calQ$-colimits if and only if it preserves $\calQ$-limits.
\end{lemma}

\begin{mydef}\label{def:global_cat_subcat}
	Consider the following subcategories of $\Cat(T)$:
	\begin{enumerate}
		\item $\Cat(T)_{\calQ\textup{-co}}$ is spanned those $T$-categories which admit $\calQ$-colimits and those functors which preserve $\calQ$-colimits. 
		\item $\Cat(T)_{\calQ\textup{-sa}}$ is the full subcategory of $\Cat(T)_{\calQ\textup{-co}}$ spanned by the $\calQ$-semiadditive $T$-categories.
		\item $\Cat(T)_{\calQ\textup{-st}}$ is the full subcategory of $\Cat(T)_{\calQ\textup{-sa}}$ spanned by those $\calQ$-semiadditive categories $\calC$ which moreover factor through the subcategory $\Cat^{\mathrm{st}}\subset \Cat$ of stable categories and exact functors. Such $T$-categories are called \emph{$\calQ$-stable}.
	\end{enumerate}
\end{mydef}

The following is a crucial example of a $\calQ$-semiadditive category.

\begin{example}\label{ex:SpanQ}
In \cite[Con.4.1]{CLLSpans}, a $\calQ$-semiadditive $T$-category $\ul{\Span}(\calQ)$ is constructed for which
\[
\ul{\Span}(\calQ)(X) \coloneqq \Span(\bar{\Qloc}_{/X}).
\]
Given a map $f\colon X\to Y$, the functor 
\[
f^*\colon \Span(\bar{\Qloc}_{/Y})\to \Span(\bar{\Qloc}_{/X})
\]
is given by applying $\Span(\bs)$ to the pullback functor $f^*\colon \bar{\Qloc}_{/Y}\to \bar{\Qloc}_{/X}$. Moreover, when $f$ itself is in $\bar{\Qloc}$, the left (and also right) adjoint $f_!$ of $f^*$ is given by $f_!(A\to X) = A\to X\to Y$, the functor on spans induced by the postcomposition functor. 

Consider the object $\ast\in \Span(\calQ)\simeq \ul{\Span}(\calQ)(\ast)$. By Theorem 5.1 of \emph{op.~cit.}, the pair $(\ul{\Span}(\calQ), \ast)$ admits the following universal property: for every $\calQ$-semiadditive $T$-category $\calC$, the functor
\[
\ev_\ast\colon\Fun^{\calQ\textup{-co}}(\ul{\Span}(\calQ),\calC) \to \calC(\ast)
\]
is an equivalence, where the left-hand side denotes the full subcategory of $\Fun(\ul{\Span}(\calQ),\Cc)$ spanned by the $\calQ$-colimit preserving functors.
\end{example}

All of the categories of \Cref{def:global_cat_subcat} are closed under limits in $\Cat(T)$. We will prove this in the one case of relevance to us; the other cases are easier.

\begin{prop}\label{prop:limits_of_stable}
	$\Cat(T)_{\calQ\textup{-st}}$ admits all small limits, and these are preserved by the inclusion $\Cat(T)_{\calQ\textup{-st}} \to \Cat(T)$.
    \end{prop}

\begin{proof}
	This follows from combining \cite[Cor.4.7.4.18]{ha} with the observation that $\Cat^{\st}\subset \Cat$ preserves limits. 
\end{proof}

\subsection{Symmetric monoidal parametrized categories}\label{ssec:monoidalparametrized}
Let us now consider monoidal structures in the parametrized world.

\begin{mydef}
    The category $\Cat(T)$ of $T$-categories is symmetric monoidal via the Cartesian symmetric monoidal structure, giving the category
	\[
	\Cat(T)^\otimes\coloneqq \CAlg(\Cat(T)) \simeq \Fun^{\mathrm{R}}(\calP^{\op},\Cat^\otimes)
	\] 
    of \emph{symmetric monoidal $T$-categories}.
\end{mydef}

\begin{example}
Suppose $\Cc$ is a symmetric monoidal $T$-category which admits fiberwise finite products. Then we obtain a symmetric monoidal $T$-category $\Cc^\times$ by equipping each $\Cc(X)$ with the Cartesian symmetric monoidal structure.
\end{example}

\begin{mydef}
	By \cite[Cor.8.2.5]{Martini2022presentable}, given any class $\calK\subset \Cat(T)$ of $T$-categories, there is a symmetric monoidal structure on the category $\Cat(T)_{\calK\textup{-co}}$ of $\calK$-cocomplete $T$-categories for which a functor $\calC \otimes \calD \to \calE$ from the tensor product of two $\calK$-cocomplete $T$-categories is equivalent to a functor $\calC\times \calD \to \calE$ which preserves $\calK$-colimits in each variable separately. Applying this to the subcategory $\Cat(T)_{\calQ\textup{-co}} \subset \Cat(T)$, we obtain a symmetric monoidal structure on $\Cat(T)_{\calQ\textup{-co}}$.
\end{mydef}

We will now make the structure of a commutative algebra in $\Cat(T)_{\calQ\textup{-co}}$ explicit. Similar statements hold in the other two cases.

\begin{mydef}\label{def:qlexsymmetricmonoidalcat}
A \emph{$\calQ$-lex symmetric monoidal $T$-category} is a symmetric monoidal $T$-category $\calC$ such that 
	\begin{enumerate}
		\item $\calC$ is a $\calQ$-cocomplete $T$-category;
		\item For every map $f\colon X\to Y$ in $\bar{\Qloc}$, the projection map 
		\[
		f_!(f^*C\otimes D) \to C\otimes f_! D 
		\]
		is an equivalence for all $C\in \calC(Y)$ and $D\in \calC(X)$.
	\end{enumerate}
	We define $\Cat(T)^\otimes_{\calQ\textup{-co}}\subset \Cat(T)^\otimes$ to be the subcategory spanned by the $\calQ$-cocomplete symmetric monoidal $T$-categories and those strong monoidal functors $F\colon \calC\to \calD$ which preserve $\calQ$-colimits.
\end{mydef}

\begin{prop}
	There is an equivalence
	$\Cat(T)^\otimes_{\calQ\textup{-co}} \simeq \CAlg(\Cat(T)_{\calQ\textup{-co}})$. 
\end{prop}

\begin{proof}
The symmetric monoidal category $\Cat(T)_{\calQ\textup{-co}}$ is a faithful suboperad of $\Cat(T)^\times$ by definition, so it is a property for a symmetric  monoidal $T$-category to lie in $\Cat(T)^\otimes_{\calQ\textup{-co}}$, and similarly for functors. Inspecting \cite[Df.2.6.1.1]{Martini2022presentable} shows that the condition is precisely that given above.
\end{proof}

\begin{prop}\label{prop:spanqmonoidal}
The product in $\calQ$ turns $\ul{\Span}(\calQ)$ into a $\cal{Q}$-lex symmetric monoidal $T$-category.
\end{prop}

\begin{proof}
Note that $\Qloc_{/\bullet}\colon \PshT^{\op}\to \Cat$ factors through $\Cat^\times$. Passing to the associated Cartesian monoidal categories, we obtain a diagram of symmetric monoidal categories. Now we may apply $\Span(\bs)$ to this to obtain $\ul{\Span}(\calQ)$ as a diagram of symmetric monoidal categories, using that $\Span(\bs)$ preserves finite products.

From here it is also simple to compute that the projection map is an equivalence, since under the description of the $f_!$ from \cref{ex:SpanQ} it unravels to the base change equivalence. 
\end{proof}

\begin{mydef}
We write $\ul{\Span}(\calQ)^\otimes$ for the $\calQ$-lex symmetric monoidal structure on $\ul{\Span}(\calQ)$ guaranteed by \cref{prop:spanqmonoidal}.
\end{mydef}

\begin{theorem}
	The forgetful functor 
	\[
	\Mod_{\ul{\Span}(\calQ)}\left(\Cat(T)_{\calQ\textup{-co}}\right) \to \Cat(T)_{\calQ\textup{-co}}
	\] 
	is fully faithful, with essential image given by $\Cat(T)_{\calQ\textup{-sa}}$. In particular, $\ul{\Span}(\calQ)$ is an idempotent algebra in $\Cat(T)_{\calQ\textup{-co}}$ and a $\calQ$-cocomplete $T$-category is a module over $\ul{\Span}(\calQ)$ if and only if it is $\calQ$-semiadditive.
\end{theorem}

\begin{proof}
	This proof is modelled on that of \cite[Cor.5.3]{harpazambi}. First, note that the forgetful functor admits a left adjoint, given by $\ul{\Span}(\calQ)\otimes -$. We will first show that the unit $\Cc \to \ul{\Span}(\calQ)\otimes \Cc$ is an equivalence if and only if $\Cc$ is $\calQ$-semiadditive. If the unit is an equivalence then $\Cc$ is tensored over $\ul{\Span}(\calQ)$, and so $\calQ$-semiadditive by \cite[Cor.6.4]{CLLSpans}. For the other direction, suppose $\Cc$ is $\calQ$-semiadditive. Then by Theorem 6.5 of \emph{op.~cit.}, $\ul{\Span}(\calQ)\otimes \calC$ is again $\calQ$-semiadditive. So it suffices to show that the induced map
	\[
	\Fun^{\calQ\textup{-co}}(\ul{\Span}(\calQ)\otimes \Cc, \calD) \to \Fun^{\calQ\textup{-co}}(\Cc,\calD)
	\]
	is an equivalence for all $\calQ$-semiadditive $\calD$. However, by currying, this functor is equivalent to the functor 
	\[
	\Fun^{\calQ\textup{-co}}(\ul{\Span}(\calQ),\Fun^{\calQ\textup{-co}}(\Cc,\calD)) \to \Fun^{\calQ\textup{-co}}(\Cc,\calD)
	\]
	induced by restricting along the inclusion $\ast\to \ul{\Span}(\calQ)$. Because $\Fun^{\calQ\textup{-co}}(\Cc,\calD)$ again admits an action by the $\calQ$-semiadditive category $\Cc$, it is $\calQ$-semadditive by Theorem 6.5 of \emph{op.~cit.} again. Therefore, the map above is an equivalence by Theorem 5.1 of \emph{op.~cit}.
	
	Now we will show that the counit is an equivalence for all $\ul{\Span}(\calQ)$-modules $\calC$. Note that by the previous argument, such a $\calC$ is $\calQ$-semiadditive. Because the forgetful functor is conservative, this follows from the fact that the composite
	\[
	\calC \to \ul{\Span}(\calQ)\otimes \calC \to \calC
	\]
	is homotopic to the identity of $\calC$ by the triangle identity, and the fact that the first map is an equivalence by the previous part. These two observations combine to prove the first statement. The final statement follows from the theory of idempotent algebras; see \cite[\textsection4.8]{ha}.
\end{proof}

\begin{cor}
	$\Cat(T)_{\calQ\textup{-sa}}$ is a symmetric monoidal category with unit $\ul{\Span}(\calQ)$.
\end{cor}

\subsection{Unstraightening symmetric monoidal parametrized categories}\label{ssec:unstraighteningparametrized}

Suppose $\Cc^\otimes\colon \PshT^{\op}\to \Cat^\otimes$ is a symmetric monoidal $T$-category. Then given two objects $X\in \Cc(A)$ and $Y\in \Cc(B)$ we can define the \emph{external tensor product} of $X$ and $Y$ by the formula:
\[
X\boxtimes Y \coloneqq \pi_1^*(X)\otimes \pi_2^*(Y) \in \Cc(A\times B).
\]
In this subsection, we will explain one way to show that the external tensor product defines a symmetric monoidal structure on $\smallint \Cc$, the coCartesian unstraightening of $\Cc^\otimes$. Our method lends itself to explicit computation, which we apply to the specific example of $\ul{\Span}(\calQ)$. We begin with a key result due to Bachmann--Hoyois, for which we need to establish the following notations.

\begin{mydef}
Given a category $S$, we denote by $\Fin(S)$ the finite coproduct completion of $S$. It is easy to see that $\Fin(S)$ is equivalent to the Cartesian unstraightening of the functor 
\[
\Fun(-,S)\colon \Fin^{\op}\to \Cat, \quad I\mapsto \Fun(I,S).
\]
A map in $\Fin(S)$ is a \emph{fold map} if it is Cartesian. This defines a pullback stable subcategory $\Fin(S)_{\nabla}\subset \Fin(S)$.
\end{mydef}

Since $\Fin(S)$ admits finite coproducts, the inclusion $S\hookrightarrow \Fin(S)$ uniquely extends to a functor 
\[
\iota\colon \Fin\times S\to \Fin(S), \quad (I,X)\mapsto \coprod_I X.
\]
Equipping the source and target of the functor $\iota$ with the adequate triple structures (in the sense of \cite{HHLNa}) 
\[
(\Fin\times S,\Fin\times S,\Fin\times S^\simeq) \quad \text{and}\quad (\Fin(S),\Fin(S),\Fin(S)_{\nabla}),
\]
we obtain a functor 
\[
\Span(\iota)\colon \Span(\Fin)\times S^{\op}\to \Span_{\nabla}(S).
\] 

\begin{prop}[{\cite[Pr.C.4]{normsmotivic}}]\label{prop:external_tensor}
    Let $S$ be an arbitrary category. Then the restriction along the functor 
    \[
    \Span(\iota)^*\colon \Fun(\Span_{\nabla}(\Fin(S)), \Cat) \to \Fun(\Span(\Fin)\times S^{\op}, \Cat)
    \]
    induces an equivalence between the following two full subcategories:
    \begin{enumerate}
        \item The full subcategory of product preserving functors $\Span_{\nabla}(S)\to \Cat$;
        \item The full subcategory of functors $\Span(\Fin)\times S^\op \to \Cat$ which preserve products in the first variable, itself equivalent to $\Fun(S^{\op}, \Cat^{\otimes})$.
    \end{enumerate}
\end{prop}

With these recollections handled we can give the main definition of this section. 

\begin{mydef}
Given a symmetric monoidal $T$-category $\Cc^\otimes$, viewed as a product preserving functor $\Cc^{\otimes}\colon \Span_{\nabla}(\Fin(\PshT))\to \Cat$ via the equivalence above, we define $\smallint \Cc^\otimes\to \Span(\Fin)$ to be the composite
\[
\smallint\Cc^{\otimes}\to \Span_{\nabla}(\Fin(\PshT))\to \Span(\Fin).
\]
This defines a functor 
\begin{equation}\label{eq:unstr_sym_mon_ext}
\smallint(\bs)\colon \Cat(T)^\otimes \to \Cat_{/\Span(\Fin)}.
\end{equation}
\end{mydef}

In the following proposition, we identify $\Cat^\otimes$ with a subcategory of $\Cat_{/\Span(\Fin)}$ by means of unstraightening and the inclusion $\Cat^\otimes \simeq \CMon(\Cat) \subset \Fun(\Span(\Fin),\Cat)$.

\begin{prop}\label{pr:factors_through_catotimes}
The functor (\ref{eq:unstr_sym_mon_ext}) above factors through $\Cat^\otimes$.
\end{prop} 

\begin{proof}
By \cite[Cor.2.49]{normsequiv}, a coCartesian fibration $p\colon \Cc\to \calD$ is the unstraightening of a product preserving functor if and only if the following two conditions are satisfied:
\begin{enumerate}
    \item $\Cc$ admits products preserved by $p$;
    \item The projection $X\times Y \to X$ is a coCartesian edge for all $X,Y\in \ccat$.
\end{enumerate}
Since being a coCartesian fibration and the conditions above are both closed under composition, to prove that $\smallint \Cc^\otimes\to \Span(\Fin)$ is the unstraightening of a product preserving functor, it suffices to prove this separately for the functors $\smallint \Cc^\otimes\to \Span_{\nabla}(\Fin(\calP))$ and $\Span_{\nabla}(\Fin(\calP))\to \Span(\Fin)$. The former satisfies these properties by the equivalence above, as it is by definition the unstraightening of a product preserving functor. On the other hand, the map $\Span_{\nabla}(\Fin(\calP))\to \Span(\Fin)$ is also the coCartesian unstraightening of a product preserving functor by \cite[Pr.3.1.2]{CLLR24Bispan}.
\end{proof}

\begin{remark}
One easily computes that the fiber of $\smallint \Cc^\otimes$ over $1\in \Span(\Fin)$ is precisely $\smallint \Cc$, and so the previous proposition supplies $\smallint \Cc$ with a symmetric monoidal structure. More explicitly, consider $(X,A)$ and $(Y,B)$, two objects in $\smallint \Cc$. The coCartesian pushforward of these two objects along $\ast\amalg \ast \to \ast$ is given by first computing the coCartesian edge in $\Span_{\nabla}(\Fin(\PshT))\to \Span(\Fin)$ starting at the coproduct of $A$ and $B$. The tensor product in $\smallint \Cc^\otimes$ is then given by the coCartesian pushforward along this map. However, the former is easily seen to be the span
\[
A\amalg B\xleftarrow{\pi_A\amalg \pi_B} A\times B\amalg A\times B\xrightarrow{\nabla} A\times B.
\]
The coCartesian pushforward of this map gives $\pi_A^*X \otimes \pi_B^* Y$, as desired. Similarly, one computes the unit of $\smallint\calC^\otimes$ to be the object $(\bbone_{\calC(\ast)},\ast)$. 
\end{remark}

\begin{remark}
The previous proposition only used that $\PshT$ admits finite products, and so \cref{pr:factors_through_catotimes} holds in that generality. In particular, given any category $S$ with finite products and a functor $F\colon S\to \Cat^\otimes$, the unstraightening $\smallint F$ of the underlying functor $S\to \Cat$ admits a canonical symmetric monoidal structure given by the external tensor product. 
\end{remark}

\begin{prop}\label{pr:symmetric_monoidal_section}
Let $\Cc^\otimes$ be a symmetric monoidal $T$-category. Then there exists, naturally in $\Cc$, a fully faithful symmetric monoidal inclusion $\Cc(\ast)^\otimes\to \smallint\Cc^\otimes$.
\end{prop}

\begin{proof}
Consider the functor $\Span(\Fin)\to \Span_{\nabla}(\Fin(\PshT))$ induced by the inclusion $\ast\to \PshT$ of the terminal object of $T$. This induces the following commutative diagram
\[\begin{tikzcd}
	{\Cc(\ast)^\otimes} & {\smallint\Cc^\otimes} \\
	{\Span(\Fin)} & {\Span_{\nabla}(\Fin(\PshT))} \\
	& {\Span(\Fin)}
	\arrow[from=1-1, to=1-2]
	\arrow[from=1-1, to=2-1]
	\arrow["\lrcorner"{anchor=center, pos=0.125}, draw=none, from=1-1, to=2-2]
	\arrow[from=1-2, to=2-2]
	\arrow[from=2-1, to=2-2]
	\arrow[equals, from=2-1, to=3-2]
	\arrow[from=2-2, to=3-2]
\end{tikzcd}\]
One can easily see that the resulting functor 
\[
\Cc(\ast)^\otimes \to \smallint\Cc^\otimes
\]
over $\Span(\Fin)$ preserves coCartesian morphisms. Moreover, since the only endomorphism of the terminal object $\ast$ is the identity, we also immediately obtain that the functor is fully faithful.
\end{proof}

Having discussed the general construction of a symmetric monoidal structure on the coCartesian unstraightening of a symmetric monoidal $T$-category, for the remainder of this section, our goal will be to explicitly compute $\smallint \ul{\Span}(\calQ)^\otimes$, the unstraightening of the symmetric monoidal $T$-category $\ul{\Span}(\calQ)^\otimes$.

\begin{mydef}
Let $\Ar_{\calQ}(\calP)$ be the full subcategory of the arrow category $\Ar(\calP) = \Fun([1],\calP)$ spanned by maps in $\calQ$ and view this as a symmetric monoidal category via the Cartesian product. A map $\eta$ in $\Ar_{\calQ}(\calP)$ is a \emph{fiberwise} map if $t(\eta)$ is invertible, where $t\colon \Ar(\calP)\to \calP$ denotes evaluation at the target of $[1]$. The fiberwise maps are closed under products and pullbacks in $\Ar_{\calQ}(\calP)$, hence we obtain the symmetric monoidal category $\Span_{\mathrm{fw}}(\Ar_{\calQ}(\calP))^\otimes$, where the forward maps are the fiberwise maps.
\end{mydef}

We will prove the following theorem:

\begin{theorem}\label{thm:unstr_of_span}
There is an equivalence
\[
\smallint\ul{\Span}(\calQ)^\otimes\simeq \Span_{\mathrm{fw}}(\Ar_{\calQ}(\calP))^\otimes 
\]
of symmetric monoidal categories.
\end{theorem}

Consider the functor $\calQ_{/\bullet}\colon \calP^{\op}\to \Cat$, and observe that it factors through the category $\Cat^\times \subset \Cat$ of categories with finite products and finite product preserving functors. Applying the fully faithful inclusion $\Cat^\otimes\subset \Cat^\times$, we obtain a symmetric monoidal $T$-category, which we identify with a product preserving functor 
\begin{equation}\label{eq:Q_times}
\calQ_{/\bullet}^{\times}\colon \Span_{\nabla}(\Fin(\calP))\to \Cat
\end{equation} via \Cref{prop:external_tensor}. We separate out one observation from the proof:

\begin{lemma}
The Cartesian unstraightening of the functor (\ref{eq:Q_times}) is given by
\[
\Span(\Fin(\ev_1))\colon\Span_{\nabla,\all}(\Fin(\Ar_{\calQ}(\calP))) \to \Span_{\nabla,\all}(\Fin(\calP))\simeq \Span_{\nabla}(\Fin(\calP))^{\op}.
\]
We denote this functor by $\Phi$.
\end{lemma}

\begin{proof}
We first note that $\calQ_{/\bullet}^\times$ is determined by the following three properties:
    \begin{enumerate}
        \item $\calQ_{/\bullet}^\times$ is product preserving,
        \item The restriction of $\calQ_{/\bullet}^\times$ along the functor 
        \[
        i_X\colon \Span(\Fin)\to \Span_{\nabla}(\Fin(\calP)),\quad  I\mapsto \coprod_I X
        \]
        is equivalent to $\Cc(X)^\times$ for all $X\in \calP$,
        \item The restriction of $\calQ_{/\bullet}^\times$ along the functor $\calP^{\op}\to \Span_{\nabla}(\Fin(\calP))$
        is equivalent to $\calQ_{/\bullet}$.
    \end{enumerate}
    Namely, property (1) implies that $\calQ_{\bullet}^\times$ is determined by its restriction to $\Fun(\calP^{\op},\Cat^\otimes)$ by \Cref{prop:external_tensor}. However, property (2) implies that this restriction factors through the fully faithful inclusion $\Cat^\times\subset \Cat^{\otimes}$, and so it suffices to identify the underlying functor $\Fun^{\times}(\calP^{\op},\Cat)$ with $\calQ_{/\bullet}$. This is precisely property (3).
    
    Now consider the functor $\Phi$ in question. Recall that $\Ar_{\calQ}(\calP)\to \calP$ is the Cartesian unstraightening of the functor $\calQ_{/\bullet}\colon \calP^{\op}\to \Cat$. A simple calculation shows that $\Fin(\Ar_{\calQ}(T))\to \Fin(T)$ is again a Cartesian fibration. Then an application of Barwick's unfurling \cite[Th.3.1]{HHLNa}, using the existence of products in $\calP$, shows that $\Phi$ is a coCartesian fibration. We write $\phi\colon \Span_{\nabla}(\Fin(T))\to \Cat$ for the corresponding functor. We will show that it satisfies the three properties characterizing $\calQ_{/\bullet}^\times$, completing the proof.
    \begin{enumerate}
        \item Both the source and the target of $\Phi$ admit coproducts, which are preserved by $\Phi$. Moreover, the summand inclusions are easily seen to be Cartesian morphisms. So by the dual of \cite[Cor.2.49]{normsequiv}, the straightening of $\Phi$ is product preserving.
        \item Since taking spans commutes with taking pullbacks, we can explicitly compute the unstraightening of the restriction of $\phi$ along $i_X$. It is easily seen to be equivalent to the description of the Cartesian unstraightening of $\Cc(X)^\times$ given in \cite[Pr.3.1.2]{CLLR24Bispan}.
        \item Similarly the unstraightening of the restriction of $\phi$ along $T^{\op}\to \Span_{\nabla}(T)$ can be computed by pullback, where it is given by $\Ar_{\calQ}(T)\to T$. This is the Cartesian unstraightening of $\calQ_{/\bullet}$ by definition.\qedhere
    \end{enumerate} 
\end{proof}

\begin{proof}[Proof of \cref{thm:unstr_of_span}]
The proof proceeds by explicitly computing both sides as coCartesian fibrations over $\Span(\Fin)$. We begin with the left-hand side.
\begin{enumerate}
    \item For the first step, we recall that the Cartesian unstraightening of $\calQ^\times_{/\bullet}$ is given by the functor 
    \[
    \Phi\colon \calD\coloneqq \Span_{\nabla,\all}(\Fin(\Ar_{\calQ}(\PshT))) \to    \Span_{\nabla,\all}(\PshT)\simeq \Span_{\nabla}(\Fin(\PshT))^{\op}.
    \]
    Write $\calD_{\mathrm{fw}}$ for the wide subcategory of $\calD$ spanned by the \emph{fiberwise} maps, those which live over equivalences in $\Span_{\nabla}(\PshT)^{\op}$.
    \item Then we can compute the coCartesian unstraightening of the composite
    \[
    \ul{\Span}(\calQ_{/\bullet}^\times) = {\Span}\circ {\calQ_{/\bullet}}\colon \Span_{\nabla}(\Fin(\PshT))\to \Cat
    \]
    to be the map
    \[
    \Span_{\all,\mathrm{fw}}(\calD)\to \Span_{\all,\simeq}(\Span_{\nabla}(\Fin(\PshT))^{\op}) \simeq \Span_{\nabla}(\Fin(\PshT))
    \]
    by \cite[Th.3.9]{HHLNa}.
    \item Finally, we conclude that the unstraightening of $\smallint \ul{\Span}(\calQ)^\otimes$ is precisely given by 
    \begin{align*}
    \Span_{\all,\mathrm{fw}}(\Span_{\nabla,\all}(\Fin(\Ar_{\calQ}(\PshT))))&\to \Span_{\all,\simeq}(\Span_{\nabla,\all}(\Fin(\PshT))) \\
    &\simeq \Span_{\nabla}(\Fin(\PshT)).
    \end{align*}
    Taking the composite with the map $\Span_{\nabla}(\Fin(\PshT))\to\Span(\Fin)$ we obtain the functor
    \[
    \Span_{\all,\mathrm{fw}}(\Span_{\nabla\textup{-pb},\all}(\Fin(\Ar_{\calQ}(\PshT))))\to \Span(\Fin)
    \]
    induced by the unique functor $\Ar_{\calQ}(\PshT)\to \ast$.
\end{enumerate}
We now explicitly compute the right-hand side.
\begin{enumerate}
    \item By an application of \cite[Pr.3.1.2]{CLLR24Bispan}, we may compute the Cartesian unstraightening of $\Ar_{\calQ}(\PshT)^{\times}$ to be 
    \[
    \Span_{\mathrm{cart},\all}(\Fin(\Ar_{\calQ}(\PshT)))\to \Span(\Fin),
    \]
    where $\mathrm{cart}$ denotes the Cartesian morphisms for the functor $\Fin(\Ar_{\calQ}(\PshT))\to \Fin$. However, these are precisely the fold maps, and so we may equivalently write this as 
    \[
    \Span_{\nabla,\all}(\Fin(\Ar_{\calQ}(\PshT)))\to \Span(\Fin),
    \]
    \item By another application of \cite[Th.3.9]{HHLNa}, we may compute the unstraightening of $\Span_{\mathrm{fw}}(\Ar_{\calQ})^\otimes = {\Span}\circ {\Ar_{\cal}(\PshT)^\times}$ as
    \[
    \Span_{\all,\mathrm{fw}}(\Span_{\nabla,\all}(\Fin(\Ar_{\calQ}(\PshT))))\to \Span_{\all,\simeq}(\Span(\Fin))\simeq \Span(\Fin)
    \]
    induced by the functor $\Ar_{\calQ}(\PshT)\to \ast$, as required.\qedhere
\end{enumerate}
\end{proof}

Let us also record for later use the following result:

\begin{prop}\label{prop:arrow_localization}
The functor $\ev_0\colon \Ar_{\calQ}(\PshT)\to \PshT$ induces a symmetric monoidal Dwyer--Kan localization 
\[
\Span_{\mathrm{fw}}(\Ar_{\calQ}(\PshT))^\otimes \to \Span_{\calQ}(\PshT)^\otimes
\]
at the morphisms of the form
\[\begin{tikzcd}
	X & X & X \\
	Z & Y & Y
	\arrow["fq"', from=1-1, to=2-1]
	\arrow[equals, from=1-2, to=1-1]
	\arrow[equals, from=1-2, to=1-3]
	\arrow["q", from=1-2, to=2-2]
	\arrow["q", from=1-3, to=2-3]
	\arrow["f", from=2-2, to=2-1]
	\arrow[equals, from=2-2, to=2-3]
\end{tikzcd}\]
where $q$ and $fq$ are maps in $\Qloc$.
\end{prop}

\begin{proof}
The functor $\ev_0\colon \Ar_{\calQ}(\PshT)\to \PshT$ clearly preserves products and so induces a symmetric monoidal structure $\Ar_{\calQ}(\PshT)^\times\to \PshT^\times$. Applying $\Span(\bs)$ to this gives a symmetric monoidal functor 
\[
\Span(\ev_0)\colon \Span_{\mathrm{fw}}(\Ar_{\calQ}(\PshT))^\otimes \to \Span_{\calQ}(\PshT)^\otimes.
\]
By \cite[Pr.8.21]{CLLSpans}, the underlying functor is a Dwyer--Kan localization, and so we conclude that $\Span(\ev_0)$ is a symmetric monoidal Dwyer--Kan localization. This functor precisely inverts the maps in the proposition, and so the final claim is then immediate.
\end{proof}

\begin{remark}
Passing to underlying categories, \Cref{thm:unstr_of_span} shows that the map 
\[
\Span(\ev_0)\colon \Span_{\mathrm{fw}}(\Ar_{\calQ}(\PshT))\to \Span(\PshT,\PshT,\PshT^\simeq) \simeq \PshT^{\op}
\]
is the coCartesian unstraightening of the functor $\ul{\Span}(\calQ)\colon \PshT^{\op}\to \Spc$. This was already observed in \cite{CLLSpans}, as was the previous proposition after restricting to underlying categories. 

Since $\ul{\Span}(\calQ)$ admits $\calQ$-limits, we know that the functor above is also a Cartesian fibration. In fact, the Cartesian edge lying over $f\colon Y\to Z$ with target $q\colon X\to Y$ is precisely 
\[\begin{tikzcd}
	X & X & X \\
	Z & Y & Y,
	\arrow["fq"', from=1-1, to=2-1]
	\arrow[equals, from=1-2, to=1-1]
	\arrow[equals, from=1-2, to=1-3]
	\arrow["q", from=1-2, to=2-2]
	\arrow["q", from=1-3, to=2-3]
	\arrow["f", from=2-2, to=2-1]
	\arrow[equals, from=2-2, to=2-3]
\end{tikzcd}\]
as a simple calculation shows. So the previous proposition proves that $\Span(\ev_0)$ is a Dwyer-Kan localization at the Cartesian morphisms.
\end{remark}

\subsection{Global sections of parametrized stable categories}\label{ssec:parametrizedglobalsections}

Let $\Cc$ be a stable symmetric monoidal category. Recall from \cite{stableoperad} that there exists an exact lax monoidal functor
\[
\hom_{\Cc}(\bbone_\Cc,-)\colon \Cc \to \Sp,
\]
which is characterized by being the initial object of $\Fun^{\otimes\textup{-lax}}_{\mathrm{lex}}(\Cc,\Sp)$. We will call this \emph{the global sections functor} of $\Cc$, and denote it by $\Gamma\colon \Cc\to \Sp$. By the universality of global sections, it is laxly natural in symmetric monoidal functors of stable categories. This encodes the existence of a natural transformation \[\hom_{\calC}(\bbone_{\calC},-) \Rightarrow \hom_{\calD}(\bbone_{\calD},F(\bs))\]
for every functor $F\colon \calC\to \calD$ of stable categories. For our decategorification procedure, we require an analog of the global sections functor for families of stable categories. Namely, given a stable symmetric monoidal $T$-category $\Cc$, we will construct a lax symmetric monoidal functor $\Gamma\colon \smallint \Cc^\otimes \to \Sp^\otimes$ equivalent to the global sections functor $\Gamma\colon \Cc(X)^\otimes \to \Sp^\otimes$ of Nikolaus when restricted to $\Cc(X)^\otimes$. Again, we require this to also be \emph{laxly} natural in $\Cc$, meaning that there exists a coherent choice of symmetric monoidal natural transformation filling the triangle
\begin{center}\begin{tikzcd}[row sep=tiny,column sep=tiny]
\smallint\ccat^\otimes\ar[rrrr,"\smallint F"]\ar[ddrr,"\Gamma"']&&&&\smallint\dcat^\otimes\ar[ddll,"\Gamma"]\\
&{}\ar[rr,Rightarrow,shift left=1mm]&&{}\\
{}&&\Sp^\otimes
\end{tikzcd}\end{center}
for any functor $F\colon \Cc\to \calD$ of stable symmetric monoidal $T$-categories. For a rigorous definition of lax transformations, we refer the reader to \cite[\textsection2]{gepner2024global2ringsgenuinerefinements} or \cite{laxtrans}. In the following proposition, we denote by $\Cat^{\otimes\textup{-lax}}$ the category of symmetric monoidal categories and \emph{lax} symmetric monoidal functors.

\begin{prop}\label{prop:global_sections}
There exists a lax natural transformation $\Gamma\colon \smallint(\bs)\Rightarrow \mathrm{const}(\Sp^\otimes)$ of functors $\Cat(T)_{\st}^\otimes\to \Cat^{\otimes\textup{-lax}}$ sending a stable symmetric monoidal $T$-category $\Cc^\otimes$ to a functor
\[
\Gamma\colon \smallint \Cc^\otimes\to \Sp^\otimes
\]
which gives the global sections functor $\Gamma\colon \Cc(X)^\otimes \to \Sp^\otimes$ of the stable category $\Cc(X)$ when restricted to the fiber over $X\in \PshT$.
\end{prop}

\begin{remark}
The functor $\Gamma\colon \smallint \Cc^\otimes \to \Sp^\otimes$ encodes the lax symmetric monoidality of global sections of stable categories under the external tensor product. Given two objects $(A,X), (B,Y)\in \smallint\Cc$ taking ``fiberwise global sections'' is lax symmetric monoidal via the comparison map 
\[\begin{tikzcd}[column sep=0mm]
	{\Gamma(A,X) \otimes \Gamma(B,Y)} & {\hom_{\Cc_A}(\bbone,X)\otimes \hom_{\Cc_B}(\bbone,X)} \\
	& {\hom_{\Cc_{A\times B}}(\bbone,\pi_1^*X)\otimes \hom_{\Cc_{A\times B}}(\bbone,\pi_2^*Y)} \\
	& {\hom_{\Cc_{A\times B}}(\bbone,\pi_1^*X\otimes \pi_2^* Y)} & {\Gamma((A,X)\boxtimes (B,Y)).}
	\arrow[equals, from=1-1, to=1-2]
	\arrow["{\pi_1^*\otimes \pi_2^*}", from=1-2, to=2-2]
	\arrow["-\otimes-", from=2-2, to=3-2]
	\arrow[equals, from=3-2, to=3-3]
\end{tikzcd}\]
We expect, but do not prove, that the global sections functor $\Gamma\colon \smallint \Cc^\otimes \to \Sp^\otimes$ admits a universal characterization analogous to the global sections functor of a stable category. 
\end{remark}

\begin{proof}[Proof of \cref{prop:global_sections}]
Let $\Cc$ be a fiberwise stable symmetric monoidal $T$-category and consider the unique symmetric monoidal exact functor $\Sp^\omega\to \Cc(\ast)$ which sends the sphere spectrum to the unit of $\Cc(\ast)$. Taking the composite with the symmetric monoidal inclusion $\Cc(\ast)\to \smallint \Cc$ of \Cref{pr:symmetric_monoidal_section}, we obtain a functor $F_{\Cc}\colon \Sp^\omega \to \smallint\Cc$, which is natural in fiberwise stable symmetric monoidal $T$-categories. Now consider the composite
\[
\smallint\Cc\xrightarrow{Y} \Psh(\smallint\Cc)\xrightarrow{F^*_\Cc} \Psh(\Sp^\omega),
\]
where both presheaf categories are endowed with the Day convolution symmetric monoidal structure. The Yoneda embedding is lax symmetric monoidal by \cite[Cor.4.8.1.12]{ha}, while $F^*_{\Cc}$ is by the universal property of Day convolution. Therefore, as a composite of lax symmetric monoidal functors, this composite is lax symmetric monoidal. Moreover, a simple calculation shows that this factors through $\Ind(\Sp^\omega)$, which is equivalent to those functors $(\Sp^\omega)^{\op}\to \Spc$ which preserve finite limits. The symmetric monoidal equivalence
$\Ind(\Sp^\omega) \simeq \Sp$ then yields the required lax symmetric monoidal functor. 

To explain how this is also laxly natural in the fiberwise stable symmetric monoidal $T$-category $\Cc$, we first note that $F_\calC^*\colon \Psh(\smallint\calC) \to \Psh(\Sp^\omega)$ is clearly \emph{strictly} natural in restriction along symmetric monoidal functors of stable symmetric monoidal $T$-categories. By passing to left adjoints in the horizontal direction, we obtain a lax \emph{mate} transformation
\begin{center}\begin{tikzcd}[row sep=tiny,column sep=tiny]
\Psh(\smallint \Cc)\ar[rrrr,"(\smallint F)_!"]\ar[ddrr,"F_\ccat^\ast"']&&&&\Psh(\smallint\dcat^\otimes)\ar[ddll,"F_\dcat^\ast"]\\
&{}\ar[rr,Rightarrow,shift left=1mm]&&{}\\
{}&&\Psh(\Sp)
\end{tikzcd}\end{center}
see \cite[Th.E]{laxtrans}. On the other hand, the Yoneda embedding is a symmetric monoidal natural transformation by \cite[Cor.2.47]{benmosheschlank}. Composing these gives the required lax natural transformation of lax symmetric monoidal functors.
\end{proof}

We will also record the following variant of this construction, which encodes a form of global sections for $\Cc$ relative to $\Cc(\ast)$.

\begin{mydef}
Let $\Cat(T)^\otimes_{\textup{R-lax}}$ be the \emph{full} subcategory of $\Cat(T)$ spanned by those symmetric monoidal $T$-categories $\Cc$ such that $f^*\colon \Cc(B)\to \Cc(A)$ admits a right adjoint for all maps $f\colon A\to B$ in $\PshT$. Let us emphasize that we do not assume that the functors in $\Cat(T)^\otimes_{\textup{R-lax}}$ commute with these limits.
\end{mydef}

\begin{prop}\label{prop:relative_global_sections}
There exists a lax natural transformation $\smallint(\bs)\Rightarrow \ev_\ast$ of functors $\Cat(T)^\otimes_{\textup{R-lax}}\to \Cat^{\otimes\textup{-lax}}$ which sends a symmetric monoidal $T$-category $\Cc^\otimes$ in $\Cat(T)^\otimes_{\textup{R-lax}}$ to a pushforward functor
\[
\smallint \Cc^\otimes\to \Cc(\ast)^\otimes
\]
characterized by the fact that it restricts on the fiber over every $X\in \PshT$ to the right adjoint $\pi_*$ of restriction induced by the unique map $\pi\colon X\to \ast$.
\end{prop}

\begin{proof}
Consider the natural symmetric monoidal inclusion $\Cc(\ast)\subset \smallint\Cc.$ We claim that pointwise this has a right adjoint, which then automatically combine to form a lax natural transformation of functors $\Cat(\Glo_\ab)^\otimes_{\textup{L-lax}}\to \Cat^{\otimes,\mathrm{lax}}$ by \cite[Th.5.3.6]{laxtrans} applied to the 2-category $\Cat^{\otimes,\mathrm{lax}}$. To prove this, we note that since there is a unique map $
\pi\colon A\to \ast$, there is an equivalence
\[
\Map_{\smallint \Cc}((\ast,X),(A,Y)) \simeq \Map_{\Cc(A)}(\pi^* X,Y) \simeq \Map_{\Cc(\ast)}(X,\pi_\ast Y)
\]
of mapping spaces. This proves that the right adjoint exists and has the form claimed by the proposition.
\end{proof}

\subsection{Decategorification}\label{ssec:decategoricification}
Finally, we arrive at the main theorem of this section, which states that for all $\Cc\in \Cat(T)^\otimes_{\calQ\textup{-st}}$, the endomorphisms of the unit $\End_{\Cc}(\bbone)\colon \PshT^{\op}\to \Sp$ canonically enhance to a lax symmetric monoidal functor 
\[
\End_{\Cc}(\bbone)\colon \Span_{\calQ}(\calP)\to \Sp.
\]

\begin{theorem}\label{thm:generaldecategorification}
	There is a limit preserving functor
	\[
	\End_{(\bs)}(\bbone)\colon \Cat(T)^\otimes_{\calQ\textup{-st}}\to \CAlg(\Sp_\calQ), \qquad \End_\ccat(\bbone)(X) = \End_{\ccat(X)}(\bbone),
	\]
    endowing the endomorphism spectra $\End_{\ccat(X)}(\bbone)$ of the units of $\ccat \in \Cat(T)^\otimes_{\calQ\textup{-st}}$ with transfers for maps which are locally in $\calQ$.
\end{theorem}

\begin{proof}
	Let us first indicate the strategy on the level of objects. Let $\calC^\otimes$ be a symmetric monoidal $\calQ$-stable $T$-category. 
	\begin{enumerate}
		\item Because $\Cc$ is, in particular, a $\calQ$-semiadditive symmetric monoidal $T$-category, there is a unique symmetric monoidal functor $F\colon \ul{\Span}(\calQ)^\otimes \to \calC^\otimes$ of $T$-categories which sends $1\in \Gamma\ul{\Span}(\calQ)^\otimes$ to $\bbone_\calC \in \Gamma\calC$. 
	\item Passing to coCartesian unstraightenings, we obtain a symmetric monoidal functor 
        \[
        \smallint F\colon \smallint \Span(\calQ)^\otimes \to \smallint \Cc^\otimes
        \]
        on the associated unstraightenings equipped with the external tensor product.
        \item On the other hand, because $\calC^\otimes$ is a fiberwise stable $T$-category, by \cref{prop:global_sections}, there exists a lax symmetric monoidal functor
	\[
        \smallint \Cc^\otimes \to \Sp^\otimes.
        \]
		\item Postcomposing these two functors, we obtain a lax symmetric monoidal functor 
		\[
		\End_{\calC}(\bbone)\colon \smallint \Span(\calQ)^\otimes \to \smallint \Sp^\otimes
		\]
		More explicitely, using \cref{thm:unstr_of_span}, we obtain a lax symmetric monoidal functor
		\[ \Span_{\mathrm{fw}}(\Ar(\PshT)^{\calQ})\to \Sp.
		\]
		\item One now computes that this functor inverts the maps from \cref{prop:arrow_localization}. By \cite[Pr.A.5(v)]{nikolausscholze}, we therefore obtain a lax symmetric monoidal functor 
		\[
		\End_{\calC}(\bbone)\colon \Span_{\calQ}(\PshT) \to \Sp.
		\]
        Namely, $F$ is by definition a $\calQ$-colimit preserving functor, and so also $\calQ$-limit preserving by \Cref{lem:QexactFunctor}. This means that $\smallint F$ preserves Cartesian morphisms, and so the maps in \cref{prop:arrow_localization} are sent to Cartesian morphisms of $\smallint \Cc$. To conclude, we can show that $\Gamma\colon \smallint \Cc\to \Sp$ sends Cartesian morphisms to equivalences. Given a map $f\colon X\to Y$ and an object $A\in \Cc_X$, the image of the Cartesian edge $(A,X)\to (f_*A,Y)$ is the map
        \[
        \Gamma(f_* A,Y) \coloneqq \hom_{\Cc_Y}(\bbone,f_*A) \simeq \hom_{\Cc_X}(f^*\bbone,A) \simeq \hom_{\Cc_Y}(\bbone,A) =: \Gamma(A,X),
        \]
        and so an equivalence.
		\item Finally, we compute that the restriction of this functor to $\PshT^{\op}$ is limit preserving, and so we have constructed a $\calQ$-ambidextrous $\E_\infty$ ring. Let $X = \colim_i X_i$ be a colimit in $\PshT$. Then this induces an equivalence $\Cc_X \simeq \lim_i\Cc_{X_i}$, since $\Cc$ is limit preserving, which in turn induces an equivalence 
        \[
        \End_{\calC}(\bbone)(X) = \hom_{\Cc_X}(\bbone,\bbone) = \lim_i \hom_{\Cc_{X_i}}(\bbone,\bbone) \simeq \lim_i \End_{\calC}(\bbone)(X_i),
        \]
        proving that the restriction of $\End_{\calC}(\bbone)$ to $\PshT^{\op}$ is limit preserving. 
	\end{enumerate}
	To make this functorial in $\Cat(T)^\otimes_{\calQ\textup{-st}}$ it suffices to make all of the construction above natural in $\Cat(T)^\otimes_{\calQ\textup{-st}}$, as we have gone to lengths to do. Let us spell this out below. To begin, we consider the composite of functors:
\[\begin{tikzcd}[column sep=small]
	{\Cat(T)^\otimes_{\calQ\textup{-st}}} & {\Cat^{\otimes\textup{-lax}}_{\Span_{\mathrm{fw}}(\Ar_{\calQ}(\PshT))/}\times_{\Cat^{\otimes\textup{-lax}}}\Cat^{\otimes\textup{-lax}}_{/\Sp}} & {\Fun^{\otimes\textup{-lax}}(\Span_{\mathrm{fw}}(\Ar_{\calQ}(\PshT)),\Sp),}
	\arrow["{(\smallint,\Gamma)}", from=1-1, to=1-2]
	\arrow["{-\circ-}", from=1-2, to=1-3]
\end{tikzcd}\]
    where the second functor is given by composition. This lands in the fiber of the functor \[(s,t)\colon \Ar^{\mathrm{opl}}(\Cat^{\otimes\textup{-lax}})\to \Cat^{\otimes\textup{-lax}}\times\Cat^{\otimes\textup{-lax}}\] over the pair $\{\Span_{\mathrm{fw}}(\Ar_{\calQ}(\PshT)),\Sp\}$, which is equivalent to $\Fun^{\otimes\textup{-lax}}(\Span(\Ar_\calQ(\PshT),\Sp)$ by \cite[Th.7.21]{HHLNa}. We then observe factors through the subcategory of functors which inverts the maps $W$ from \Cref{prop:arrow_localization}, and so we obtain the composite
\[\begin{tikzcd}
	{\Cat(\Glo_\ab)^\otimes_{\calQ\textup{-st}}} & {\Fun^{\otimes\textup{-lax}, W}(\Span_{\mathrm{fw}}(\Ar_{\calQ}(\PshT)),\Sp)} & {\Fun^{\otimes\textup{-lax}}(\Span_{\calQ}(\PshT),\Sp).}
	\arrow[from=1-1, to=1-2]
	\arrow["\simeq", from=1-2, to=1-3]
\end{tikzcd}\] 
As observed in (6) above, this factors through the subcategory of functors whose restriction to $\calP^{\op}$ preserves limits. This is identified with $\CAlg(\Sp_\calQ)$ in \cref{prop:qcommutativeglobalringsaslaxfunctors}, leading to the functor in the statement. That the resulting functor $\End(\bbone)\colon \Cat(T)^\otimes_{\calQ\textup{-st}}\to \Fun^{\otimes\textup{-lax}}(\Span_{\calQ}(\PshT),\Sp)$ is itself limit preserving follows from the fact that limits of $\calQ$-stable categories are computed underlying, as shown in \Cref{prop:limits_of_stable}. Namely, given a limit $\Cc\simeq \lim_i\Cc_i$ of $\calQ$-stable $T$-categories we may compute that
\[
\End_{\Cc}(\bbone)(X) \coloneqq \hom_{\Cc_X}(\bbone,\bbone) \simeq \lim_i\hom_{\Cc_{X_i}}(\bbone,\bbone) \simeq \lim_i \End_{\Cc_i}(\bbone)(X). \qedhere
\]
\end{proof}

We record the following variant for use later. The proof is exactly as above, replacing the appeal to \cref{prop:global_sections} by \cref{prop:relative_global_sections}.

\begin{prop}\label{prop:rel_global_sections}
Let $\Cc$ be a symmetric monoidal $\calQ$-semiadditive $T$-category such that $f^*\colon \Cc(B)\to \Cc(A)$ admits a right adjoint $f_*$ for all maps $f\colon A\to B$ in $\PshT$. Then there exists an object 
$\Gamma_{\Cc}\bbone \in \CAlg(\Mon_{\calQ}\Cc(\ast))$, such that $\Gamma_{\Cc}\bbone(A) \simeq \pi_\ast \bbone_{\calC_A}$, where $\pi \colon A\to \ast$ is the unique map to the point.
\end{prop}

Recall, given a $\calQ$-ambidextrous $\E_\infty$-ring $E$ in $\calC$ and $Y\in \calQ$, we associated in \cref{def:homology} a lax symmetric monoidal functor 
\[
\D E\colon \PshT_{/Y} \to \Mod_{E(Y)}\calC
\] which was our candidate for $E$-homology. As we will show now, when $E$ comes via decategorification from a $\calQ$-semiadditive $\ab$-global category $\calC$, this deserves to be called a homology theory.

\begin{prop}\label{prop:homology_colim}
Let $\calC$ be a $\calQ$-semiadditive $T$-category such that $f^*\colon \calC(Y)\to \calC(X)$ admits a right adjoint for all maps of ab-global spaces, and let $Y\in \calQ$. Then the functor
\[
\D \Gamma_{\calC}\bbone \colon \PshT_{/Y}\to \Mod_{E(Y)}\calC(\ast)
\] associated with $\Gamma_{\calC}\bbone$ by \cref{def:homology} is colimit preserving.
\end{prop}

\begin{proof}
Applying \cref{prop:homology_underlying} it suffices to show that the composite
\[
\calQ\to \Span_{\calQ}(\PshT) \xrightarrow{\Gamma_{\calC}\bbone} \calC(\ast)
\]
is left Kan extended from $T$. So consider $Y\in \calQ$ and recall that $\calC_Y \simeq \lim_{A\in T_{/Y}^{\op}} \calC_{A}$. Then, applying the dual of  \cite[Th.B]{AdjDesc} and the $\calQ$-semiadditivity of $\calC$, we obtain the following natural equivalences
\[
\colim_{A\in T_{/Y}} (\pi_{A})_!\xrightarrow{\sim} \pi_! \xrightarrow{\sim} \pi_*.
\]
Applying this to the unit of $\calC_Y$ we obtain an equivalence 
\[
\colim_{A\in T_{/Y}} \pi_!\bbone_{\calC_{A}} \simeq \D \Gamma_{\calC}\bbone(A).
\] 
However, the functoriality of the diagram on the left agrees with that of $\D \Gamma_{\calC}\bbone$ by \cite[Pr.5.26]{CLLSpans}, and so we obtain the desired equivalence
\[
\D \Gamma_{\calC}\bbone(Y)\simeq \colim_{A\in T_{/Y}} \D\Gamma_{\calC}\bbone(\BG).\qedhere
\]
\end{proof}

\begin{remark}
In \cite{benmoshe_tempered}, Ben-Moshe also applies the results of \cite{CLLSpans} to coherently endow the decategorification of a $\pi$-semiadditive symmetric monoidal $\ab$-global category $\calC$ with transfers. Our construction agrees with his in this case. Unfortunately, his construction no longer works when $\calQ$ is some family of transfers which does not contain the maps $\BG\to \ast$ for all $G\in \Glo_\ab$, such as $\awayfromS$. Since we are crucially interested in the partial functoriality of transfers in tempered cohomology along general maps of $\P$-divisible groups, we are required to make the more elaborate constructions above.
\end{remark}

\section{Tempered local systems and parametrized stability}\label{sec:temperedlocalsystems}

The motivating example of a $\pi$-stable $\ab$-global category is the $\ab$-global category $\LocSys_\bfG$ of \emph{tempered local systems} of an \emph{oriented $\bfP$-divisible group} $\bfG$, introduced by Lurie in \cite{ec3}. Our goal in this section is to recall these notions, then examine the naturality of $\LocSys_\bfG$ in the underlying oriented $\bfP$-divisible group $\bfG$.

In more detail, we recall basic notions from \cite{ec3} in \Cref{ssec:recollectionofpdivs} in a language suited for this article. In \Cref{ssec:functorialityofTLS}, we show that $\LocSys_{(\bs)}$ is a functor from the category of oriented $\bfP$-divisible groups to that of symmetric monoidal ab-global categories, without saying anything about $\calQ$-stability. In \Cref{ssec:pcovings}, we define an inductive subcategory $\awayfromS \subset\pi\subset  \Spc_\ab^\gl$ of morphisms of $\pi$-finite spaces which are perpendicular to a set of primes $S$. In \Cref{ssec:characterisation_awayfromp}, we give a characterization of these morphisms in classical homotopy theory. In \Cref{ssec:temperedambidexterity}, we define an analogous class of morphisms of oriented $\bfP$-divisible groups that are sufficiently away from $S$, and show that $\LocSys_{(\bs)}$ defines a functor sending such morphisms to morphisms of symmetric monoidal $\awayfromS$-stable ab-global categories; see \Cref{thm:extendedfunctoriality}. Finally, in \Cref{ssec:stacks}, we show how all of this discussion formally extends to oriented $\bfP$-divisible groups over stacks using the formulation of stacks from \cite{reconstruction}.

\subsection{Recollections on \texorpdfstring{$\bfP$}{P}-divisible groups and tempered local systems}\label{ssec:recollectionofpdivs}

We start by recalling some definitions from \cite{ec3} in the form we will use them. It will be convenient to use geometric language, so we write
\[
\Aff \coloneqq \CAlg^\op
\]
for the category of affine stacks, defined (for the present purposes) as opposite to the category of $\bfE_\infty$ rings. Given  $\A \in \CAlg$, write $\Spec \A \in \Aff$ for the corresponding affine stack and $\Aff_\A = \Aff_{/\Spec\A} = \CAlg_\A^\op$. We will extend the constructions of this section to non-affine stacks in \cref{ssec:stacks}.

\begin{mydef}[{\cite[Df.3.5.1]{ec3}}]\label{def:pdivisible}
An \emph{affine $\bfP$-divisible group} $\bfG$ is a functor
\[
\bfG[\bs]\colon \Ab_\fin^\op\to\Aff
\]
satisfying the following conditions:
\begin{enumerate}
\item $\bfG$ sends biCartesian squares to Cartesian squares;
\item $\bfG$ sends injections to finite flat morphisms \cite[Df.5.2.3.1]{sag} of positive degree.
\end{enumerate}
\end{mydef}

The (affine) $\bfP$-divisible groups assemble into a full subcategory
\[
\pdivaff \subset \Fun(\Ab_\fin^\op,\Aff).
\]
Condition (1) in \cref{def:pdivisible} guarantees that the evaluation
\[
\pdivaff \to \Aff,\qquad \bfG \mapsto \bfG[0]
\]
is a Cartesian fibration. For a fixed $\einfty$ ring $\A$ we may identify
\[
\BT(\A) \coloneqq \{\Spec \A\}\times_{\Aff}\pdivaff \subset \Fun(\Ab_\fin^\op,\Aff_\A)
\]
as the category of \emph{$\bfP$-divisible groups over $\A$}, so that $\pdivaff$ is the coCartesian unstraightening of a functor
\[
\BT\colon \CAlg = \Aff^\op \to \Cat; \qquad \A \mapsto \BT(\A).
\]

\begin{rmk}
Let $\bfG'$ and $\bfG$ be $\bfP$-divisible group over $\einfty$ rings $\A'$ and $\A$. A morphism $\alpha\colon \bfG' \to \bfG$ of $\bfP$-divisible groups sits in a diagram
\begin{center}\begin{tikzcd}
\bfG'\ar[rr,"\tilde{\alpha}"]\ar[dr]\ar[rrrr,"\alpha",bend left]&&f^\ast \bfG\ar[rr]\ar[dl]&&\bfG\ar[d]\\
&\Spec \A' \ar[rrr,"f"]&&&\Spec \A
\end{tikzcd},\end{center}
where $f$ is isomorphic to $\alpha[0]\colon \bfG'[0] \to \bfG[0]$ and the rightmost square is Cartesian, that is,
\[
(f^\ast\bfG)[H] \simeq \Spec \A' \times_{\Spec \A}\bfG[H]
\]
for all $H \in \Ab_\fin^\op$. The morphism $\alpha$ is \emph{Cartesian} (with respect to the Cartesian fibration $\pdivaff\to\Aff$) if and only if the outer square is Cartesian, that is, if and only if $\tilde{\alpha}\colon \bfG' \to f^\ast \bfG$ is an isomorphism identifying $\bfG'$ as the base change of $\bfG$ along $f$.
\end{rmk}

\begin{notation}
Let $\bfG$ be a $\bfP$-divisible group over an $\einfty$ ring $\A$. Then $\bfG$ splits as a sum
\[
\bfG\simeq\bigoplus_p \bfG_{(p)}
\]
of \emph{$p$-divisible groups} for each prime $p$ \cite[Rmk.2.6.7]{ec3}. For a fixed prime $p$, write
\[
\bfG_p \coloneqq \Spec (\A_p^\wedge)\times_{\Spec \A} \bfG_{(p)}
\]
for the base change of the underlying $p$-divisible group to the $p$-completion of $\A$, and 
\[
\bfG_p^\circ\subset\bfG_p
\]
for its identity component, a formal group. See \cite[\textsection 2]{ec2} for details.
\end{notation}

\begin{mydef}
Recall that $\Glo_\ab\subset\spaces$ is the full subcategory of connected finite groupoids with abelian fundamental group. The \emph{shifted Pontryagin duality functor} is the composite
\[
\Ab_{\fin}^{\op} \xrightarrow{(\dual{\bs})} \Ab_{\fin} \xrightarrow{\bfB(\bs)} \Glo_{\ab},
\]
of Pontryagin duality $(\dual{\bs}) = \Hom(\bs,\bfQ/\bbZ)$ and the functor $\bfB(\bs)$ which sends an abelian group to its delooping.
\end{mydef}

\begin{mydef}[{\cite[Df.2.6.12]{ec3}}]\label{df:preorientations}
Let $\bfG$ be an affine $\bfP$-divisible group and set $\Spec\A = \bfG[0]$. A \emph{preorientation} of $\bfG$ is an extension of $\bfG$ through shifted Pontryagin duality:
\begin{center}\begin{tikzcd}
\Ab_\fin^\op\ar[r,"{\bfG[\bs]}"]\ar[d,"\bfB(\dual{\bs})"']&\Aff\\
\Glo_\ab\ar[ur,"\bfG(\bs)"',dashed]
\end{tikzcd}.\end{center}
This is said to be an \emph{orientation} if, for every prime $p$, the Bott map
\[
\omega_{\bfG_p^\circ} \to \Sigma^{-2}\A_p^\wedge
\]
associated with $\bfG_p^\circ$ is an isomorphism. See \cite[\textsection 4.3]{ec2} for details.
\end{mydef}

As with $\bfP$-divisible groups, the categories of (pre)oriented affine $\bfP$-divisible groups form full subcategories
\[
\pdivoraff \subset \pdivpreaff \subset \Fun(\Glo_\ab,\Aff),
\]
and evaluation on the terminal object $\ast \in \Glo_\ab$ defines Cartesian fibrations
\[
\pdivpreaff \to \Aff,\qquad \pdivoraff \to \Aff.
\]
For an $\bfE_\infty$ ring $\A$, we write
\[
\BT^{\ori}(\A) \subset \BT^{\pre}(\A) \subset \Fun(\Glo_\ab,\Aff_\A)
\]
for the fibers over $\Spec \A$.

\begin{notation}\label{notation:preorientednotation}
Let $\bfG$ be a preoriented $\bfP$-divisible group. In this case, we write
\begin{itemize}
\item $\bfG[\bs]\colon \Ab_\fin^\op\to\Aff$ for the underlying $\bfP$-divisible group of $\bfG$;
\item $\bfG(\bs)\colon \Glo_\ab \to \Aff$ for the functor defining $\bfG$;
\item $\Gamma\calO_\bfG^{(\bs)}\colon \Glo_\ab^\op\to\CAlg$ for the functor formally dual to $\bfG(\bs)$.
\end{itemize}
When $\bfG$ is a preoriented $\bfP$-divisible group over a fixed $\bfE_\infty$ ring $\A$, we write $\A_\bfG^{(\bs)} = \Gamma\calO_\bfG^{(\bs)}$, just as in \cite[Con.4.0.3]{ec3}.
\end{notation}

We now recall the construction of tempered local systems. Fix a functor $\bfG\colon \Glo_\ab \to \Aff$ and set $\Spec \A = \bfG[0]$.

\begin{mydef}[{\cite[Not.5.1.2]{ec3}}]
Given a global space $\sfX \in \abglobalspaces$, the composite
\[
\A_{\bfG,\sfX}\colon (\Glo_\ab)_{/\sfX}^\op \to \Glo_\ab^\op \xrightarrow{\A_\bfG^{(\bs)}} \CAlg
\]
defines a commutative algebra in $\Fun((\Glo_\ab)_{/\sfX}^\op,\Sp)$, equipped with the pointwise symmetric monoidal structure. We write
\[
\Mod_{\A_{\bfG,\sfX}} \coloneqq \Mod_{\A_{\bfG,\sfX}}(\Fun((\Glo_\ab)^\op_{/\sfX},\Sp))
\]
for the category of \emph{$\A_{\bfG,\sfX}$-modules}.
\end{mydef}

\begin{rmk}
Our notation differs from Lurie, who writes $\ul{\A}_\bfG^\sfX$ where we write $\A_{\bfG,\sfX}$; this is to free the notation $\ul{\A}_\bfG$ for use later as the $\pi$-ambidextrous spectrum associated with $\bfG$ when $\bfG$ is oriented.
\end{rmk}

From now on we suppose further that $\bfG$ is a preoriented $\bfP$-divisible group over $\A$.

\begin{mydef}[{\cite[Df.5.2.4]{ec3}}]\label{def:temperedlocalsystems}
Given a global space $\sfX \in \abglobalspaces$, a \emph{$\bfG$-tempered local system} on $\sfX$ is an $\A_{\bfG,\sfX}$-module $\fmodule$ which satisfies the following conditions:
\begin{enumerate}
\item Let $\alpha\colon T \to T'$ be a morphism in $(\Glo_\ab)_{/\sfX}$ with connected fibers. Then $\calF(\alpha)$ induces an equivalence
\[
\A_\bfG^{T'}\otimes_{\A_\bfG^T}\calF(T)\to\calF(T')
\]
of $\A_\bfG^{T'}$-modules.
\item Let $T \in (\Glo_\ab)_{/\sfX}$ and let $T_0\to T$ be a connected covering space. Define the relative augmentation ideal
\[
I_\bfG(T_0/T) = \ker\left(\pi_0 \A_\bfG^T \to \pi_0 \A_\bfG^{T_0}\right).
\]
Then the canonical map $\calF(T) \to \calF(T_0)^{\h \Aut(T_0/T)}$ exhibits $\calF(T_0)^{\h T_0/T}$ as the completion of $\calF(T)$ with respect to $I_\bfG(T_0/T)$.
\end{enumerate}
We write
\[
\LocSys_\bfG(\sfX)\subset\Mod_{\A_{\bfG,\sfX}}
\]
for the full subcategory of $\bfG$-tempered local systems on $\sfX$.
\end{mydef}

\subsection{Functoriality of tempered local systems}\label{ssec:functorialityofTLS}

Our goal in this subsection is to establish the following.

\begin{theorem}\label{thm:basicfunctoriality}
The construction of tempered local systems assembles into a functor
\[
\LocSys_{(\bs)}\colon (\pdivoraff)^\op\to\Cat(\Glo_\ab)^\otimes,\qquad \bfG \mapsto (\sfX \mapsto \LocSys_\bfG(\sfX))
\]
\end{theorem}

When restricted to the wide subcategory $\pdiv^{\ori}_{\aff,\cart}\subset\pdivoraff$ spanned by the Cartesian morphisms, \cref{thm:basicfunctoriality} was proved by Ben-Moshe \cite[\textsection4.1]{benmoshe_tempered}. Our theorem generalizes this by allowing non-Cartesian morphisms of oriented $\bfP$-divisible groups, such as non-invertible isogenies. The basic strategy is the same, making use of the following bit of general categorical input.

\begin{lemma}[Ben-Moshe]\label{lem:modfunctoriality}
The construction of $\Gamma\ul{\calO}_\bfG^{\sfX}$-modules refines to a functor
\[
\Fun(\Glo_\ab,\Aff)^\op \to \Fun((\abglobalspaces)^\op,\CAlg(\Cat)),\qquad \bfG \mapsto (\sfX \mapsto \Mod_{\Gamma\ul{\calO}_\bfG^\sfX}).
\]
\end{lemma}
\begin{proof}
This construction is carried out in \cite[Con.4.3, Steps 1--3]{benmoshe_tempered}. To be precise, Ben-Moshe writes down a functor
\[
\Fun(\Glo_\ab,\Aff)^\op\times(\abglobalspaces)^{\op} \to \Cat,\qquad (\bfG,\sfX) \mapsto \Mod_{\Gamma\ul{\calO}_\bfG^\sfX}
\]
which does not incorporate the symmetric monoidal structure on $\Mod_{\Gamma\ul{\calO}_\bfG^\sfX}$. However, this is merely an artifact of the use of the forgetful functor $\CAlg(\Cat) \to \Cat$ in \cite[Con.4.3, Step 2]{benmoshe_tempered}, and by omitting this use one obtains the lemma as written.
\end{proof}

We will deduce \cref{thm:basicfunctoriality} from \cref{lem:modfunctoriality} by verifying that if $\alpha\colon \bfG'\to\bfG$ is a map of oriented $\bfP$-divisible groups, then $\alpha^\ast\colon \Mod_{\Gamma\ul{\calO}_{\bfG'}^{\sfX}} \to \Mod_{\Gamma\ul{\calO}_{\bfG}^{\sfX}}$ restricts to a functor between categories of tempered local systems. For Cartesian morphisms, this is exactly \cite[Pr.6.2.1(d)]{ec3}. The proof for general morphisms requires a little more care.

\begin{lemma}\label{lm:reduction_to_monochromatic}
Let $\A$ be a complex-periodic $\E_\infty$ ring locally of finite height, and consider a commutative diagram
	\begin{equation}\label{eq:maps_of_rings}\begin{tikzcd}
			&\A\ar[dl,"f"']\ar[dr,"g"]\\
			\R\ar[rr,"\alpha"]&&\sfS
	\end{tikzcd}\end{equation}
	of $\bfE_\infty$ rings. If
	\begin{enumerate}
		\item $f$ and $g$ are finite flat (in the sense of \cite[\textsection 5.2.3]{sag}, i.e.\ finite projective);
		\item $\alpha$ is finite \'etale after base change to $L_{K(p,n)}\A$ for all primes $p$ and heights $n$;
	\end{enumerate}
	then $\alpha$ is finite \'etale.
\end{lemma}
\begin{proof}
It suffices to prove that $\alpha$ is finite \'etale after localization at each prime ideal of $\pi_0 \A$, so we may suppose that $\A$ is local and therefore $L_n$-local for some fixed prime $p$ and $n < \infty$. Arguing by induction on $n$, we may assume that $\alpha$ is finite \'etale after base change to $L_{n-1}\A$. Let $I\subset \pi_0 \A$ denote the $n$-th Landweber ideal. Then the top chromatic fracture square of $\A$ may be identified with its $I$-adic fracture square:
\[
\begin{tikzcd}
\A\ar[r]\ar[d]&L_{K(n-1)}\A\ar[d]\\
L_{n-1}\A\ar[r]&L_{n-1}L_{K(n-1)}\A
\end{tikzcd}
\qquad = \qquad
\begin{tikzcd}
\A\ar[r]\ar[d]&\A_{I}^\wedge\ar[d]\\
L_{I}\A\ar[r]&L_{I}(\A_{I}^\wedge)
\end{tikzcd}.
\]

As $\R$ and $\sfS$ are finite flat over $\A$, the map $\alpha$ is base changed from connective covers, meaning that the natural square of $\E_\infty$ rings
    \[\begin{tikzcd}
        {\R_{\geq 0}}\ar[r]\ar[d]    &   {\sfS_{\geq 0}}\ar[d]  \\
        {\R}\ar[r]                   &   {\sfS}
    \end{tikzcd}\]
    is coCartesian. It therefore suffices to prove that $\R_{\geq 0}\rightarrow \sfS_{\geq 0}$ is finite \'etale. As $\R$ and $\sfS$ are finite flat $\A$-algebras, the map $\pi_0 \sfR \rightarrow \pi_0 \sfS$ is of finite presentation \cite[\textsection 1 Pr.1.4.7]{egafouri}. By \cite[Lm.B.1.3.3]{sag}, it therefore suffices to prove that the relative cotangent complex $L_{\sfS_{\geq 0}/\sfR_{\geq 0}}$ vanishes. To that end, it suffices to show that the localization $L_I(L_{\sfS_{\geq 0}/\sfR_{\geq 0}})$ and completion $(L_{\sfS_{\geq 0}/\sfR_{\geq 0}})_I^\wedge$ vanish.
    
    In the former case, as localization is exact, it commutes with connective covers and we may identify
\[
L_I(L_{\sfS_{\geq 0}/\sfR_{\geq 0}}) \simeq L_{L_I\sfS_{\geq 0}/L_I\sfR_{\geq0}} \simeq L_{(L_I\sfA\otimes_\sfA\sfS)_{\geq 0}/(L_I\sfA\otimes_\sfA\sfR)_{\geq 0}}.
\]
As finite \'etale morphisms are preserves by connective covers, this vanishes under the inductive assumption that $\sfR \to \sfS$ is finite \'etale after $I$-localization, i.e.\ after $L_{n-1}$-localization.

In the latter case, it suffices to prove that $F \otimes L_{\sfS_{\geq 0}/\sfR_{\geq 0}}\simeq 0$ for any $I$-torsion $\A_{\geq 0}$-module $F$. In general, if $L$ is any $\A_{\geq 0}$-module, then 
\[
((L_{\geq 0})_I^\wedge)_{\geq 0}\simeq (L_I^\wedge)_{\geq 0}.
\]
In particular, there is a natural map $(\A_I^\wedge)_{\geq 0} \to (\A_{\geq 0})_I^\wedge$ realizing the latter as an algebra over the former, and so
\[
F\otimes L \simeq F \otimes (\A_{\geq 0})_I^\wedge \otimes_{\A_{\geq 0}}L \simeq F \otimes (\A_{\geq 0})_I^\wedge\otimes_{(A_I^\wedge)_{\geq 0}}(\A_I^\wedge)_{\geq 0} \otimes_{\A_{\geq 0}} L,
\]
so to prove $F\otimes L\simeq 0$ it suffices to prove that $(\A_I^\wedge)_{\geq 0}\otimes_{\A_{\geq 0}}L\simeq 0$. As $\R$ is finite flat over $\A$ we may identify
\[
(\A_I^\wedge)_{\geq 0}\otimes_{\A_{\geq 0}}R_{\geq 0}\simeq (\A_I^\wedge\otimes_\A \R)_{\geq 0},
\]
and similarly for $\sfS$, and therefore
\[
(\A_I^\wedge)_{\geq 0}\otimes_{\A_{\geq 0}}L_{\sfS_{\geq 0}/\sfR_{\geq 0}}\simeq L_{(\A_I^\wedge\otimes_\A\sfS)_{\geq 0}/(\A_I^\wedge\otimes_\A\sfR)_{\geq 0}}.
\]
The assumption that $\alpha$ is finite \'etale after base change to $\A_I^\wedge$, i.e.\ after $K(n-1)$-localization, implies that this vanishes and the lemma follows.
\end{proof}

\begin{prop}\label{lem:orientedseparable}
Let $\alpha\colon \bfG' \to \bfG$ be a homomorphism of \emph{oriented} $\bfP$-divisible groups over an $\bfE_\infty$ ring $\A$. Then $\al(T)\colon \bfG'(T) \to \bfG(T)$ is finite \'etale for all $T \in \Glo_\ab$.
\end{prop}

\begin{proof}
By \Cref{lm:reduction_to_monochromatic}, it suffices to show that $\alpha(T)$ is finite \'etale after base change to $L_{K(p,n)}\A$ for each prime $p$ and height $n$. We may thus reduce to the case where $\A$ is $K(n)$-local for a fixed (implicit) prime $p$ and height $n$. If $n=0$, then $\A$ is rational. In this case, $\bfG'[H]$ and $\bfG[H]$ are themselves finite \'etale over $\A$, and therefore any map between them is finite \'etale. Suppose then that $\A$ is $K(n)$-local for some $0 < n < \infty$, so that the Quillen $\bfP$-divisible group $\bfG^\calQ_\A$ exists \cite[\textsection 4.6]{ec2}. We can reinterpret the fact that $\al$ is a map of oriented $\bfP$-divisible groups as follows: \cite[Pr.2.4.1]{ec3} provides the left square in a commutative diagram of $\bfP$-divisible groups
	\begin{center}\begin{tikzcd}
			0\ar[r]&\bfG^{\calQ}_{\A}\ar[r, "e'"]\ar[d, "="]&\bfG^{\prime}\ar[d,"\alpha"]\ar[r,dashed]&\bfG'{}^\et\ar[d]\ar[r]&0\\
			0\ar[r]&\bfG^\calQ_{\A}\ar[r, "e"]&\bfG\ar[r,dashed]&\bfG^\et\ar[r]&0,
	\end{tikzcd}\end{center}
    where the horizontal maps encode the orientations of $\bfG'$ and $\bfG$, respectively, and the commutativity of this square follows from the fact that $\al$ preserves orientations. By \cite[Pr.2.5.6]{ec3}, the maps $e$ and $e'$ are monomorphisms of $\bfP$-divisible groups, and therefore by \cite[Pr.2.4.8]{ec2} their cokernels $\G^{\prime\et}$ and $\G^\et$ both exist as $\bfP$-divisible groups, providing the given exact rows. In particular, the right square exists and is Cartesian. This, in turn, induces a Cartesian square when evaluated on $T\in \Glo_\ab$ of the form
    \begin{center}\begin{tikzcd}
	\bfG'(T)\ar[r]\ar[d, "{\al(T)}"]&\bfG^{\prime\et}(T)\ar[d]\\
	\bfG(T)\ar[r]&\bfG^\et(T),
	\end{tikzcd}\end{center}
    so it suffices to show that the right-hand vertical map is finite \'{e}tale. As $\bfG^\calQ_\A$ is the connected component of $\bfG$ and $\bfG'$ \cite[Th.4.6.16]{ec2}, it follows that the cokernels $\bfG^\et$ and $\bfG^{\prime\et}$ are \'{e}tale. In particular, both $\bfG^{\prime\et}(T)$ and $\bfG^\et(T)$ are finite \'{e}tale over $\Spec \A$, implying that any map between them is again finite \'etale.
\end{proof}

\begin{lemma}\label{prop:preservetemper}
Let $\alpha\colon \bfG' \to \bfG$ be a morphism of oriented $\bfP$-divisible groups, and let $\sfX$ be a global space. Then the restriction
\[
\alpha^\ast\colon \Mod_{\calO_{\bfG,\sfX}} \to \Mod_{\calO_{\bfG',\sfX}}
\]
preserves full subcategories of tempered local systems.
\end{lemma}
\begin{proof}
If $\alpha$ is a Cartesian morphism, then this is shown in \cite[Pr.6.2.1]{ec3}. As morphisms of oriented $\P$-divisible groups factor into a Cartesian map followed by a map over a fixed $\E_\infty$ ring, we therefore reduce to the case where $\alpha$ is a morphism of oriented $\bfP$-divisible groups over a fixed $\einfty$ ring $\A$. We must show that if $\fmodule\in \LocSys_\bfG(\sfX)$, then $\alpha^\ast\fmodule \in \LocSys_{\bfG'}(\sfX)$ satisfies the two conditions of \cref{def:temperedlocalsystems}. By definition, 
\[
\alpha^\ast \calF \in \Mod_{\A_{\bfG',\sfX}}(\Fun((\Glo_\ab)^\op_{/\sfX},\spectra))
\]
is defined by
\[
(\alpha^\ast\calF)(T) = \A_{\bfG'}^T \otimes_{\A_\bfG^T}\fmodule(T)
\]
for $T \in (\Glo_\ab)^\op_{/\sfX}$. The two conditions are now verified as follows.

(1)~~Suppose $T' \to T$ is a morphism in $(\Glo_\ab)^\op_{/\sfX}$ with connected fibers, then as $\fmodule$ satisfies (1) it follows that 
\begin{align*}
\A_{\bfG'}^{T'}\otimes_{\A_{\bfG'}^T}\alpha^\ast\fmodule(T)&\simeq \A_{\bfG}^{T'}\otimes_{\A_{\bfG'}^T}\A_{\bfG'}^T\otimes_{\A_\bfG^T}\fmodule(T)\\
&\simeq \A_{\bfG'}^{T'}\otimes_{\A_\bfG^{T'}}\A_\bfG^{T'}\otimes_{\A_\bfG^T}\fmodule(T)\\
&\simeq \A_{\bfG'}^{T'}\otimes_{\A_\bfG^{T'}}\fmodule(T') \simeq \alpha^\ast\fmodule(T'),
\end{align*}
and so $\alpha^\ast\fmodule$ satisfies (1).

(2)~~By \cref{lem:orientedseparable}, the morphism $\A_{\bfG}^T \to \A_{\bfG'}^T$ is finite \'etale. In particular, $\A_{\bfG'}^T$ is dualizable as an $\A_{\bfG}^T$-module, implying that 
\[
\A_{\bfG'}^{T}\otimes_{\A_{\bfG}^{T}}(\bs)\colon \Mod_{\A_{\bfG}^T} \to \Mod_{\A_{\bfG}^{T'}}
\]
preserves limits. Since $\A_{\bfG}^T \to \A_{\bfG'}^T$ is flat the relative augmentation ideal $I_{\bfG}(T_0/T) \subset \pi_0 \A_{\bfG}^T$ associated with a covering map $T_0 \to T$ satisfies 
\[
\pi_0 \A_{\bfG'}^T\otimes_{\pi_0 \A_{\bfG}^T} I_{\bfG}(T_0/T)\simeq I_{\bfG'}(T_0/T).
\]
Combining these two observations, we find that if $\fmodule$ is a $\bfG$-tempered local system, then
\[
\A_{\bfG'}^T\otimes_{\A_{\bfG}^T}\fmodule(T) \to (\A_{\bfG'}^T\otimes_{\A_\bfG^T}\fmodule(T_0))^{\h \Aut(T_0/T)} \simeq \A_{\bfG'}^T\otimes_{\A_{\bfG}^T}\fmodule(T_0)^{\h \Aut(T_0/T)}
\]
exhibits the target as the $I_{\bfG'}(T_0/T)$-completion of the source.
\end{proof}

We can now give the following.

\begin{proof}[Proof of \cref{thm:basicfunctoriality}]
By \cref{lem:modfunctoriality}, we have a functor
\[
\Fun(\Glo_\ab,\Aff)^\op \to \Fun((\abglobalspaces)^\op,\CAlg(\Cat)),\qquad \bfG \mapsto \left(\sfX \mapsto \Mod_{\calO_{\bfG,\sfX}}\right).
\]
By \cref{prop:preservetemper}, if $\bfG$ is $\bfP$-divisible group then we may restrict along the subcategories
\[
\LocSys_\bfG(\sfX)  \subset \Mod_{\calO_{\bfG,\sfX}}
\]
to obtain a functor
\begin{align*}
(\pdivoraff)^\op &\to \Fun((\abglobalspaces)^\op,\CAlg(\Cat)),\qquad \bfG \mapsto \left(\sfX \mapsto  \LocSys_\bfG(\sfX)\right).
\end{align*}
By \cite[Pr.5.1.9, Rmk.5.2.11]{ec3}, this lands in the full subcategory 
\[
\Cat(\Glo_\ab)^\otimes\subset \Fun((\abglobalspaces)^\op,\CAlg(\Cat))
\]
of limit preserving functors.
\end{proof}

This proves the basic functoriality of $\LocSys_{(\bs)}$ as a functor from oriented $\bfP$-divisible groups into symmetric monoidal ab-global categories.

\subsection{Relatively \texorpdfstring{$\pi$}{pi}-finite morphisms away from a set of primes}\label{ssec:pcovings}

Given a homomorphism $\alpha\colon \bfG'\to\bfG$ of oriented $\bfP$-divisible groups, the functor 
\[
\alpha^\ast\colon\LocSys_\bfG\to\LocSys_{\bfG'}
\]
of \Cref{thm:basicfunctoriality} between $\ab$-global categories generally does \emph{not} preserve parametrized colimits indexed by a map of relatively $\pi$-finite spaces. Informally, if $\alpha_{(p)}\colon \bfG'_{(p)} \to \bfG_{(p)}$ is not an isomorphism, then the induced map on tempered local systems will generally fail to commute with colimits along maps involving interesting equivariance at the prime $p$. As a consequence, the resulting map $\A_{\bfG}^{(\bs)} \to \A_{\bfG'}^{(\bs)}$ will also fail to commute with transfers along such maps. In this subsection, we isolate a class of maps of orbispaces that are ``away'' from a set of primes that will allow us to properly formulate and resolve this.

\begin{mydef}
A \emph{colattice} is a torsion abelian group $\Lambda$ with the property that, for all positive integers $n$, the map $n\colon \Lambda \to \Lambda$ is a surjection with finite kernel. In this case, we write
\[
\dual{\Lambda} \coloneqq \Hom(\Lambda,\bfQ/\bfZ)
\]
for the \emph{lattice} dual to $\Lambda$.
\end{mydef}

\begin{rmk}
If $\Lambda$ is a colattice, then there exists a (noncanonical) equivalence
\[
\Lambda\simeq\bfQ_{p_1}^{n_1}/\bfZ_{p_1}^{n_1}\oplus\cdots\oplus\bfQ_{p_k}^{n_k}/\bfZ_{p_k}^{n_k}
\]
for some finite list of primes $p_1,\ldots,p_k$ and nonnegative integers $n_1,\ldots,n_k$. In this case,
\[
\dual{\Lambda} \simeq \bfZ_{p_1}^{n_1}\oplus\cdots\oplus\bfZ_{p_k}^{n_k}.
\]
\end{rmk}

Given a $\pi$-finite space $F$, we will abuse notation and write $F = \bfR(F)$ for the image of $F$ under the restricted Yoneda embedding $\bfR\colon \Spc\to \Spc^{\gl}_{\ab}$ of \Cref{df:rightadjoint}. In particular, in this section, we will at times make the implicit identification
\[
\bfB H = BH.
\]

\begin{mydef}\label{def:formalloopspace}
Let $\Lambda$ be a colattice. The \emph{formal loop space functor with respect to $\Lambda$} is the unique colimit preserving functor
\[
\calL^\Lambda\colon \abglobalspaces\to\abglobalspaces
\]
extending the assignment
\begin{equation}\label{eq:formaloops_onGlo}
\calL^\Lambda(T) \coloneqq \map(B\dual{\Lambda},T) \simeq \underset{H \subset \Lambda\text{ finite}}{\colim} \map(B\dual{H},T)\in \Spc_{\pi}\subset \abglobalspaces
\end{equation}
for $T \in \Glo_\ab$. This construction is evidently natural in $\Lambda$. We also abbreviate 
\[
 \calL_p\coloneqq\calL^{\bfQ_p/\bfZ_p}.
\]
\end{mydef}

\begin{remark}\label{rmk:formalloops_commuteswithlimits}
\cref{def:formalloopspace} agrees with the definition given by Lurie in \cite[Con.3.4.3]{ec3}, i.e.\
\[
\calL^\Lambda(\sfX) \simeq \underset{H \subset \Lambda\text{ finite}}{\colim}\left(\ul{\mathrm{Map}}(\bfB \dual{H}, \sfX)\right),
\]
where $\ul{\mathrm{Map}}$ is the internal mapping object in $\abglobalspaces$. To see this, observe that Lurie's definition commutes with colimits by \cite[Rmk.3.4.6]{ec3}, and is therefore determined by its values on $\Glo_\ab$, where it agrees with \cref{def:formalloopspace} by \cite[Pr.3.4.7]{ec3}. In particular, it is easily seen from this latter construction that $\calL^\Lambda(\bs)$ preserves finite limits.
\end{remark}

\begin{ex}\label{ex:pifiniteformalloopspace}
Let $F$ be a $\pi$-finite space. Then the space $\map(B\dual{\Lambda},F)$ is again $\pi$-finite and by \cite[Pr.3.4.7]{ec3}, the natural map
\[
\calL^\Lambda(F) \to \map(B\dual{\Lambda},F)
\]
is an equivalence of global spaces. In particular, $\calL^\Lambda(F)$ is again a $\pi$-finite space.
\end{ex}

\begin{mydef}
The functor $\calL^\Lambda$ comes equipped with two natural transformations
\[
\sfX\xrightarrow{c} \calL^\Lambda(\sfX) \xrightarrow{\epsilon} \sfX.
\]
induced by restriction along the unique homomorphisms $0 \leftarrow \Lambda \leftarrow 0$ of colattices. Clearly, the composite $\epsilon\circ c\colon \sfX\to \calL^{\Lambda}(\sfX) \to \sfX$ is the identity.
\end{mydef}

\begin{mydef}\label{def:pcover}
Let $S$ be a set of prime numbers. A relatively $\pi$-finite map $f\colon \sfY\to\sfX$ of global spaces is \emph{away from $S$} if the square
\begin{center}\begin{tikzcd}
\calL_p \sfX \ar[r,"\calL_p f"]\ar[d]&\calL_p \sfY\ar[d]\\
\sfX\ar[r,"f"]&\sfY
\end{tikzcd}\end{center}
is Cartesian for all $p \in S$.
\end{mydef}

Fix a set $S$ of prime numbers for the rest of this subsection.

\begin{ex}\label{ex:pipcover}
Let $\alpha\colon F \to F'$ be a map of $\pi$-finite spaces. Then \Cref{ex:pifiniteformalloopspace} implies that $\alpha$ is away from $S$ if and only if the square
\begin{center}\begin{tikzcd}
\map(B\bfZ_p,F)\ar[r]\ar[d]&\map(B\bfZ_p,F')\ar[d]\\
F\ar[r]&F'
\end{tikzcd}\end{center}
of ordinary spaces is Cartesian for all $p\in S$. 
\end{ex}

Our next goal is to prove that the class of morphisms away from $S$ is generated by an inductible subcategory. This will make use of the following.

\begin{lemma}\label{lem:pcoverbasic}
Let $f\colon \sfY \to \sfX$ be a map of global spaces.
\begin{enumerate}
\item If $f$ is of the form $f = (f_i) \colon \coprod_{i\in I} \sfY_i \to \sfX$ for a collection of maps $f_i\colon \sfY_i \to \sfX$, or of the form $f = (f_i)\colon \coprod_{i\in I}\sfY_i \to \coprod_{i\in I}\sfX_i$ for a collection of maps $f_i\colon \sfY_i \to \sfX_i$, then $f$ is away from $S$ if and only if each $f_i$ is away from $S$.
\item Suppose
\begin{center}\begin{tikzcd}
\sfY'\ar[r]\ar[d,"f'"]&\sfY\ar[d,"f"]\\
\sfX'\ar[r]&\sfX
\end{tikzcd}\end{center}
is a Cartesian square in $\Spc^\gl_\ab$.
\begin{enumerate}
\item If $f$ is away from $S$, then $f'$ is away from $S$.
\item If $f'$ is away from $S$ and $\calL_p\sfX' \to \calL_p\sfX$ is an effective epimorphism for all $p \in S$, then $f$ is away from $S$.
\end{enumerate}
\item Suppose given another map $g\colon \sfZ \to \sfY$ of global spaces.
\begin{enumerate}
\item If $g$ is away from $S$, then $f$ is away from $S$ if and only if $g\circ f$ is away from $S$.
\item If $f$ is away from $S$ an effective epimorphism, then $g$ is away from $S$ if and only if $g\circ f$ is away from $S$.
\end{enumerate}
\item If $f$ is away from $S$, then the diagonal $\sfY\to\sfY\times_\sfX\sfY$ is away from $S$.
\item If $f$ is away from $S$, then $\calL^\Lambda f$ is away from $S$ for any colattice $\Lambda$.
\end{enumerate}
\end{lemma}
\begin{proof}
(1)~~This holds as $\calL_p$ preserves coproducts and coproducts are disjoint in the presheaf category $\abglobalspaces$.

(2)~~For $p \in S$, consider the cube
\begin{center}\begin{tikzcd}
\calL_p\sfY'\ar[rr]\ar[dr]\ar[dd]&&\calL_p\sfY\ar[dr]\ar[dd]\\
&\sfY'\ar[rr]\ar[dd]&&\sfY\ar[dd]\\
\calL_p\sfX'\ar[rr]\ar[dr]&&\calL_p\sfX\ar[dr]\\
&\sfX'\ar[rr]&&{\sfX.}
\end{tikzcd}\end{center}
The front and back faces are Cartesian. It follows that if the right face is Cartesian, then the left face is Cartesian, and that the converse holds provided $\calL_p\sfX' \to \calL_p\sfX$ is an effective epimorphism as needed.

(3)~~The proof is analogous to (2), using the pasting law for pullbacks.

(4)~~Consider the diagram
\begin{center}\begin{tikzcd}
\sfY\ar[r,"\Delta"]\ar[dr,equals]&\sfY\times_\sfX\sfY\ar[r]\ar[d]&\sfY\ar[d]\\
&\sfY\ar[r]&\sfX
\end{tikzcd}.\end{center}
As $\sfY \to \sfX$ is away from $S$, so is $\sfY\times_\sfX\sfY \to \sfY$ by (1a). As the composite $\sfY \to \sfY\times_\sfX\sfY\to\sfY$ is the identity, it is clearly away from $S$. Therefore, $\sfY \to \sfY\times_\sfX\sfY$ is away from $S$ by (3a).

(5)~~This holds as $\calL_p\calL^\Lambda\simeq \calL^\Lambda\calL_p$ and $\calL^\Lambda$ preserves finite limits by \Cref{rmk:formalloops_commuteswithlimits}.
\end{proof}

\begin{prop}\label{lem:pcoverlocal}
Let $f\colon \sfY \to \sfX$ be a relatively $\pi$-finite map of global spaces. Then $f$ is away from $S$ if and only if for every $T \in \Glo_\ab$ and Cartesian square
\begin{center}\begin{tikzcd}
F\ar[r]\ar[d]&T\ar[d]\\
\sfY\ar[r,"f"]&\sfX
\end{tikzcd}\end{center}
in $\Spc^\gl_\ab$, the map $F \to T$ is away from $S$.
\end{prop}
\begin{proof}
We claim that if $\sfX$ is a global space and $\Lambda$ is a colattice, then there exists an effective $\coprod_{i\in I} T_i \to \sfX$ with $T_i \in \Glo_\ab$ for which $\coprod_{i\in I} \calL^\Lambda T_i \to \calL^\Lambda \sfX$ is again an effective epimorphism. The proposition then follows from \cref{lem:pcoverbasic}(1,2).

In fact, such a cover is given by the tautological effective epimorphism $\coprod_{i\in I} T_i \to \sfX$ indexed by $I = \coprod_{T\in \Glo_\ab}\pi_0 \sfX(T)$. To see this, fix $T \in \Glo_\ab$ and $x \in \calL^\Lambda \sfX(T)$. Then $x$ is represented by a map $f_x\colon T \times \BH \to \sfX$ for some finite subgroup $H \subset \Lambda$, and this in the image of $\calL^\Lambda f_x \colon \calL^\Lambda(T\times \BH) \to \calL^\Lambda \sfX$.
\end{proof}

The previous result shows that the class of maps that are away from $S$ is generated locally by those maps of $\pi$-finite spaces that are away from $S$. Denote this subcategory of $\abglobalspaces$ by $\awayfromS$.

\begin{cor}\label{cor:qstpreinductible}
The wide subcategory $\awayfromS\subset\abglobalspaces$ forms an inductible subcategory. Moreover, the local class $\overline{\awayfromS}$ generated by $\awayfromS$ (see \cref{def:local_class_generated}) is exactly the wide subcategory spanned by the relatively $\pi$-finite maps which are away from $S$.
\end{cor}
\begin{proof}
Immediate from \cref{lem:pcoverbasic} and \cref{lem:pcoverlocal}.
\end{proof}

\subsection{Characterization of morphisms away from a prime}\label{ssec:characterisation_awayfromp}
Our goal for this subsection is to give a more explicit description of the relatively $\pi$-finite maps that are away from $S$. By \cref{lem:pcoverlocal}, to understand the local structure of such maps, it suffices to understand when a map of the form $F \to \BH$, with $H$ a finite abelian group and $F$ a $\pi$-finite space, is away from $S$. This reduces to a problem in ordinary homotopy theory by \cref{ex:pipcover}. By \cref{lem:pcoverbasic}(1), we may as well assume that $F$ is connected. As a map is away from $S$ if and only if it is away from $p$ for all $p\in S$, we may focus on the case where $S = \{p\}$. Finally, as $\BH \to \BH_{(p)}$ is easily seen to be away from $p$, by \cref{lem:pcoverbasic}(3a) we may as well assume that $H$ is a finite $p$-group. Having made these reductions our analysis now starts with the following.

\begin{lemma}\label{lem:trivialpcoverptorsion}
	Let $f\colon G\rightarrow H$ be a homomorphism of finite groups, and suppose that the induced map $f\colon \BG\rightarrow \BH$ is away from $p$. Then $f$ induces a bijection between the elements in $G$ and $H$ of $p$-power order.
\end{lemma}
\begin{proof}
	Given a finite group $G$, the elements of $p$-power order in $G$ can be identified with the set $\Hom(\bfZ_p,G)$. If $f$ is away from $p$, then taking vertical fibers in the Cartesian square
	\begin{center}\begin{tikzcd}
			\Map(B\bfZ_p,BG)\ar[r]\ar[d]&\Map(B\bfZ_p,BH)\ar[d]\\
			BG\ar[r]&BH
	\end{tikzcd}\end{center}
	furnishes an isomorphism $\Hom(\bfZ_p,G)\cong \Hom(\bfZ_p,H)$, and the lemma follows.
\end{proof}

\begin{lemma}\label{lem:pcoverabeliangroup}
Let $f\colon G\rightarrow H$ be a homomorphism of finite groups. Suppose that $H$ is an abelian $p$-group and that the induced map $\BG \to \BH$ is away from $p$. Then $f$ is the projection in a splitting
	\[
	G\cong K \times H
	\]
	with $p\nmid |K|$.
\end{lemma}
\begin{proof}
As $f$ is away from $p$, then by \cref{lem:trivialpcoverptorsion}, it restricts to a bijection between elements of $p$-power order. We claim that all elements of $p$-power order in $G$ are central. As $H$ is an abelian $p$-group, this would imply the claimed splitting $G \cong K \times H$. To prove this claim, we identify
	\[
	\Map(B\bfZ_p,BG)\simeq \Hom(\bfZ_p,G)_{\h G}\simeq \coprod_{g\in \Hom(\bfZ_p,G)/G}BC_G(g),
	\]
	where $G$ acts on $\Hom(\bfZ_p,G)$ by conjugation and $C_G(g) = \{x\in G : gx = xg\}\subset G$ is the centralizer of $g$. Thus if $f$ is away from $p$, then the square
	\begin{center}\begin{tikzcd}
			\coprod_{g\in \Hom(\bfZ_p,G)/G}BC_G(g)\ar[r]\ar[d]&{H\times BH \simeq \Map(\Z_p, BH)}\ar[d]\\
			BG\ar[r]&BH
		\end{tikzcd}\end{center}
is Cartesian. This is only possible if the left vertical map induces isomorphisms $C_G(g)\cong G$ for all $g\in G$ of $p$-power order, implying that $g$ is central as claimed.
\end{proof}

This lemma admits the following generalization.

\begin{prop}\label{prop:pifinawayfromp}
	Let $F$ be a pointed connected $\pi$-finite space and $H$ be a finite abelian $p$-group. Given a map $f\colon F\rightarrow \BH$, the following are equivalent:
	\begin{enumerate}
		\item $f$ is away from $p$
		\item All elements in $\pi_1 F$ of $p$-power order act trivially on $\pi_n F$ for all $n \geq 1$;
		\item $f$ is the projection in a splitting $F\simeq f^{-1}(e)\times BH$, and the homotopy groups of $f^{-1}(e)$ have order coprime to $p$.
	\end{enumerate}
\end{prop}
\begin{proof}
	(1)$\Rightarrow$(2). In general, a map $F \to \BH$ is away from $p$ if and only if
	\begin{center}\begin{tikzcd}
			\Map(B\bfZ_p,F)\ar[r]\ar[d]&\Map(B\bfZ_p,BH)\ar[d]\\
			F\ar[r]&BH
	\end{tikzcd}\end{center}
is Cartesian. Under the assumption that $F$ is connected, this holds if and only if the induced map on vertical fibers furnishes an equivalence
	\[
	\Map_\ast(B\bfZ_p,F)\simeq \Hom(\bfZ_p,H).
\]
In particular, if this holds, then $\Map_\ast(B\bfZ_p,F)$ is $0$-truncated.

Now suppose that $f$ is away from $p$. We will induct down on $n$ to prove that the following hold for all $n\geq 2$:
	\begin{enumerate}
		\item[(a${}_n$)] $F_{\leq n-1}\rightarrow BH$ is away from $p$;
		\item[(b${}_n$)] $\pi_n F$ has order coprime to $p$;
		\item[(c${}_n$)] All elements of $p$-power order in $\pi_1 F$ act trivially on $\pi_n F$.
	\end{enumerate}
Item (a${}_2$) combines with \cref{lem:pcoverabeliangroup} to establish (2) in the case $n = 1$. Combined with (c${}_n$) for all other $n$, this establishes (2).

The base for the downwards induction is provided by some $n$ for which $F\simeq F_{\leq n-1}$, where all three conditions are trivially satisfied. For the inductive step, we fix some $n\geq 2$ and suppose that (a${}_{n+1}$--c${}_{n+1}$) hold. To prove that (a${}_n$--c${}_n$) hold, we analyze the top Postnikov square of $F_{\leq n}$. Given a group $G$ acting on an abelian group $M$, define
	\[
	B^n_G M = (B^n M)_{\h G}.
	\]
Abbreviate $G = \pi_1 F$. Then the top Postnikov square of $F_{\leq n}$ takes the form
	\begin{center}\begin{tikzcd}
			F_{\leq n}\ar[r]\ar[d]&BG\ar[d]\\
			F_{\leq n-1}\ar[r]&B_{G}^{n+1}\pi_n F
\end{tikzcd}.\end{center}
Mapping $B\Z_p$ into this provides a Cartesian square
	\begin{equation}\label{eq:homzppostnikov}\begin{tikzcd}
			\Hom(\bfZ_p,H)\ar[r]\ar[d,"\simeq"]&\Hom(\bfZ_p,G)\ar[d,"\simeq"]\\
			\Map_\ast(B\bfZ_p,F_{\leq n})\ar[r]\ar[d]&\Map_\ast(B\bfZ_p,BG)\ar[d]\\
			\Map_\ast(B\bfZ_p,F_{\leq n-1})\ar[r]&\Map_\ast(B\bfZ_p,B_G^{n+1}\pi_n F)
		\end{tikzcd},\end{equation}
where the top left equivalence uses the identification from the beginning of the proof. Note that $\Map_\ast(B\bfZ_p,F_{\leq n-1})$ is $(n-1)$-truncated and that the top row of this square consists of discrete spaces. As $n \geq 2$, it follows that $\Map_\ast(B\bfZ_p,B^{n+1}_G\pi_n F)$ is $(n-1)$-truncated. We will use this to show that all $p$-power torsion in $G$ acts trivially on $\pi_n F$.

For each map $f\colon \bfZ_p \to G$, we may form the fiber sequence
\begin{center}\begin{tikzcd}
\Map_{\ast//BG}(B\bfZ_p,B_G^{n+1}\pi_n F)\ar[r]\ar[d]&\Map_\ast(B\bfZ_p,B_G^{n+1}\pi_n F)\ar[d]\\
\{f\}\ar[r]&\Map_\ast(B\bfZ_p,BG)
\end{tikzcd}.\end{center}
Using that the bottom corner is discrete, we find that the top left corner is again $(n-1)$-truncated. Now, note there is an adjunction
	\[
	\Map_{\ast//BG}(B\bfZ_p,B_G^{n+1}\pi_n F)\simeq \Map_{\ast//B\bfZ_p}(B\bfZ_p,B_{\bfZ_p}^{n+1}\pi_n F)
	\]
	and that the latter space definitionally fits into a Cartesian square
	\begin{center}\begin{tikzcd}
			\Map_{\ast//B\bfZ_p}(B\bfZ_p,B_{\bfZ_p}^{n+1}\pi_n F)\ar[r]\ar[d]&\Map_{/B\bfZ_p}(B\bfZ_p,B_{\bfZ_p}^{n+1}\pi_n F)\ar[d]\\
			\ast\ar[r]&B^{n+1}\pi_n F
		\end{tikzcd}.\end{center}
    As $\Map_{\ast//\bfZ_p}(B\bfZ_p,B_{\bfZ_p}^{n+1}\pi_n F)$ is $(n-1)$-truncated, and so we conclude from the long exact sequence that
    \[
    \pi_{n+1} \Map_{/B\bfZ_p}(B\bfZ_p,B_{\bfZ_p}^{n+1}\pi_n F)\cong \pi_{n+1} B^{n+1} \pi_n F \cong \pi_n F
    \]
    and
    \[
    \pi_n \Map_{/B\bfZ_p}(B\bfZ_p,B_{\bfZ_p}^{n+1}\pi_n F) \cong 0. 
    \]
    If $G$ is a group acting on $M$ then
	\[
	\pi_k \Map_{/BG}(BG,B_G^n M) \cong H^{n-k}(G;M)
	\]
	is isomorphic to the group cohomology of $G$ with coefficients in $M$. This allows us to reinterpret the first isomorphism above as
	\[
	\pi_n F\cong \pi_{n+1} \Map_{/B\bfZ_p}(B\bfZ_p,B_{\bfZ_p}^{n+1}\pi_n) \cong H^0(\bfZ_p;\pi_n F) \cong (\pi_n F)^{\bfZ_p}.
	\]
Thus we have shown that for any map $f\colon \bfZ_p \to G$, we have $(\pi_n F)^{\bfZ_p} \cong \pi_n F$. In other words, all elements in $G$ of order a power of $p$ act trivially on $\pi_n F$. This establishes (c${}_n$).

Having established that the induced action of $\bfZ_p$ on $\pi_n F$ is trivial, we may identify
	\[
	\Hom(\bfZ_p,\pi_n F) = H^1(\bfZ_p;\pi_n F) \cong \pi_n \Map_{/B\bfZ_p}(B\bfZ_p,B_{\bfZ_p}^{n+1}\pi_n F) \cong 0.
	\]
This is only possible if $\pi_n F$ has order coprime to $p$, establishing (b$_n$).

As $\pi_n F$ has order coprime to $p$ and every element of $p$-power order in $G$ acts trivially on $\pi_n F$, we find that the Postnikov truncation $B_G^{n+1}\pi_n F \to BG$ induces an equivalence
	\[
	\Map_\ast(B\bfZ_p,B_G^{n+1}\pi_n F) \rightarrow\Hom(\bfZ_p,G).
	\]
Combined with \cref{eq:homzppostnikov} and the assumption that $F_{\leq n} \to \BH$ is a away from $p$, we deduce that 
\[
\Map_\ast(B\bfZ_p,F_{\leq n-1}) \simeq \Map_\ast(B\bfZ_p,F_{\leq n}) \simeq \Hom(\bfZ_p,H),
\]
implying that $F_{\leq n-1} \to H$ is again away from $p$, establishing (a$_n$).

	(2)$\Rightarrow$(3). Suppose that all elements of $p$-power order in $\pi_1 F$ act trivially on $\pi_n F$ for $n\geq 1$. In particular, $\pi_1 F \cong K \times H$ with $H\subset \pi_1 F$ the subset of elements of $p$-power order, and $H$ acts trivially on the higher homotopy groups of $F$. Suppose moreover, that $f^{-1}(e)$ is a connected space with homotopy groups of order coprime to $p$.  We induct on $n$ to prove that $F_{\leq n} \cong f^{-1}(e)_{\leq n}\times BH$ for all $n \geq 1$, having just established the base case $n=1$. For the inductive step, consider the top Postnikov square of $F_{\leq n}$, which is of the form
	\begin{center}\begin{tikzcd}
			F_{\leq n}\ar[r]\ar[d]&BK\times BH\ar[d]\\
			f^{-1}(e)_{\leq n-1}\times BH\ar[r]&B_{K\times H}^{n+1}\pi_n F
		\end{tikzcd}.\end{center}
	As $H$ acts trivially on $\pi_n F$, this may be further refines to a Cartesian square
	\begin{center}\begin{tikzcd}
			F_{\leq n}\ar[r]\ar[d]&BK\ar[d]\\
			f^{-1}(e)_{\leq n-1} \times BH\ar[r,"k"]&B_K^{n+1}\pi_n F
		\end{tikzcd}.\end{center}
		
Our assumptions imply that $\Map(\phi^{-1}(e)_{\leq n-1},B_K^{n+1}\pi_n F)$ is a $\pi$-finite space with homotopy groups coprime to $p$ at all basepoints. As $H$ is a $p$-group, basic obstruction theory implies that the diagonal map
\begin{align*}
\map(\phi^{-1}(e)_{\leq n-1},B_K^{n+1}\pi_n F) &\to \map(\phi^{-1}(e)_{\leq n-1},B^{n+1}_K\pi_n F)^{BH}\\
&\simeq \map(\phi^{-1}(e)_{\leq n-1}\times BH,B^{n+1}_K \pi_n F)
\end{align*}
is an equivalence. In other words, the $k$-invariant $ f^{-1}(e)_{\leq n-1} \times BH \to B_K^{n+1}\pi_n F$ factors through the projection $f^{-1}(e)_{\leq n-1} \times BH \to f^{-1}(e)_{\leq n-1}$, implying that
	\[
	F_{\leq n}\simeq \phi^{-1}(e)_{\leq n}\times BH
	\]
	as claimed.
	
	(3)$\Rightarrow$(1). Set $F' = \phi^{-1}(e)$. We must prove that if $F'$ has homotopy groups coprime to $p$, then the projection $F' \times BH \to BH$ is away from $p$. As maps away from $p$ are stable under base change, it suffices to prove that $F' \to \ast$ is away from $p$, i.e.\ that
\[
\Map(B\bfZ_p,F')\simeq F'.
\]
This follows by basic obstruction theory using the Postnikov tower of $F'$ as above.
\end{proof}

Given an abelian group $H$, we denote by $H_{(S)}
\subset H$ the subgroup spanned by all elements of $H$ whose order is only divisible by $S$.

\begin{cor}\label{cor:pcovcharacterize}
Let $F$ be a $\pi$-finite space and $H$ be a finite abelian group. Then a map $f\colon F \to \BH$ is away from $S$ if and only if there exists a decomposition
\[
F\simeq \coprod_{i\in I}F_i \times \BH_{(S)},
\]
where $I$ is a finite set, $F_i$ is a connected $\pi$-finite space with homotopy groups coprime to $p$ for $p\in S$, and $f$ is a sum of maps of the form
\[
f_i\times \id_{\BH_{(S)}} \colon F_i \times \BH_{(S)} \to \BH[1/S]\times \BH_{(S)}\simeq \BH.\qedhere
\]
\end{cor}

\cref{lem:pcoverabeliangroup} also admits the following globalization.

\begin{prop}\label{prop:pcovgroup}
	Let $f\colon K\rightarrow G$ be a homomorphism of finite groups. Then the induced map $f\colon BK\rightarrow BG$ is away from $p$ if and only if the following conditions are satisfied:
	\begin{enumerate}
		\item $f(K)$ contains all elements in $G$ of order a power of $p$;
		\item Given $x,y\in K$ with $x$ of $p$-power order, if $f(xy) = f(yx)$ then $xy=yx$;
		\item $p\nmid |\ker f|$.
	\end{enumerate}
\end{prop}
\begin{proof}
	First, suppose $f$ is away from $p$. Then \Cref{lem:trivialpcoverptorsion} immediately implies (1) and (3). For (2), note that in general
	\[
	\Map(B\bfZ_p,BG)\simeq \Hom(\bfZ_p,G)_{\h G}\simeq \coprod_{x\in \Hom(\bfZ_p,G)/G}BC_G(x).
	\]
	It follows that if $x\in K$ is $p$-torsion, then the square
	\begin{center}\begin{tikzcd}
			\coprod_{y\in K : (\exists g\in G) (gf(y)g^{-1} = f(x))}BC_K(y)\ar[d]\ar[r]&BC_G(f(x))\ar[d]\\
			BK\ar[r,"f"]&BG
	\end{tikzcd}\end{center}
	is Cartesian. In particular, the square
	\begin{center}\begin{tikzcd}
			C_K(x)\ar[r]\ar[d]&C_G(f(x))\ar[d]\\
			K\ar[r,"f"]&G
	\end{tikzcd}\end{center}
	is Cartesian, i.e.\ if $x\in K$ is $p$-torsion then $C_K(x) = f^{-1}(C_G(f(x)))$. This is another way of stating (2).
	
    Suppose conversely that $f$ satisfies (1--3). Fix a map $q\colon H\rightarrow G$ with $H$ a finite abelian $p$-group, and consider the Cartesian square
	\begin{center}\begin{tikzcd}
			F\ar[r]\ar[d]&BH\ar[d,"q"]\\
			BK\ar[r,"f"]&BG
		\end{tikzcd}.\end{center}
	Following the discussion before \Cref{lem:trivialpcoverptorsion}, it suffices to prove that the map $F\to BH$ is away from $p$. By \cref{prop:pifinawayfromp}, this is equivalent to it being a projection $F'\times BH \to BH$ where $F'$ has homotopy groups coprime to $p$. By the double coset formula, the path components of $F$ are of the form $BL_g$ for some $g\in G$, where
	\[
	L_g = \{(k,h) \in K \times H : f(k) = g q(h) g^{-1}\}.
	\]
After possibly replacing $q$ with a conjugate map we may as well consider the case $g = e$, so that
	\[
	L_g = K \times_G H = \{(k,h) \in K \times H : f(k) = q(h)\}.
	\]
	Condition (1) ensures that $K\times_G H \rightarrow H$ is a surjection. The kernel of this map is isomorphic to the kernel of $f$, i.e.\ there is an exact sequence
	\[
	1\rightarrow \ker f \rightarrow K\times_G H \rightarrow H \rightarrow 1.
	\]
	Condition (3) says that $\ker f$ has order coprime to $p$, and so it suffices to prove that all $p$-torsion in $K\times_G H$ is central. Indeed, suppose that $(k,h) \in K\times_G H$ has order a power of $p$, and let $(k',h') \in K \times_G H$ be arbitrary. As $(k,h)$ has order a power of $p$, the same is true of $k$. As $H$ is abelian, we can identify
	\[
	f(kk') = f(k)f(k') = q(h)q(h') = q(hh') =  q(h'h) = q(h')q(h) = f(k')f(k) = f(k'k).
	\]
	Condition (2) then implies that $kk' = k'k$, and thus $(k,h)$ is central as claimed.
\end{proof}

\begin{cor}
	Let $f\colon K \rightarrow G$ be a homomorphism of finite groups. Then the induced map $f\colon BK\rightarrow BG$ is away from $S$ if and only if $K \to \image(f)$ and $\image(f)\to G$ are away from $S$.
	\qed
\end{cor}

\begin{ex}
The map of classifying spaces associated with
	\begin{enumerate}
		\item The inclusion $C_3\subset\Sigma_3$ is away from $3$;
		\item The surjection $D_{30}\rightarrow D_6$ is away from $3$;
		\item The maps $C_2\rightarrow\Sigma_3$ and $\Sigma_3\rightarrow C_2$ are not away from $2$;
        \item The maps $e \to G$ and $G\to e$ are away from $p$ if and only if $p\nmid |G|$.
	\end{enumerate}
    
\end{ex}

\subsection{Parametrized exactness of tempered local systems}\label{ssec:temperedambidexterity}

Lurie's \emph{tempered ambidexterity theorem} \cite[Th.7.2.10]{ec3} can be reformulated in the language of parametrized higher category theory as follows.

\begin{theorem}[Lurie]\label{thm:temperedambidexterity}
Let $\bfG$ be an oriented $\bfP$-divisible group over an $\einfty$ ring. Then $\LocSys_\bfG \in \Cat(\Glo_\ab)^\otimes$ is a symmetric monoidal $\pi$-stable $\ab$-global category, i.e.\ lifts to
\[
\LocSys_\bfG \in \Cat(\Glo_\ab)_{\pi\textup{-st}}^\otimes.
\]
\end{theorem}
\begin{proof}
By \cite[Pr.5.2.12.]{ec3}, $\LocSys_\bfG$ is fiberwise stable. By Pr.7.2.2 of \emph{op.~cit.}, $\LocSys_{\bfG}$ is $\pi$-cocomplete, while Th.7.2.10 is the statement that it is $\pi$-semiadditive. The tensor product clearly commutes with fiberwise finite colimits, and the projection formula holds by Th.7.3.1 of \emph{op.~cit.}
\end{proof}

\begin{remark}
We note that Lurie's work even shows that $\LocSys_\bfG$ is a \emph{presentably} symmetric monoidal $\pi$-stable $\ab$-global category, i.e.\ it admits all parametrized limits and colimits and the tensor product commutes with parametrized colimits in each variable. We will not make essential use of this.
\end{remark}

In general, given a morphism $\alpha\colon \bfG' \to \bfG$ of oriented $\bfP$-divisible groups, the induced functor $\alpha^\ast\colon \LocSys_{\bfG'} \to \LocSys_{\bfG}$ of $\ab$-global categories does \emph{not} preserve all $\pi$-limits, and therefore \cref{thm:basicfunctoriality} does \emph{not} combine with \cref{thm:temperedambidexterity} to provide a lift of $\LocSys_{(\bs)}$ to a functor valued in $\Cat(\Glo_\ab)^\otimes_{\pi\hyp\st}$.  Before stating the main theorem of this section, we introduce some notation.

\begin{mydef}\label{df:Cartesian_away_from_S}
Let $S$ be a set of prime numbers. A map $\alpha\colon \bfG'\to\bfG$ of $\bfP$-divisible groups is \emph{Cartesian away from $S$} if for all primes $p\notin S$, the morphism $\alpha_{(p)}\colon \bfG'_{(p)}\to\bfG_{(p)}$ of $p$-divisible groups is Cartesian. Write
\[
\pdiv^{\ori}_{S\nmid,\aff}\subset\pdiv^{\ori}_\aff
\]
for the wide subcategory spanned by the morphisms which are Cartesian away from $S$.
\end{mydef}

\begin{rmk}
A Cartesian morphism of $\bfP$-divisible groups is by definition Cartesian away from any set $S$ of prime numbers. In particular, the composite
\[
\pdiv^{\ori}_{S\nmid,\aff}\subset\pdiv^{\ori}_\aff\to\Aff
\]
is a Cartesian fibration. For an $\E_\infty$ ring $\A$, write $\BT^{\ori}_{S\nmid}(\A)$ for the fiber over $\A$. Then
\[
\BT^{\ori}_{S\nmid}(\A)\subset\BT^{\ori}(\A)
\]
is the wide subcategory spanned by those morphisms $\alpha\colon \bfG'\to\bfG$ of $\bfP$-divisible groups over $\A$ for which $\alpha_{(p)}$ is an isomorphism for all $p\notin S$.
\end{rmk}

Recall that $\awayfromS$ denotes the inductible subcategory of $\abglobalspaces$ spanned by maps of $\pi$-finite spaces that are away from the set $S$ of primes. Our goal in this section is to establish the following.

\begin{theorem}\label{thm:extendedfunctoriality}
The construction of tempered local systems assembles into a functor
\[
(\pdiv_{S\nmid,\aff}^\ori)^\op\to\Cat(\Glo_\ab)^\otimes_{\awayfromS\hyp\st},\qquad \bfG\mapsto\LocSys_\bfG.
\]
\end{theorem}

Before the proof, it will be convenient to introduce some auxiliary definitions.

\begin{notation}\label{notation:restrictionadjunction}
Let $\bfG$ be an oriented $\bfP$-divisible group and $p\colon \sfY \to \sfX$ be a map of global spaces. Then the restriction
\[
p^\ast\colon \LocSys_{\bfG}(\sfX) \to \LocSys_{\bfG}(\sfY)
\]
preserves all limits and colimits, and therefore sits in a string of adjunctions denoted as
\[
p_! \dashv p^\ast \dashv p_\ast;
\]
see \cite[Not.7.0.1]{ec3}.
\end{notation}

\begin{mydef}
Let $\alpha\colon \bfG'\to\bfG$ be a morphism of oriented $\bfP$-divisible groups and $q\colon \sfY\to\sfX$ be a map of global spaces. Say that \emph{$\alpha^\ast$ preserves $q$-limits} if the square
	\begin{center}\begin{tikzcd}
			\LocSys_{\bfG}(\sfX)\ar[r,"q^\ast"]\ar[d,"\alpha^\ast"]&\LocSys_{\bfG}(\sfY)\ar[d,"\alpha^\ast"]\\
			\LocSys_{\bfG'}(\sfX)\ar[r,"q^\ast"]&\LocSys_{\bfG'}(\sfY)
	\end{tikzcd}\end{center}
is right adjointable in the sense that the natural mate transformation 
	\[
	\alpha^\ast q_\ast \rightarrow q_\ast \alpha^\ast
	\]
	is an equivalence.
\end{mydef}

If $\calQ\subset\abglobalspaces$ is an inductible category consisting of $\pi$-finite morphisms and $\alpha\colon\bfG'\to\bfG$ is a morphism of $\bfP$-divisible groups, then by definition
\[
\alpha^\ast\colon\LocSys_{\bfG'}\to\LocSys_\bfG
\]
is a map of $\calQ$-stable $\ab$-global categories if and only if $\alpha^\ast$ preserves $q$-limits for all $q\in \calQ$. We start by identifying those morphisms of $\bfP$-divisible groups that always satisfy this.

\begin{prop}\label{prop:Cartesianpilimits}
Let $\alpha\colon \bfG' \to \bfG$ be a Cartesian morphism of oriented $\bfP$-divisible groups. Then $\alpha^*$ preserves all $\pi$-limits.
\end{prop}
\begin{proof}
As proved by \cite[Pr.4.6]{benmoshe_tempered}, base change along a Cartesian morphism of $\bfP$-divisible groups induces a parametrized left adjoint 
\[
\alpha^*\colon \LocSys_{\bfG}\to \LocSys_{\bfG'}
\]
on tempered local systems, and therefore preserves all parametrized \emph{colimits}. Combining this with \cref{lem:QexactFunctor} proves the proposition above.
\end{proof}

As maps in $\pdivoraff$ decompose into a Cartesian morphism followed by a morphism over a fixed $\E_\infty$ ring, we have thus reduced to considering morphisms over a given fixed $\E_\infty$ ring.

\begin{prop}\label{prop:commutesqlims}
Let $\alpha\colon \bfG'\to\bfG$ be a morphism of $\bfP$-divisible groups over an $\bfE_\infty$ ring $\A$, and suppose that $\alpha_{(p)}\colon \bfG'_{(p)}\to\bfG_{(p)}$ is an isomorphism for all $p\notin S$. Then $\alpha^\ast$ commutes with $q$-limits for all $q\in \awayfromS$.
\end{prop}

\begin{proof}
Let $q\colon \sfY\to\sfX$ be a relatively $\pi$-finite map of orbispaces away from $S$, and consider the square
\begin{equation}\label{eq:alphaqadj}
\begin{tikzcd}
\LocSys_{\bfG}(\sfX)\ar[r,"q^\ast"]\ar[d,"\alpha^\ast"]&\LocSys_\bfG(\sfY)\ar[d,"\alpha^\ast"]\\
\LocSys_{\bfG'}(\sfX)\ar[r, "q^\ast"]&\LocSys_{\bfG'}(\sfY).
\end{tikzcd}.\end{equation}
Note that
\[
\sfX\simeq\underset{T\in (\Glo_\ab)_{/\sfX}}{\colim}T,\qquad \sfY\simeq\underset{T\in (\Glo_\ab)_{/\sfX}}{\colim}(\sfY\times_\sfX T).
\]
As $\LocSys_\bfG$ and $\LocSys_{\bfG'}$ are $\ab$-global categories, it follows that (\ref{eq:alphaqadj}) is a limit of the corresponding squares for the maps $\sfY\times_\sfX T \to T$. As adjointable squares are closed under limits \cite[Cor.4.7.4.18]{ha}, we therefore reduce to the case where $\sfX = T \in \Glo_\ab$. By \cref{cor:pcovcharacterize}, we therefore reduce to the case where $q$ is of the form
\[
(q'_i,\id_T)\colon \coprod_{i\in I}F_i\times T \to T' \times T,
\]
where
\begin{enumerate}
\item $I$ is a finite set;
\item Each $F_i$ is a connected $\pi$-finite space and $T,T'\in \Glo_\ab$;
\item The orders of $\pi_\ast F_i$ and $\pi_1 T'$ are coprime to the order of $\pi_1 T$;
\item If a prime $p$ divides the order of $\pi_\ast F_i$ or $\pi_1 T'$, then $p\notin S$.
\end{enumerate}
This map factors as the composite
\[
\nabla\circ (q'_i,\id_T)\colon \coprod_{i\in I}F_i\times T \to  \coprod_{i\in I} T'\times T \to T' \times T.
\]
As $\alpha^\ast$ is additive, it preserves $\nabla$-limits. We may therefore reduce to the case where $I$ is a singleton, and abbreviate $F = F_i$ and $q' = q'_i$. In other words, we have reduced to verifying that the square 
\begin{equation}
\begin{tikzcd}
\LocSys_{\bfG}(T'\times T)\ar[r,"q^\ast"]\ar[d,"\alpha^\ast"]&\LocSys_\bfG(F\times T)\ar[d,"\alpha^\ast"]\\
\LocSys_{\bfG'}(T'\times T)\ar[r,"q^\ast"]&\LocSys_{\bfG'}(F\times T).
\end{tikzcd}\end{equation}
is right adjointable, where $q = q'\times \id_T\colon F \times T \to T'\times T$ and these spaces satisfy (1--4) above. We will show this by a direct and straightforward---if notationally heavy---computation of the comparison map $\alpha^\ast q_\ast \to q_\ast\alpha^\ast$.

Write $\Glo_\ab^{S\nmid}$ and $\Glo_\ab^S$ for the full subcategories of $\Glo_\ab$ spanned by those $T$ for which the prime factors of $\pi_1 T$ lie either out of or in $S$, and observe that
\[
\Glo_\ab^{S\nmid}\times\Glo_\ab^{S} \to \Glo_\ab,\qquad (T',T)\mapsto T'\times T
\]
is an equivalence. In particular, suppose given a map $T'' \to T'\times T$ in $\Glo_\ab$. Then this map splits canonically as
\[
T''[1/S]\times T''_{(S)} \to T' \times T,
\]
where $T''[1/S] \in \Glo_\ab^{S\nmid}$ and $T''_{(S)} \in \Glo_\ab^S$, at which point
\[
T'' \times_{(T'\times T)}(F\times T)\simeq (T''[1/S] \times_{T'}F) \times T''_{(S)}
\]
and
\[
(\Glo_{\ab})_{/T''\times_{(T'\times T)}(F\times T)} \simeq (\Glo_\ab^{S\nmid})_{/T''[1/S]\times_{T'}F}\times (\Glo_\ab^{S})_{/T''_{(S)}}.
\]
In particular, the functor
\begin{align*}
(\Glo_\ab^{S\nmid})_{/T''[1/S]\times_{T'}F} &\to (\Glo_\ab)_{/T''\times_{(T'\times T)}(F\times T)},\\
 (f\colon T''' \to T''[1/S]\times_{T'}F) &\mapsto (f\times T''_{(S)}\colon T'''\times T''_{(S)} \to T''[1/S]\times_{T'}F \times T''_{(S)})
\end{align*}
is cofinal.

Now fix $\fmodule \in \LocSys_\bfG(F\times T)$. We may describe the map
\[
\alpha^\ast q_\ast \fmodule \to q_\ast \alpha^\ast\fmodule
\]
as follows. Fix a map $T'' \to T'\times T$. We then have, with the above notation,
\begin{align*}
(\alpha^\ast q_\ast\fmodule)(T'') &= \A_{\bfG'}^{T''}\otimes_{\A_\bfG^{T''}}\lim_{T'''\in (\Glo_\ab)_{/T''\times_{T'\times T}F\times T}}\calF(T''')\\
&\simeq \A_{\bfG'}^{T''}\otimes_{\A_{\bfG}^{T''}} \lim_{T''' \in (\Glo_\ab^{S\nmid})_{/T''[1/S]\times_{T'}F}}\fmodule(T'''\times T''_{(S)}).
\end{align*}
By the definition of a $\bfP$-divisible group, we have
\[
\A_{\bfG}^{T''}\simeq \A_{\bfG}^{T''[1/S]}\otimes_\A \A_\bfG^{T''_{(S)}},\qquad 
\A_{\bfG'}^{T''}\simeq \A_{\bfG'}^{T''[1/S]}\otimes_\A \A_{\bfG'}^{T''_{(S)}}.
\]
By assumption, the map $\A_{\bfG}^{T''[1/S]} \to \A_{\bfG'}^{T''[1/S]}$ is an isomorphism. It follows that
\[
(\alpha^\ast q_\ast\fmodule)(T'')\simeq \A_{\bfG'}^{T''_{(S)}}\otimes_{\A_{\bfG}^{T''_{(S)}}}\lim_{T''' \in (\Glo_\ab^S)_{/T''[1/S]\times_{T'}F}}\fmodule(T'''\times T''_{(S)}).
\]
Similarly, we compute
\begin{align*}
(q_\ast\alpha^\ast\fmodule)(T'') &= \lim_{T'''\in (\Glo_\ab)_{/T''\times_{(T'\times T)}(F\times T)}}\A_{\bfG'}^{T'''}\otimes_{\A_{\bfG}^{T'''}}\fmodule(T''')\\
&\simeq \lim_{T''' \in (\Glo_\ab^S)_{/T''[1/S]\times_{T'}F}}\A_{\bfG'}^{T'''\times T''_{(S)}}\otimes_{\A_{\bfG'}^{T'''\times T''_{(S)}}} \fmodule(T'''\times T''_{(S)})\\
&\simeq \lim_{T''' \in (\Glo_\ab^S)_{/T''[1/S]\times_{T'}F}}\A_{\bfG'}^{T''_{(S)}}\otimes_{\A_{\bfG'}^{T''_{(S)}}} \fmodule(T'''\times T''_{(S)}),
\end{align*}
and so $(\alpha^\ast q_\ast \fmodule)(T'') \to (q_\ast \alpha^\ast \fmodule)(T'')$ is the comparison map
\begin{align*}
\A_{\bfG'}^{T''_{(S)}}\otimes_{\A_{\bfG}^{T''_{(S)}}}& \lim_{T''' \in (\Glo_\ab^S)_{/T''[1/S]\times_{T'}F}}\fmodule(T'''\times T''_{(S)}) \\
\to
&\lim_{T''' \in (\Glo_\ab^S)_{/T''[1/S]\times_{T'}F}}\A_{\bfG}^{T''_{(S)}}\otimes_{\A_{\bfG'}^{T''_{(S)}}} \fmodule(T'''\times T''_{(S)}).
\end{align*}
This is an equivalence as $\A_{\bfG}^{T''_{(S)}} \to \A_{\bfG'}^{T''_{(S)}}$ is finite \'etale and therefore dualizable by \cref{lem:orientedseparable}.
\end{proof}

Putting everything together, we have the following.

\begin{proof}[Proof of \cref{thm:extendedfunctoriality}]
We must show that if $\alpha\colon \bfG'\to\bfG$ is a morphism of oriented $\bfP$-divisible groups which is Cartesian away from $S$, then $\alpha^\ast$ commutes with $q$-limits for all $q\in \awayfromS$. The morphism $\alpha$ factors as a composite 
\[
\alpha = \alpha'\circ\alpha''\colon \bfG'\to\bfG''\to\bfG,
\]
where $\alpha'\colon \bfG''\to\bfG$ is a Cartesian morphism and $\alpha''\colon\bfG'\to\bfG''$ is a morphism of oriented $\bfP$-divisible groups over a fixed base for which $\alpha''_{(p)}$ is an isomorphism for all $p \notin S$. Both $\alpha^{\prime\ast}$ and $\alpha^{\prime\prime\ast}$ preserves $q$-limits by \cref{prop:commutesqlims} and \cref{prop:Cartesianpilimits} respectively, and therefore so too does $\alpha^\ast$.
\end{proof}

A certain converse to \cref{thm:extendedfunctoriality} also holds over each fixed $\bfE_\infty$ ring.

\begin{lemma}\label{lem:cplims}
Let $\bfG_0$ be an oriented $\bfP$-divisible group over an $\bfE_\infty$ ring $\A$. Let $\tilde{\alpha}\colon \Lambda'\to\Lambda$ be a map of colattices, and consider the induced map $\alpha\colon \bfG_0 \oplus \ul{\Lambda}' \to \bfG_0 \oplus \ul{\Lambda}$ of oriented $\bfP$-divisible groups over $\A$. Write $q\colon \B C_p \to \ast$ for the projection. If $\alpha^\ast$ preserves $q$-limits, then $\alpha_{(p)}$ is an isomorphism.
\end{lemma}
\begin{proof}
In general, suppose given a map $q\colon T' \to T$ in $\Glo_\ab$. The Beck--Chevalley map
\[
\alpha^\ast q_\ast \to q_\ast \alpha^\ast
\]
of functors $\LocSys_\bfG(T') \to \LocSys_{\bfG'}(T)$, after evaluating on the monoidal unit and evaluating on the terminal object of $(\Glo_\ab)_{/T}$, is the natural map
\[
\A_{\bfG'}^T \otimes_{\A_{\bfG}^T}\A_{\bfG}^{T'} \to \A_{\bfG'}^{T'}.
\]
Therefore, if $\alpha$ preserves $q$-limits, then the square
\begin{center}\begin{tikzcd}
\bfG'(T')\ar[r]\ar[d]&\bfG(T')\ar[d]\\
\bfG'(T)\ar[r]&\bfG(T)
\end{tikzcd}\end{center}
is Cartesian. By choosing a basepoint of $T'$, we may write $q$ in the form $\BH' \to \BH$ for a group homomorphism $H' \to H$. If $\alpha$ is of the form $\bfG_0\oplus \ul{\Lambda'}\to\bfG_0 \oplus \ul{\Lambda}$ for a map $\tilde{\alpha}\colon \Lambda' \to \Lambda$ of colattices and $\A$ is nonzero, then this square is Cartesian if and only if the square
\begin{center}\begin{tikzcd}
\Hom(\dual{H}',\Lambda')\ar[r]\ar[d]&\Hom(\dual{H}',\Lambda)\ar[d]\\
\Hom(\dual{H},\Lambda')\ar[r]&\Hom(\dual{H},\Lambda)
\end{tikzcd}\end{center}
of sets is Cartesian. In the case where $H' \to H$ is of the form $C_p \to e$, this holds if and only if the induced map $\Lambda'[p]\to\Lambda[p]$ on $p$-torsion is an isomorphism, which in turn holds if and only if $\tilde{\alpha}_{(p)}\colon \Lambda'_{(p)}\to\Lambda_{(p)}$ and therefore $\alpha_{(p)}\colon \bfG'_{(p)}\to\bfG_{(p)}$ is an isomorphism.
\end{proof}

\begin{prop}\label{prop:pconverse}
Let $\alpha\colon\bfG'\to\bfG$ be a morphism of oriented $\bfP$-divisible groups over an $\bfE_\infty$ ring $\A$, and suppose that $\alpha^\ast\colon \LocSys_\bfG\to\LocSys_{\bfG'}$ is a map of $\awayfromS$-stable $\ab$-global categories. Then $\alpha_{(p)}$ is an isomorphism for all $p\notin S$.
\end{prop}
\begin{proof}
For any map $\A \to \B$ of $\bfE_\infty$ rings, we have a commutative diagram
\begin{center}\begin{tikzcd}
\LocSys_\bfG\ar[d,"\alpha^\ast"]\ar[r]&\LocSys_{f^\ast\bfG}\ar[d,"\alpha^\ast"]\\
\LocSys_{\bfG'}\ar[r]&\LocSys_{f^\ast\bfG'}
\end{tikzcd}\end{center}
of $\ab$-global categories in which the horizontal functors preserve all $\pi$-limits by \cref{prop:Cartesianpilimits}. The horizontal functors are conservative as one allows $\B$ to vary over the localizations $\A_\frakp$ for prime ideals $\frakp\subset\pi_0 \A$, so we may suppose without loss of generality that $\pi_0 \A$ is local, so that $\bfG$ and $\bfG'$ have constant height. As $\A$ supports an oriented $\bfP$-divisible group, it follows that $\A$ is $L_n$-local for some $n < \infty$ \cite[Cor.2.5.7]{ec3}. Therefore, the horizontal functors are conservative after allowing $\B$ to vary over the $K(i)$-localizations of $\A$ for $i\leq n$, so we may reduce to the case where $\A$ is $K(n)$-local for some $n$. In particular, $\A$ admits a Quillen $\bfP$-divisible group by \cite[\textsection 4.6]{ec2}. By applying \cite[Pr.2.7.15]{ec3} with $\bfG_0$ the Quillen $\bfP$-divisible group of $\A$, we may, after possibly passing to a further faithfully flat extension, reduce to the case where $\alpha$ is of the form $\bfG_0 \oplus \ul{\Lambda}'\to \bfG_0 \oplus \ul{\Lambda}$ for a map $\tilde{\alpha}\colon \Lambda'\to\Lambda$ of colattices. As $\bfB C_p \to \ast$ lies in $\awayfromS$ if and only if $p\notin S$, the proposition therefore follows from \cref{lem:cplims}.
\end{proof}

\subsection{Extension to stacks}\label{ssec:stacks}

For applications, such as to elliptic cohomology, one wants to be able to work with oriented $\bfP$-divisible groups defined over non-affine base stacks. We now explain how the constructions of this section extend to non-affine base stacks. This is essentially a formal consequence of the following.

\begin{prop}\label{prop:flatdescent}
The functors
\[
\BT,\BT^\pre,\BT^\ori_{S\nmid}\colon \CAlg \to \Cat,
\]
as well as
\begin{align*}
\LocSys_{(\bs)}\colon &(\pdiv_{S\nmid,\aff}^{\ori})^\op \to \Cat(\Glo_\ab)^\otimes_{\awayfromS\hyp\st},
\end{align*}
all satisfy flat descent.
\end{prop}
\begin{proof}
For $\BT$, $\BT^\pre$, and $\BT^\ori$, the proof is analogous to \cite[Pr.3.2.2]{ec2}, and we omit the details. For $\bfG \mapsto \LocSys_{\bfG}$, this is \cite[Pr.6.2.6]{ec3}, combined with the fact that the forgetful functor $\Cat(\Glo_\ab)^\otimes_{S\nmid\pi\hyp\st} \to \Fun((\abglobalspaces)^\op,\Cat)$ preserves limits by \cref{prop:limits_of_stable}.
\end{proof}

In \cite[\textsection 2.1, \textsection 3.1-3.2]{reconstruction}, we introduced a category $\Stk$ of \emph{fpqc stacks} characterized by the following properties:

\begin{enumerate}
\item There is a fully faithful embedding $\Aff \subset \Stk$;
\item Restriction along $\Aff\subset\Stk$ induces an equivalence between limit preserving functors out of $\Stk^\op$ and functors out of $\CAlg$ that satisfy flat descent.
\end{enumerate}

\begin{rmk}
By \cite[Pr.3.2.2.2]{reconstruction}, there is a functor $h_{(\bs)}\colon \SpDMnc \to \Stk$ from the category of spectral Deligne--Mumford stacks with the property that any construction on $\SpDMnc$ which satisfies flat descent factors uniquely through $h_{(\bs)}$. In particular, in what follows, the reader could instead take ``stack'' to mean ``nonconnective spectral Deligne--Mumford stack'', provided ``limit preserving functor'' is interpreted to mean ``\'etale sheaf''.
\end{rmk}

\cref{prop:flatdescent} now formally implies the existence of unique limit preserving functors
\[
\BT,\BT^\pre,\BT^\ori_{S\nmid}\colon \Stk^\op \to \Cat
\]
naturally satisfying 
\[
\BT(\Spec \A) = \BT(\A).
\]
The Cartesian unstraightenings of these functors define categories
\[
\pdivstk,\qquad \pdivprestk,\qquad \pdivorstk_{S\nmid}
\]
of $\bfP$-divisible groups and (pre)oriented $\bfP$-divisible groups over arbitrary base stacks.

\begin{rmk}
The categories of ((pre)oriented) $\bfP$-divisible groups over a stack could also be defined directly. A $\bfP$-divisible group over a stack is a functor $\bfG[\bs]\colon \Ab_\fin^\op \to \Stk$ satisfying the same conditions as \cref{def:pdivisible}; a preorientation on this is an extension through $B\dual{(\bs)}\colon \Ab_\fin^\op\to\Glo_\ab$. Finally, a preorientation is an orientation if it restricts to an orientation over every map $\Spec \A \to \bfG[0]$.
\end{rmk}

\begin{ex}\label{ex:universalorientedec}
    An important example of a $\bfP$-divisible group over a nonaffine stack comes from the torsion object associated with the \emph{universal oriented elliptic curve} $\sfE^\ori$ over the \emph{moduli stack of oriented elliptic curves} $\M_\Ell^\ori$. We will revisit this in \Cref{ex:ambidextroustmf} and \cite{geometricnorms}.
\end{ex}

The existence of the category $\pdivorstk_{S\nmid}$ allows us to state the following.

\begin{theorem}\label{thm:stackyfunctoriality}
Let $S$ be a set of prime numbers. Then there is a unique functor
\[
\LocSys_{(\bs)}\colon (\pdiv^{\ori}_{S\nmid})^\op \to \Cat(\Glo_\ab)^\otimes_{\awayfromS\hyp\st}
\]
extending \cref{thm:extendedfunctoriality} whose restriction to the wide subcategory of Cartesian morphisms preserves limits.
\end{theorem}
\begin{proof}
Given \cref{prop:flatdescent}, this follows from \cite[Pr.2.1.2.10]{reconstruction}.
\end{proof}

\section{Equivariant homotopy theory of tempered cohomology}\label{sec:thetemperedglobalspectrum}

By applying the decategorification procedure of \cref{sec:parametrized} to the $\pi$-stable categories of tempered local systems of \Cref{sec:temperedlocalsystems}, we may associate to every oriented $\bfP$-divisible group a $\pi$-ambidextrous global spectrum representing its associated tempered cohomology theory; see \Cref{thm:temperedglobalspectrum}. Our goal in this section is to describe this process in more detail and to study the resulting global spectra.

Firstly, in \Cref{ssec:piambi_tempered_cohomology}, we put two and two together and refine Lurie's tempered cohomology theories to $\pi$-ambidextrous global $\E_\infty$ rings, as well as give a handful of canonical examples. In \Cref{ssec:sheaf_ofG}, we show how the functor $\O_\bfG$ of \Cref{main:pisheaf} may be viewed as a quasi-coherent sheaf of $\pi$-ambidextrous global $\E_\infty$ rings on the underlying stack $\sfM$. In \Cref{ssec:temperedHOMOLOGY}, we define \emph{tempered homology} and prove a general base change theorem. The sheafy considerations of the previous subsection allow us to define geometric objects $\bfG(\sfX)$ associated with certain ($\pi$-compact) global spaces $\sfX$ by way of a relative spectrum construction. This then leads to a geometric incarnation of geometric fixed points $\Phi^H \bfG$, which we define in \Cref{ssec:geometric_FP}. We then show that $\Phi^H \bfG = \emptyset$ if $H$ is nonabelian. Finally, in \Cref{ssec:abeliangeofixed}, we give a moduli interpretation of $\Phi^H \bfG$ for abelian $H$ as a \emph{stack of injections}, and prove that this stack satisfies computable blueshift properties.

\subsection{Tempered cohomology as a \texorpdfstring{$\pi$}{pi}-ambidextrous spectrum}\label{ssec:piambi_tempered_cohomology}

In \cite{ec3}, Lurie studies the following construction.

\begin{mydef}[{\cite[\textsection4]{ec3}}]
Let $\bfG$ be a preoriented $\bfP$-divisible group over a $\bfE_\infty$ ring $\A$. The \emph{tempered cohomology theory} associated with $\bfG$ is the unique limit preserving functor
\[
\A_\bfG^{(\bs)}\colon (\abglobalspaces)^\op\to\CAlg
\]
extending the preoriented $\bfP$-divisible group $\A_\bfG^{(\bs)}\colon \Glo_\ab^\op\to\CAlg_\A \to \CAlg$.
\end{mydef}

Taking limits and using the universal property of the category of stacks $\Stk$, this construction extended to preoriented $\bfP$-divisible groups over arbitrary base stacks in the sense of \cref{ssec:stacks}, leading to a definition of the tempered cohomology theory 
\[
\Gamma\calO_\bfG^{(\bs)}\colon (\abglobalspaces)^\op\to\CAlg
\]

associated with a preoriented $\bfP$-divisible group $\bfG$ over an arbitrary base stack $\sfM$. We will consider refinements of this construction taking values in categories of sheaves on $\sfM$ below.

The tempered cohomology theory associated with a preoriented $\bfP$-divisible group $\bfG$ is a \emph{naïve} global equivariant cohomology theory: it does not come equipped with any kind of transfers. For example, for a fixed group $G$, its restriction to the category of $G$-spaces need not be represented by a genuine $G$-spectrum. Our work in this paper allows us to precisely interpret Lurie's ambidexterity results as asserting that the tempered cohomology theory associated with an \emph{oriented} $\bfP$-divisible group carries significantly more structure.

\begin{theorem}\label{thm:temperedglobalspectrum}
Let $\bfG$ be an oriented $\bfP$-divisible group. Then the tempered cohomology theory
\[
\Gamma\calO_\bfG^{(\bs)}\colon (\abglobalspaces)^\op\to\CAlg
\]
is represented by a $\pi$-ambidextrous global spectrum
\[
\Gamma\ul{\calO}_\bfG \in \CAlg(\spectra^\gl_\pi).
\]
For any set of primes $S$, this assignment refines to a functor
\[
(\pdiv^{\ori}_{S\nmid})^\op \to \CAlg(\spectra^\gl_\awayfromS),\qquad \bfG \mapsto \Gamma\ul{\calO}_\bfG
\]
whose restriction to the wide subcategory of Cartesian morphisms preserves limits.
\end{theorem}
\begin{proof}
\cref{thm:stackyfunctoriality} provides, for any set of primes $S$, a functor
\[
\LocSys_{(\bs)}\colon (\pdiv^{\ori}_{S\nmid})^\op \to \Cat(\Glo_\ab)^\otimes_{\awayfromS\textup{-st}}
\]
whose restriction to the wide subcategory of Cartesian morphisms preserves limits, and \cref{thm:generaldecategorification} provides a limit preserving functor
\[
\Gamma\colon \Cat(\Glo_\ab)^\otimes_{\awayfromS\textup{-st}} \to \CAlg(\Sp^\gl_\awayfromS).
\]
By construction, we may identify 
\[
(\Gamma(\LocSys_{\bfG}))^T \simeq \Gamma\calO_\bfG^T
\]
for any $T \in \Glo_\ab$, see also \cite[Ex.7.1.5]{ec3}, and so the composite $\Gamma\circ \LocSys$ provides the desired refinement.
\end{proof}

We can interpret \cref{thm:temperedglobalspectrum} as saying that the choice of an oriented $\bfP$-divisible group $\bfG$ over a stack $\sfM$ provides a lift of the global sections $\Gamma(\calO_\sfM) \in \CAlg$ of $\sfM$ from an $\bfE_\infty$ ring to a $\pi$-ambidextrous $\bfE_\infty$ ring $\Gamma(\ul{\calO}_\bfG) \in \CAlg(\spectra^\gl_\pi)$. 

Let us now expand upon the examples from the introduction.

\begin{ex}
Let $\A$ be a complex-periodic $K(n)$-local $\bfE_\infty$ ring. The Quillen $\bfP$-divisible group $\bfG_\A^\calQ$ over $\A$ is oriented, providing a lift of $\A$ to a $\pi$-ambidextrous global $\bfE_\infty$ ring
\[
\ul{\A} \coloneqq \ul{\A}_{\bfG_\A^\calQ} \in \CAlg(\Sp^\gl_\pi).
\]
By \cite[Th.4.2.5]{ec3}, this theory is Borel in the sense that the Atiyah--Segal comparison map
\[
\ul{\A}^{\sfX} \to \A^{|\sfX|}
\]
is an equivalence for any global space $\sfX$, where $|\sfX| = \sfX(\ast)$ is the underlying space of $\sfX$. The $\pi$-ambidextrous global structure of $\ul{\A}$ therefore encodes the association, to any map $f\colon Y \to X$ of spaces with $\pi$-finite fibers, the transfer $f_!\colon \A^{Y} \to \A^{X}$ constructed originally by Hopkins and Lurie \cite{ambi}.
\end{ex}

\begin{ex}\label{ex:complexktheoryexample}
In \cite[Th.6.5.1]{ec2}, Lurie proves that complex $K$-theory arises naturally from the multiplicative $\bfP$-divisible group $\mu_{\bfP^\infty}$. The multiplicative $\bfP$-divisible group $\mu_{\bfP^\infty}$ is canonically oriented over $\KU$, and in \cite[\textsection 3.6]{ec3} Lurie explains how the associated tempered cohomology theory may be identified as equivariant complex $K$-theory. Applying \cref{thm:temperedglobalspectrum}, it follows that complex $K$-theory refines to a $\pi$-ambidextrous global $\E_\infty$ ring
\[
\gKU \coloneqq \ul{\KU}_{\mu_{\bfP^\infty}} \in \CAlg(\Sp^\gl_\pi).
\]
For any $\pi$-finite space $F$, the truncation $F \to \tau_{\leq 1}F$ induces an equivalence
\[
\gKU^{\tau_{\leq 1}F} \simeq \gKU^{F}.
\]
This implies that the values of $\gKU$ on the $\pi$-finite spaces are determined by its values on the finite groups, which are determined by the underlying global spectrum of $\gKU$. Our construction of $\gKU$ as a $\pi$-global spectrum refines this by introducing the additional structure of \emph{deflations}, or transfers along surjections. If $f\colon K \to H$ is an arbitrary homomorphism of finite groups, then using character theory, one can show that the transfer
\[
f_! \colon \gKU^{\BK} \to \gKU^{\BH},
\]
on $\pi_0$, may be identified as the induction map
\[
\mathbf{C}[H]\otimes_{\mathbf{C}[K]}(\bs)\colon RU(K) \to RU(H)
\]
on complex representation rings; this is shown for injective homomorphisms in \cite[Pr.7.6.7]{ec3}, and the same proof applies in general. In other words, the fact that equivariant $K$-theory lifts to a $\pi$-ambidextrous spectrum is a categorification of the fact that representation rings admit transfer maps along arbitrary group homomorphisms.
\end{ex}

Our extension of tempered cohomology to oriented $\bfP$-divisible groups over non-affine base stacks allows us to extend this example to real topological $K$-theory.

\begin{ex}\label{ex:realktheoryexample}
     Let $\sfM_\Tori^\ori$ be the moduli stack of \emph{oriented tori}; in \cite[\textsection A]{geometricnorms}, we will show this can be written as the quotient $\Spec \KU /C_2$, where $C_2$ acts on $\KU$ by the usual complex conjugation action. The universal oriented torus $\G_m$ over $\sfM_\Tori^\ori$ has associated oriented $\bfP$-divisible group $\mu_{\bfP^\infty}$. Applying \Cref{thm:temperedglobalspectrum}, we obtain a $\pi$-ambidextrous global $\E_\infty$ ring spectrum
     \[\gKO \coloneqq \Ga(\ul{\O}_{\mu_{\bfP^\infty}/\sfM_\Tori^\ori}) \in \CAlg(\Sp^\gl_\pi).\]
     The oriented torus $\G_m$ over $\KU$ defines a map $\gKO \to \gKU$ of $\pi$-ambidextrous global spectra, which by construction and the above identification as a quotient stack yields an equivalence
     \[
     \gKO\simeq \gKU^{\h C_2}
     \]
of $\pi$-ambidextrous global $\E_\infty$ rings.
\end{ex}

The same story goes through for equivariant elliptic cohomology theories constructed out of the \emph{Katz--Mazur group scheme}.

\begin{ex}\label{ex:tatektheory}
Let $\T_\KM$ be the \emph{oriented Katz--Mazur $\P$-divisible group} over $\Spec \KU[q^\pm]$ defined in \cite[Th.4.1]{globaltate}; here $q$ lives in degree $0$. Applying \Cref{thm:temperedglobalspectrum} now produces a $\pi$-ambidextrous global $\E_\infty$ ring
\[\gKU_\KM \coloneqq \Ga(\ul{\O}_{\T_\KM}) \in \CAlg(\Sp^\gl_\pi),\]
a $\pi$-ambidextrous analogue of \emph{quasi-elliptic cohomology} \cite{huanquasielliptic}; by \cite[Th.4.1]{globaltate}, these theories have the same values on $\bfB H$ for finite abelian $H$ after taking homotopy groups. Apply \Cref{thm:temperedglobalspectrum} to the $\bfP$-divisible group associated with the \emph{oriented Tate curve} $\T$ over $\KU\llpar q\rrpar$ of \cite[Th.5.1]{globaltate} yields the $\pi$-ambidextrous global $\E_\infty$ ring
\[\gKU_\Tate \coloneqq \Ga(\ul{\O}_{\T[\bfP^\infty]}) \in \CAlg(\Sp^\gl_\pi).\]
Copying the argument of Theorem 7.14 of \emph{loc.\ cit.\ }shows that the natural map
\[\KU\llpar q\rrpar \otimes_{\KU[q^\pm]} \gKU_\KM \xrightarrow{\simeq} \gKU_\Tate\]
is an equivalence. The true advantage of $\gKU_\Tate$ comes from its extension to compact Lie groups; see Definition 7.5 of \emph{loc.\ cit}. Keep in mind though, the natural map
\[\KU[q^\pm] \otimes_{\KU} \gKU \to \gKU_\KM\]
is not an equivalence in any sense; only when evaluated at the trivial group does it become an equivalence, see Warning 7.6 of \emph{loc.\ cit}. One can also further take $C_2$-quotients by the natural complex-conjugation action on $\KU[q^\pm]$ and $\KU\llpar q \rrpar$, which lead to \emph{real} notions of these aforementioned $\pi$-ambidextrous global $\E_\infty$ ring $\gKO_\KM$ and $\gKO_\Tate$, see Section 6 of \emph{loc.\ cit}. 
\end{ex}

Finally, we have the following example obtained from the universal oriented elliptic curve.

\begin{ex}\label{ex:ambidextroustmf}
     Let $\calE$ be the universal oriented elliptic curve over the moduli stack of oriented elliptic curves $\M_\Ell^\ori$, as defined in \Cref{ex:universalorientedec}. By \cite[Con.2.9.6]{ec3} (also see \cite[Pr.6.1]{elltempcomp}), the torsion subgroup $\calE[\P^\infty]\subset\calE$ defines an oriented $\P$-divisible group over $\M_\Ell^\ori$. Applying \Cref{thm:temperedglobalspectrum}, we obtain a $\pi$-ambidextrous global $\E_\infty$ ring
     \[
     \gTMF \coloneqq \Ga\ul{\calO}_{\calE[\P^\infty]} \in \CAlg(\Sp^\gl_\pi)
     \]
refining the underlying spectrum $\TMF$ of topological modular forms. By \cite{elltempcomp}, the underlying naïve global cohomology theory on $\abglobalspaces$ associated with $\gTMF$ agrees with that underlying the equivariant topological modular forms constructed by Gepner--Meier \cite{davidandlennart}. It seems reasonable to expect that $\gTMF$ has the same underlying $\ab$-global spectrum as the global spectrum of topological modular forms constructed in \cite{gepner2024global2ringsgenuinerefinements}, but the constructions are substantially different and we do not yet have such a comparison.
\end{ex}

\begin{ex}
    By definition, the map $\T \colon \Spec \KU\llpar q \rrpar / C_2 \to \M_\Ell^\ori$ defining the oriented Tate curve induces a Cartesian map between the oriented $\bfP$-divisible groups associated with $\T$ and $\calE$. In particular, we have a \emph{$\pi$-ambidextrous global $q$-expansion map}
    \[\gTMF \to \gKO\llpar q \rrpar\]
    in $\CAlg(\Sp_\pi^\gl)$, compatible with all transfers along relatively $\pi$-finite morphisms of global spaces.
\end{ex}

Let us now expand upon the Adams operations example of \Cref{intro_ex:adams_on_KU}.

\begin{ex}\label{ex:AO_on_KO}
    Fix an integer $\ell$ and write $\M_{\Tori,\Sph[1/\ell]}^\ori = \M_\Tori^\ori \times \Spec \Sph[1/\ell]$. Recall that a torus over a stack $\sfM$ is a relatively affine and flat group scheme $\bfG$ over $\sfM$ which is fpqc-locally isomorphism to the multiplicative group scheme $\bfG_m$. From the definition of $\M_\Tori^\ori$ as the moduli stack of oriented tori, we have
    \[\Map_{\Stk}(\M_{\Tori,\Sph[1/\ell]}^\ori, \M_\Tori^\ori ) \simeq \Tori^\ori(\M_{\Tori,\Sph[1/\ell]}^\ori).\]
    The pair $(\G_m,e)$ of the torus $\G_m$ over $\M_{\Tori,\Sph[1/\ell]}^\ori$ and its tautological orientation $e$ gives the open immersion $\M_{\Tori,\Sph[1/\ell]}^\ori \to \M_{\Tori}^\ori$ by inverting $\ell$. The pair $(\G_m, [\ell](e))$, where $[\ell]$ is multiplication by $\ell$ on $\G_m$, hence also on its formal group $\widehat{\G}_m$, corresponds to a map of stacks
    \[\psi^\ell \colon \M_{\Tori,\Sph[1/\ell]}^\ori \to \M_{\Tori}^\ori\]
    which we call the \emph{$\ell$th Adams operation} on $\M_{\Tori,\Sph[1/\ell]}^\ori$. This name is justified as by Snaith's theorem \cite[\textsection 6.5]{ec2}, on the cover $\Spec\KU[1/\ell]$ this agrees with the map $[\ell]\colon \Sigma^\infty_+ BU(1)[\beta^{-1}] \to \Sigma^\infty_+BU(1)[\ell^{-1},\beta^{-1}]$ induced by multiplication by $\ell$ on the abelian group $BU(1)$, which is a model for the Adams operation $\psi^\ell\colon \KU \to \KU[1/\ell]$.

    By construction, the pair
    \begin{equation}\label{eq:map_inducing_AOonKO}(\psi^\ell, [\ell]) \colon (\sfM_{\Tori,\Sph[1/\ell]}^\ori, \mu_{\bfP^\infty}) \to (\sfM_{\Tori}^\ori, \mu_{\bfP^\infty}),\end{equation}
    where $\mu_{\bfP^\infty} \subset \G_m$ is the $\bfP$-torsion of the universal oriented torus $\G_m$, is a map of oriented $\bfP$-divisible groups. Indeed, this is equivalent to
    \[[\ell] \colon (\mu_{\bfP^\infty}, e) \to (\mu_{\bfP^\infty}, [\ell](e)) \simeq (\psi^\ell)^\ast(\mu_{\bfP^\infty}, e)\]
    being a map of oriented $\bfP$-divisible groups over $\sfM_{\Tori,\Sph[1/\ell]}^\ori$, which follows from the definition of $\psi^\ell$. In particular, by applying \Cref{thm:temperedglobalspectrum} to (\ref{eq:map_inducing_AOonKO}) we obtain maps of $\awayfromS$-ambidextrous global $\E_\infty$ rings
    \[\psi^\ell \colon \gKO_{\awayfromS} \to \gKO_{\awayfromS}[1/\ell]\]
    for all sets of primes $S$ containing the divisors of $\ell$ refining the usual Adams operations on $\KO$.
\end{ex}

Similar arguments used in \Cref{ex:AO_on_KO} also apply to $\M_\Ell^\ori$ and topological modular forms.

\begin{ex}\label{ex:AO_on_TMF}
    Writing $\M_{\Ell,\Sph[1/\ell]}^\ori = \M_{\Ell}^\ori \times \Spec \Sph[1/\ell]$, we also have
    \[\Map_{\Stk}(\M_{\Ell,\Sph[1/\ell]}^\ori, \M_{\Ell}^\ori) \simeq \Ell^\ori(\M_{\Ell,\Sph[1/\ell]}^\ori).\]
    The twist $(\calE, [\ell](e))$ of the universal orientation of $\calE$ over $\M_{\Ell,\Sph[1/\ell]}^\ori$ by $[\ell]$ corresponds to a map of stacks
    \[\psi^\ell \colon \M_{\Ell,\Sph[1/\ell]}^\ori \to \M_{\Ell}^\ori,\]
    from which we obtain a map of oriented $\bfP$-divisible groups
    \[(\psi^\ell, [\ell]) \colon (\M_{\Ell,\Sph[1/\ell]}^\ori, \calE[\bfP^\infty]) \to (\M_{\Ell}^\ori, \calE[\bfP^\infty]).\]
    Applying \Cref{thm:temperedglobalspectrum} to this map yields a map of $\awayfromS$-ambidextrous global $\E_\infty$ rings
    \[\psi^\ell \colon \gTMF_{\awayfromS} \to \gTMF_{\awayfromS}[1/\ell]\]
    for each $S$ containing all divisors of $\ell$.
\end{ex}

\begin{remark}
We do not directly compare the Adams operation $\psi^\ell\colon \TMF\to\TMF[1/\ell]$ considered in \cref{ex:AO_on_TMF} with those considered in \cite{heckeontmf,luriestheorem}, although we expect that they agree. As the two constructions agree in $\h\Sp$ on all even periodic affine covers of $\M_\Ell^\ori$, they are computationally indistinguishable from the perspective of the descent spectral sequence, as used in \cite{adamsontmf}. In particular, either construction could be used for applications such as in \cite{heighttwojat3,v232_periodic_families}.
\end{remark}

More examples of maps of oriented $\bfP$-divisible groups and their associated maps of ambidextrous global spectra will appear in \cite{geometricnorms}. However, in what follows, we will primarily be concerned with properties of the $\pi$-global theories associated with a single oriented $\bfP$-divisible group, and so will only make use of naturality in Cartesian morphisms, i.e.\ of the functor
$
(\pdiv^{\ori}_{\cart})^\op \to \CAlg(\Sp^\gl_\pi).
$

\subsection{The structure sheaf of an oriented \texorpdfstring{$\bfP$}{P}-divisible group}\label{ssec:sheaf_ofG}

Fix a stack $\sfM$ and an oriented $\bfP$-divisible group $\bfG$ on $\sfM$. \cref{thm:temperedglobalspectrum} provides from this data a $\pi$-ambidextrous ring $\Gamma\ul{\calO}_\bfG \in \CAlg(\Sp^\gl_\pi)$ refining the tempered cohomology theory associated with $\bfG$. Our goal in this section is describe the sense in which a ``sheaf $\ul{\calO}_\bfG$ on $\sfM$'' exists before passage to global sections. We first recall some basic notions regarding sheaves.

\begin{mydef}
Let $\dcat$ be a complete $\infty$-category. A \emph{sheaf on $\sfM$ with values in $\dcat$} is a functor
\[
\Aff_{/\sfM}^\op\to\dcat
\]
which satisfies flat descent. These assemble into the full subcategory
\[
\Shv(\sfM;\dcat)\subset\Fun(\Aff_{/\sfM}^\op,\dcat).
\]
\end{mydef}

\begin{rmk}
By \cite[Pr.2.1.2.10]{reconstruction}, Kan extension determines an equivalence between sheaves on $\sfM$ and limit preserving presheaves on $\Stk_{/\sfM}^\op$, i.e.\ identifies
\[
\Shv(\sfM;\dcat)\subset\Fun(\Stk_{/\sfM}^\op,\dcat)
\]
with the full subcategory of limit preserving functors. We will freely translate between these two identifications of $\Shv(\sfM;\dcat)$.
\end{rmk}

\begin{ex}\label{ex:structuresheaf}
The \emph{structure sheaf} of a stack $\sfM$ is the sheaf of $\bfE_\infty$ rings given by the forgetful functor
\[
\calO_\sfM\colon \Aff_{/\sfM}^\op\to\CAlg,\qquad \calO_\sfM(\Spec\A\to\sfM) = \A.
\]
Its extension to stacks over $\sfM$ is, by construction, the functor 
\[
\Gamma(\calO_{(\bs)})\colon \Stk_{/\sfM}^\op\to\CAlg
\]
sending a stack over $\sfM$ to its ring of global sections.
\end{ex}

\begin{ex}
The composite
\[
\Aff_{/\sfM}^\op \to \CAlg\to \CAlg(\PrLst),\qquad (\Spec\A\to\sfM) \mapsto \Mod_\A
\]
postcomposing the structure sheaf of $\sfM$ with $\Mod_{(\bs)}\colon \CAlg\to\CAlg(\PrLst)$ satisfies flat descent by \cite[Cor.D.6.3.3]{sag}, and therefore determines a sheaf on $\sfM$ with values in $\CAlg(\PrLst)$. Its extension to stacks over $\sfM$ is, by construction, the functor
\[
\QCoh(\bs)\colon \Stk_{/\sfM}^\op\to\CAlg(\PrLst)
\]
sending a stack over $\sfM$ to its category of quasi-coherent sheaves; see \cite[Df.3.1.2.2]{reconstruction}.
\end{ex}

By construction, the structure sheaf $\calO_\sfM$ lifts to an object of 
\[
\Shv(\sfM;\CAlg)\subset\Fun(\Aff_{/\sfM}^\op,\CAlg)\simeq\CAlg(\Fun(\Aff_{/\sfM}^\op,\Sp)),
\]
where $\Fun(\Aff_{/\sfM}^\op,\Sp)$ is equipped with its pointwise symmetric monoidal structure. For size reasons, $\Shv(\Aff_{/\sfM}^\op,\Sp)$ may fail to be a localization of $\Fun(\Aff_{/\sfM}^\op,\Sp)$, and so may not inherit the structure of a symmetric monoidal category. For our purposes, this issue may be sidestepped by considering $\Shv(\Aff_{/\sfM},\Sp)$ as the underlying category of the full suboperad $\Shv(\sfM,\Sp)^\otimes\subset\Fun(\Aff_{/\sfM}^\op,\Sp)^\otimes$ it generates. With this convention, we have
\[
\Shv(\sfM;\CAlg)\simeq\CAlg(\Shv(\sfM;\Sp)),
\]
and may identify
\[
\Mod_{\calO_\sfM}(\Shv(\sfM;\Sp)) = \Mod_{\calO_\sfM}(\Fun(\Aff_{/\sfM}^\op,\Sp))\times_{\Fun(\Aff_{/\sfM}^\op,\Sp)}\Shv(\sfM;\Sp)
\]
as the full subcategory of $\Mod_{\calO_\sfM}(\Fun(\Aff_{/\sfM}^\op,\Sp))$ spanned by those $\calO_\sfM$-modules whose underlying presheaf satisfies flat descent.

\begin{lemma}\label{lem:qcohassheaves}
There is a fully faithful and colimit preserving functor
\[
\QCoh(\sfM) \subset \Mod_{\calO_\sfM}(\Shv(\sfM;\Sp))
\]
realizing $\QCoh(\sfM)$ as the full subcategory of $\calO_\sfM$-modules $\fmodule$ with the further property that, for any morphism
\begin{center}\begin{tikzcd}
\Spec\B\ar[rr,"f"]\ar[dr,"p"']&&\Spec\A\ar[dl,"q"]\\
&\sfM
\end{tikzcd}\end{center}
in $\Aff_{/\sfM}$, the base change map $\B\otimes_\A \fmodule(q) \to \fmodule(p)$ is an equivalence.
\end{lemma}

\begin{proof}
Write $\Mod$ for the total space of the Cartesian fibration associated with the functor $\QCoh\colon\Aff\to\CAlg(\PrLst)$. By \cite[Th.5.10]{denissilluca}, $\Mod_{\calO_\sfM}(\Fun(\Aff_{/\sfM}^\op,\Sp))$ may be identified with the category of sections of $\Mod$ over $\Aff_{/\sfM}$:
\[
\Mod_{\calO_\sfM}(\Fun(\Aff_{/\sfM}^\op,\Sp))\simeq \left\{\begin{tikzcd}
&\Mod\ar[d]\\
\Aff_{/\sfM}\ar[r]\ar[ur,dashed]&\Aff
\end{tikzcd}\right\}.
\]
On the other hand, by definition $\QCoh(\sfM)$ is the limit of the functor 
\[
\QCoh(\bs)\colon \Aff_{/\sfM}^\op \to \CAlg(\PrLst),
\]
which may be identified with the full subcategory of \emph{Cartesian} sections. This realizes $\QCoh(\sfM)\subset\Mod_{\calO_\sfM}(\Fun(\Aff_{/\sfM}^\op,\Sp))$ as the full subcategory of $\calO_\sfM$-modules satisfying the given base change condition. If $\fmodule \in \QCoh(\sfM)$ and $\Spec(\B)\to\Spec(\A)$ is a faithfully flat cover over $\sfM$, then quasicoherence implies that
\[
\Tot\fmodule(\Spec(\B^{\otimes_\sfA \bullet+1}))\simeq \Tot \left(\B^{\otimes_\sfA\bullet+1}\otimes_\sfA\fmodule(\Spec\sfA)\right)\simeq\fmodule(\Spec\sfA),
\]
guaranteeing that $\fmodule$ is a sheaf as claimed.
\end{proof}

We now specialize to tempered cohomology.

\begin{prop}\label{prop:temperedsheaf}
The oriented $\bfP$-divisible group $\bfG$ over $\sfM$ refines the structure sheaf of $\sfM$ to a sheaf $\ul{\calO}_\bfG \in \Shv(\sfM; \CAlg(\Sp_\pi^\gl))$ of $\pi$-ambidextrous global $\E_\infty$ rings.
\end{prop}
\begin{proof}
If $S$ is the set of all primes, then $\pdiv^{\ori}_{S\nmid}\simeq\pdiv^{\ori}_{\cart}\subset\pdiv^\ori$ is the subcategory of $\bfP$-divisible groups and Cartesian morphisms. In particular, there is an equivalence
\[
\Stk_{/\sfM}\simeq(\pdiv^\ori_\cart)_{/\bfG}.
\]
Now the composite
\[
\Stk_{/\sfM}^\op\simeq(\pdiv^\ori_\cart)_{/\bfG}^\op \to (\pdiv^\ori_\cart)^\op \xrightarrow{\Gamma\ul{\calO}^{(\bs)}}\CAlg(\Sp^\gl_\pi)
\]
is the desired sheaf of $\pi$-ambidextrous global $\E_\infty$ rings on $\sfM$.
\end{proof}

Recall from \cref{prop:qcommutativeglobalringsaslaxfunctors} that a $\pi$-commutative global $\E_\infty$ ring in a symmetric monoidal category $\ccat$ is equivalent to a lax symmetric monoidal functor $\Span_\pi(\abglobalspaces)\to\ccat$ whose restriction to the backward maps preserves limits. This definition makes sense more generally when $\ccat$ is allowed to be an arbitrary operad. By adjunction, we may identify
\begin{align*}
\Fun(\Aff_{/\sfM}^\op,\CAlg(\Sp^\gl_\pi))&\simeq\CAlg^\gl_\pi(\Fun(\Aff_{/\sfM}^\op,\Sp^\gl_\pi)) \\
&\subset \Fun^{\otimes\textup{-lax}}(\Span_\pi(\abglobalspaces),\Mod_{\calO_\sfM}(\Fun(\Aff_{/\sfM}^\op,\Sp))).
\end{align*}
As $\Shv(\sfM;\Sp)\subset\Fun(\Aff_{/\sfM}^\op,\Sp)$ is closed under limits this restricts to an equivalence
\[
\Shv(\sfM;\CAlg(\Sp^\gl_\pi))\simeq\CAlg^\gl_\pi(\Shv(\sfM;\Sp)) \subset\Fun^{\otimes\textup{-lax}}(\Span_\pi(\abglobalspaces),\Mod_{\calO_\sfM}(\Shv(\sfM;\Sp))).
\]

\begin{notation}
Given a $\bfP$-divisible group $\bfG$ over a stack $\sfM$, we write
\[
\ul{\calO}_\bfG \in \CAlg^\gl_\pi(\Mod_{\calO_\sfM}(\Shv(\sfM;\Sp))) \subset \Fun^{\otimes\textup{-lax}}(\Span_\pi(\abglobalspaces),\Mod_{\calO_\sfM}(\Shv(\sfM;\Sp))).
\]
for the sheaf of $\pi$-commutative global $\calO_\sfM$-algebras lifting $\calO_\bfG$ via \cref{prop:temperedsheaf} and \cref{prop:lifts_to_modules}. 
\end{notation}

\subsection{Tempered (co)homology and base change}\label{ssec:temperedHOMOLOGY}

Let $\bfG$ be an oriented $\bfP$-divisible group over a stack $\sfM$. Given $\sfX \in \abglobalspaces$, one may evaluate the $\pi$-commutative ring $\ul{\calO}_\bfG$ on $\sfX$ to obtain a sheaf of $\calO_\sfM$-algebras
\[
\ul{\calO}_\bfG^\sfX \in \CAlg(\Mod_{\calO_\sfM}(\Shv(\sfM;\Sp))),\qquad \ul{\calO}_\bfG^\sfX(f\colon \Spec\A\to\sfM) = \A^\sfX_{f^\ast\bfG}.
\]
Our goal in this subsection is to explain how Lurie's base change theorem \cite[Th.4.7.1]{ec3} identifies conditions under which this sheaf is quasi-coherent, and to use this to define well-behaved \emph{tempered homology} theories.

\begin{mydef}
A global space $\sfX$ is \emph{$\pi$-compact} if it lies in the full subcategory of $\abglobalspaces$ generated by the $\pi$-finite spaces under finite colimits and retracts.
\end{mydef}

\begin{ex}
Let $G$ be a finite group and $X$ be a compact $G$-space. Then the global space $X//G$ associated with $X$ is in the full subcategory of $\abglobalspaces$ generated by $\bfB H$ for $H \subset G$ under finite colimits and retracts. In particular, $X//G$ is $\pi$-compact.
\end{ex}

\begin{prop}\label{pr:quasicoherent_on_picompactspaces}
Let $\sfX$ be a $\pi$-compact global space. Then $\ul{\calO}_\bfG^\sfX$ is a quasi-coherent $\calO_\sfM$-module.
\end{prop}
\begin{proof}
By working locally on $\sfX$, it suffices to prove that if $f\colon \A\to\B$ is a map of $\bfE_\infty$ rings and $\bfG$ is an oriented $\bfP$-divisible group over $\A$, then the comparison map
\[
\B\otimes_\A\A_\bfG^\sfX \to \B_\bfG^\sfX
\]
is an equivalence. The class of $\sfX$ for which this map is an equivalence is closed under finite colimits and retracts, so we may reduce to the case where $\sfX = F$ is a $\pi$-finite space, which is exactly Lurie's base change theorem \cite[Th.4.7.1]{ec3}.
\end{proof}

Thus, a $\pi$-compact space $\sfX$ determines 
\[
\ul{\calO}_\bfG^\sfX \in \CAlg(\QCoh(\sfM)) \subset \CAlg(\Mod_{\O_{\sfM}}(\Shv(\sfM; \Sp))).
\]
By \cite[Th.2.2.3.1]{reconstruction}, taking relative spectra determines an equivalence
\[
\CAlg(\QCoh(\sfM))^\op\simeq\Stk_{/\sfM}^\aff,
\]
where $\Stk_{/\sfM}^\aff\subset\Stk_{/\sfM}$ is the full subcategory spanned by the affine morphisms. \cref{pr:quasicoherent_on_picompactspaces} therefore gives access to the following geometric objects.

\begin{mydef}\label{def:picompactstack}
Let $\sfX$ be a $\pi$-compact global space. We write
\[
\bfG(\sfX) \in \Stk_{/\sfM}
\]
for the relatively affine stack over $\sfM$ associated with $\ul{\calO}_\bfG^\sfX \in \CAlg(\QCoh(\sfM)) \simeq (\Stk^\aff_{/\sfM})^\op$.
\end{mydef}

\begin{ex}
As $T$ varies over $\Glo_\ab\subset\abglobalspaces$, the stacks $\bfG(T)$ over $\sfM$ are exactly those stacks which constitute the oriented $\bfP$-divisible group $\bfG$, i.e.\ \cref{def:picompactstack} is compatible with \cref{notation:preorientednotation} under the inclusion $\Aff \subset \Stk$.
\end{ex}

These stacks are the natural recipients for equivariant homology theories associated with oriented $\bfP$-divisible groups.

\begin{construction}\label{constr:temperedhomology}
Given a $\pi$-finite space $F$, we construct a lax symmetric monoidal and colimit preserving \emph{$F$-global $\bfG$-tempered homology} functor
\[
\calO_\bfG^F[\bs]\colon (\abglobalspaces)_{/F} \to \QCoh(\bfG(F))
\]
as follows. By \cref{def:homology}, associated with the $\pi$-ambidextrous $\bfE_\infty$ ring $\ul{\calO}_\bfG$ is a lax symmetric monoidal homology functor
\[
\calO_\bfG^F[\bs]\colon (\abglobalspaces)_{/F} \to \Mod_{\calO_\bfG^F}(\Mod_{\calO_\sfM}(\Shv(\sfM;\Sp))),
\]
satisfying $\calO_\bfG^F[T]\simeq \calO_\bfG^T$ for any map $T \to F$ with $T \in \Glo_\ab$. In particular, we have $\calO_{\bfG}^F[T]\in \Mod_{\calO_\bfG^F}(\QCoh(\sfM))$. By \cref{prop:homology_colim} this functor preserves colimits, and therefore by \cref{lem:qcohassheaves} and the above observation lands in the full subcategory
\[
\QCoh(\bfG(F))\simeq\Mod_{\calO_\bfG^F}(\QCoh(\sfM))\subset \Mod_{\calO_\bfG^F}(\Mod_{\calO_\sfM}(\Shv(\sfM;\Sp)))
\]
of quasi-coherent sheaves on $\bfG(F)$.
\end{construction}

The name \emph{$F$-global homology} is in analogy to $G$-global homotopy theory for a finite group $G$, as $G$-global spaces, in the sense of \cite{lenz}, are precisely objects of the slice $(\Spc_{\gl})_{/\BG}$.

\begin{ex}
Let $K \to F$ be a map of $\pi$-finite spaces. Restriction defines a map $\calO_\bfG^F \to \calO_\bfG^K$ in $\CAlg(\QCoh(\sfM))$. This promotes $\calO_\bfG^K$ to an object of $\Mod_{\calO_\bfG^F}(\QCoh(\sfM))\simeq\QCoh(\bfG(F))$, and by construction we have
\[
\calO_\bfG^F[K]\simeq\calO_\bfG^K \in \QCoh(\bfG(F)).
\]
Functoriality of $\calO_\bfG^F[\bs]$ associates to any map $f\colon K \to K'$ of $\pi$-finite spaces over $F$ the $\calO_\bfG^F$-linear transfer map $f_!\colon \calO_\bfG^K \to \calO_\bfG^{K'}$. When $K' = F$, this linearity is exactly \emph{Frobenius reciprocity}.
\end{ex}

\begin{rmk}\label{rmk:0semiaffine}
By instead applying \cref{def:homology} to the $\pi$-ambidextrous spectrum $\Gamma(\ul{\calO}_\bfG)$ of global sections of $\ul{\calO}_\bfG$, one obtains a lax symmetric monoidal $F$-global homology functor
\[
\Gamma(\ul{\calO}_\bfG^F) = \mathsf{D}\Gamma(\ul{\calO}_\bfG) \colon (\abglobalspaces)_{/F} \to \Mod_{\Gamma(\calO_\bfG^F)}.
\]
This satisfies $\Gamma(\ul{\calO}_\bfG^F)[T]\simeq \Gamma(\ul{\calO}_\bfG^F[T])$ for $T \in (\Glo_\ab)_{/F}$ by construction, and so there is a natural coassembly map
\[
\Gamma(\ul{\calO}_\bfG^F)[\sfX] \to \Gamma(\ul{\calO}_\bfG^F[\sfX])
\]
for $\sfX \in (\abglobalspaces)_{/F}$ which is an equivalence if $\Gamma\colon \QCoh(\sfM) \to \Sp$ preserves colimits, i.e.\ if $\sfM$ is \emph{$0$-semiaffine} in the sense of \cite[Df.2.2.1.5]{reconstruction}; see \cite[\textsection4.2.3 \& 4.4]{reconstruction} for examples. In particular, this holds for the primary examples $\sfM = \sfM_\Tori^\ori$ and $\sfM = \sfM_\Ell^\ori$ of \Cref{ex:realktheoryexample,ex:ambidextroustmf}.
\end{rmk}

\begin{ex}\label{ex:gequivhomology}
Let $H$ be a finite group. From the underlying $H$-spectrum object $\ul{\calO}_{\bfG,H}$ in $\QCoh(\sfM)$ of $\ul{\calO}_{\bfG}$, as defined in \cref{restricted_G-spectrum}, we may extract a $H$-equivariant homology theory
\[
\spaces_H \to \Mod_{\calO_\sfM}(\Shv(\sfM;\Sp)),\qquad X \mapsto (\ul{\calO}_{\bfG,H} \otimes \Sigma^\infty X)^H.
\]
This is lax symmetric monoidal and sends the point to $\calO_{\bfG}^{\bfB H}$, and therefore lifts to a functor
\[
\calO_{\bfG,H}[\bs//H]\colon \spaces_H \to \QCoh(\bfG(\bfB H)).
\]
By inspection, this is equivalent to the restriction of \cref{constr:temperedhomology} along the functor
\[
\spaces_H \to (\abglobalspaces)_{/\bfB H}
\]
sending an orbit $H/K$ to the map $\bfB K \to \bfB H$.
As $\QCoh(\bfG(\bfB H))$ is pointed, this also factors through points $H$-spaces to define a lax symmetric monoidal functor
\[
\widetilde{\calO}_{\bfG,H}[\bs//H]\colon \spaces_{H,\ast} \to \QCoh(\bfG(\bfB H)).
\]
\end{ex}

As originally hinted at by Lurie in \cite[Rmk.3.14]{lurieecsurvey}, tempered homology is in many ways better behaved than tempered cohomology when working at the level of sheaves. For example, it satisfies the following general base change theorem.

\begin{theorem}\label{thm:homologybasechange}
Let $f\colon \sfN \to \sfM$ be a morphism of stacks, $\bfG$ be an oriented $\bfP$-divisible group over $\sfM$ and $F$ be a $\pi$-finite space. We will, slightly abusively, still write $f$ for the morphism $ f^\ast \bfG(F) \to \bfG(F)$ induced by $f$. Then there is a natural equivalence
\[
f^\ast\calO_\bfG^F[\sfX] \to \calO_{f^\ast\bfG}^F[\sfX]
\]
in $\QCoh(f^\ast\bfG(F))$ for any $\sfX \in (\abglobalspaces)_{/F}$.
\end{theorem}
\begin{proof}
As both sides preserve colimits in $\sfX$, it suffices to construct this natural equivalence in the case where $\sfX = T \in (\Glo_\ab)_{/F}$. Here we have
\[
f^\ast\calO_\bfG^F[T]\simeq f^\ast\calO_\bfG^T\simeq\calO_{f^\ast\bfG}^T\simeq\calO_{f^\ast\bfG}^F[T]
\]
by construction.
\end{proof}

\subsection{Geometric fixed points}\label{ssec:geometric_FP}

Fix an oriented $\bfP$-divisible group $\bfG$ on a stack $\sfM$. Following \cref{restricted_G-spectrum}, for each finite group $H$ we may restrict $\ul{\calO}_{\bfG}$ to a sheaf
\[
\ul{\calO}_{\bfG,H} \in \CAlg_H(\Mod_{\calO_\sfM}(\Shv(\sfM;\Sp)))
\]
of $H$-equivariant $\bfE_\infty$ rings on $\sfM$. In particular, we may form the $H$-geometric fixed points
\begin{equation}\label{eq:geometricfp_asqcoh}
\ul{\calO}_{\bfG}^{\Phi H} \in \CAlg_{\calO_\bfG^{\bfB H}}(\Mod_{\calO_\sfM}(\Shv(\sfM;\Sp))) \simeq \CAlg(\QCoh(\bfG(\bfB H)))
\end{equation}
of $\ul{\calO}_\bfG$. Explicitly, if we write $\widetilde{E}\calP$ for the $H$-space with contractible $K$-fixed points for all proper subgroups $K\subset H$ and empty $H$-fixed points, then as $\widetilde{E}\calP \in \spaces_{H,\ast}$ is idempotent under the smash product and $\widetilde{\ul{\calO}}_\bfG^{\bfB H}[\bs//H]$ is lax symmetric monoidal (see \Cref{ex:gequivhomology}), we may form
\[\ul{\calO}_\bfG^{\Phi H} = \widetilde{\ul{\calO}}_\bfG^{\bfB H}[\widetilde{E} \calP//H] \in \CAlg(\QCoh(\bfG(\bfB H))). \]
Write
\[
\Phi^H \bfG \in \Stk^{\aff}_{/\bfG(\bfB H)}
\]
for the associated \emph{geometric fixed point stack} of $\bfG$, defined as the relative spectrum of (\ref{eq:geometricfp_asqcoh}).

\begin{remark}\label{rmk:eulerclass_and_geo_fixed_points}
Recall that associated with any $H$-representation $V$ is a \emph{representation sphere} $S^V \in \Spc_{H,\ast}$, obtained as the one-point compactification of $V$. The \emph{Euler class} of $V$ is the pointed map $a_V\colon S^0 \to S^V$ induced by the inclusion $0\subset V$. Letting $\ol{\rho}_H$ denote the reduced regular representation of $H$, there is an equivalence
\[
\widetilde{E} \calP \simeq \colim (S^{\ol{\rho}_H}\xrightarrow{a_{\ol{\rho}_H}} S^{\cdot 2\ol{\rho}_H} \xrightarrow{\cdot a_{\ol{\rho}_H}} S^{3\ol{\rho}_H} \xrightarrow{\cdot a_{\ol{\rho}_H}} \dots ) \simeq S^{\infty \ol{\rho}_H},
\]
see for example \cite[Pr.2.7]{derivedindandres}. In particular, $\widetilde{E}\calP \simeq \spherespectrum_H[a_{\ol{\rho}_H}^{-1}]$ in $\Sp_H$.
\end{remark}

\begin{remark}
The fact that $\Phi^H\bfG$ is defined as a stack over $\sfM$ allows one to pull back the oriented $\bfP$-divisible group $\bfG$ to $\Phi^H \bfG$, and so enhance the structure sheaf of $\Phi^H\bfG$ itself to a sheaf of $\pi$-commutative global $\calO_{\Phi^H\bfG}$-algebras. In particular, if $\sfM = \Spec(\A)$ is affine, then this refines the $\bfE_\infty$ ring $\ul{\A}^{\Phi H}_\bfG$ to a $\pi$-commutative global ring
\[
\ul{\A_\bfG^{\Phi H}} \in \CAlg_{\ul{\A}}(\Sp_\pi^\gl)
\]
under $\ul{\A}_\bfG$, satisfying
\[
(\ul{\A_\bfG^{\Phi H}})^F[\sfX]\simeq \A_\bfG^{\Phi H}\otimes_\A \A_\bfG^F[\sfX]
\]
for any $\pi$-finite space $F$ and $\sfX \in (\abglobalspaces)_{/\sfX}$. For example taking $\ul{A}_\bfG = \gKU$ to be the global spectrum of equivariant $K$-theory, we may identify
\[
\gKU^{\Phi C_n}\simeq KU[1/n](\zeta_n)
\]
where $\zeta_n$ is a primitive $n$th root of unity \cite[\textsection 7.7]{tomDieck1979transformation}, and thus
\[
(\ul{\gKU^{\Phi C_n}})^{\Phi C_n}\simeq \gKU^{\Phi C_n}\otimes_{\KU} \gKU^{\Phi C_n}\simeq  \KU[1/n](\zeta_n)^{\times |\Aut(C_n)|};
\]
note that this differs from both $\gKU^{\Phi C_n^2}$ and $\gKU^{\Phi C_{n^2}}$ for $n \neq 0$.
\end{remark}

Our goal in this subsection and the next is to identify the geometric fixed point stacks $\Phi^H\bfG$ more concretely, and to derive some consequences. We begin by considering in this subsection the case where $H$ is nonabelian.

\begin{prop}\label{pr:nonabelian_vanishing}
Let $H$ be a nonabelian group. Then $\Phi^H \bfG\simeq\emptyset$.
\end{prop}
\begin{proof}
By \cref{thm:homologybasechange}, the formation of geometric fixed point stacks is compatible with base change, i.e.\
\[
\Phi^Hf^\ast\bfG\simeq f^\ast\Phi^H\bfG
\]
for any map $f\colon \sfN\to\sfM$ of stacks. We may therefore work locally on $\sfM$ to reduce to the case where $\bfG$ is an oriented $\bfP$-divisible group over a $\bfE_\infty$ ring $\A$ for which $\pi_0 \A$ is local. As $\bfG$ is oriented, it follows that $\A$ is $L_n$-local for some implicit prime $p$ and some finite height $n < \infty$ and therefore so is the $\A$-algebra $\ul{\A}_\bfG^{\Phi H}$. By \cite{hahn2022bousfield}, we find that $\ul{\A}_\bfG^{\Phi H}\simeq 0$ if and only if $\bfQ \otimes \ul{\A}_\bfG^{\Phi G}\simeq 0$. Again, applying \cref{thm:homologybasechange}, we therefore reduce to the case where $\A$ is rational. By \cite[Th.7.6.3, Rmk.7.6.4]{ec3}, every element of $\pi_0^H \ul{\A}_{\bfG}$ is equal to a sum of classes transfered from abelian subgroups of $H$. As geometric fixed points annihilate the image of the transfer, it follows that $1 \in \pi_0^H\ul{\A}_{\bfG}$ is annihilated by the map $\ul{\A}_{\bfG}^{\bfB H} \to \ul{\A}_{\bfG}^{\Phi H}$, and therefore $\ul{\A}_{\bfG}^{\Phi H}\simeq 0$ as claimed.
\end{proof}

\begin{cor}\label{cor:tomdieckesque}
For any finite group $H$, there is a natural comparison map
\[
\calO_\bfG^{\bfB H} \to \prod_{\substack{(L) \leq H \\ L \text{ abelian}}} (\calO_\bfG^{\Phi L})^{\h W_HL}
\]
in $\CAlg(\QCoh(\sfM))$ which is an equivalence if the order of $H$ is invertible in $\sfM$. Here, the product is indexed over the conjugacy classes of abelian subgroups of $H$.
\end{cor}
\begin{proof}
For any subgroup $L\subset H$ and $H$-spectrum $X$, the $L$-geometric fixed point $\Phi^L X$ spectrum carries an action of the Weyl group $W_H L$, and this refines to a functor
\[
\Phi^L\colon \Sp^H \to \Fun(BW_HL,\Sp).
\]
By allowing $L$ to vary over the conjugacy classes of subgroups of $H$, we obtain a functor
\[
\Phi\colon \Sp^H \to \prod_{(L)\leq H}\Fun(BW_HL,\Sp)
\]
which is known to be an equivalence away from the order of $H$. This provides, for any $H$-spectrum object $X$ a natural map
\[
X^H \to \prod_{(L)\leq H} (\Phi^L X)^{\h W_H L}
\]
which is an equivalence when $|H|$ acts invertible on $X$. The corollary follows by applying this to $X = \ul{\calO}_{\bfG,H}$ and observing that summands corresponding to nonabelian subgroups vanish by \cref{pr:nonabelian_vanishing}.
\end{proof}

\begin{ex}
If $H$ is a finite abelian group, then the action of $W_HL \simeq H/L$ on $\calO_\bfG^{\Phi L}$ is rationally trivial, and so \cref{cor:tomdieckesque} provides a map
\[
\calO_\bfG^{\bfB H}\to \prod_{L\subset H}\calO_\bfG^{\Phi L}
\]
which is an equivalence after inverting the order of $H$. Geometrically, this provides a map
\[
\coprod_{L\subset H}\Phi^L \bfG \to \bfG(\bfB H)
\]
which is an equivalence after base change to $\spherespectrum[1/|H|]$. We will identify the stacks $\Phi^L \bfG$ in geometric terms in \cref{ssec:abeliangeofixed} below; compare, for example, the corresponding splittings of equivariant $\TMF$ discussed in \cite[\textsection 4.2]{tmfwls}.
\end{ex}

\subsection{Abelian geometric fixed points}\label{ssec:abeliangeofixed}

It remains to identify $\Phi^H\bfG$ when $H$ is a finite abelian group. By construction, $\bfG(\bfB H)$ may be regarded as a stack of homomorphisms $\dual{H} \to \bfG$. We will show that
\[
\Phi^H\bfG\subset\bfG(\bfB H)
\]
admits an algebro-geometric identification as an open substack corresponding to the \emph{injections} $\dual{H} \to \bfG$. To define this, we require some preliminaries.

\begin{construction}
Given an $\einfty$ ring $\A$, write $|\Spec \A|$ for the Zariski spectrum of the ordinary ring $\pi_0 \A$. This defines a functor to the category of topological spaces
\[
\CAlg^\op \to \Top,\qquad \A \mapsto |\Spec \A|
\]
which satisfies flat descent \cite[Pr.1.6.2.2]{sag}, and therefore extends uniquely to a colimit preserving functor
\[
|\bs|\colon \Stk \to \Top
\]
sending a stack $\sfX$ to its \emph{Zariski space} $|\sfX|$.
\end{construction}

\begin{remark}\label{rmk:interpretation_classicalstacks}
    As $|\Spec \A| = |\Spec \pi_0 \A|$ by definition, for a stack $\sfM$ presented as a colimit of affines $\sfM \simeq \colim \Spec \A$, we have the identification
    \[|\sfM| \simeq \colim |\Spec \A| \simeq \colim |\Spec \pi_0 \A| \simeq |\sfM^\heartsuit|.\]
    In other words, the Zariski space construction only depends on underlying classical stacks in the sense of \cite[Df.3.1.1.7]{reconstruction}, and the colimit preserving functor $|\bs| \colon \Stk \to \Top$ factors as
    \[\Stk \xrightarrow{(\bs)^\heartsuit} \Stk^\heartsuit \xrightarrow{|\bs|} \Top.\]
\end{remark}

\begin{mydef}
A morphism $\sfZ \to \sfX$ of stacks is a \emph{closed immersion} if it is affine and for every Cartesian diagram of the form
\begin{center}\begin{tikzcd}
\Spec \B\ar[r]\ar[d]&\Spec \A\ar[d]\\
\sfZ\ar[r]&\sfX
\end{tikzcd},\end{center}
the map $\pi_0 \A \to \pi_0 \B$ is a surjection.
\end{mydef}

\begin{ex}
Let $p\colon \sfE \to \sfX$ be an affine morphism which admits a section $s\colon \sfX \to \sfE$. Then $s$ is a closed immersion. Indeed, by working locally on $\sfX$ we reduce to the case where $p$ is of the form $\Spec \B \to \Spec \A$ for a ring map $\A \to \B$ which admits a retraction, and is therefore a surjection on all homotopy groups.
\end{ex}

\begin{lemma}\label{lm:opencomplements}
Let $\{\sfZ_i \to \sfX\}$ be a finite family of closed immersions. Then there is a substack $\sfX\setminus\bigcup_i\sfZ_i\subset\sfX$, the \emph{open complement of $\bigcup_i\sfZ_i$ in $\sfX$}, uniquely characterized by the existence of natural Cartesian diagrams of spaces
\begin{equation}\label{eq:UP_open_complement}\begin{tikzcd}
\Map_\Stk(\sfY,\sfX\setminus\sfZ)\ar[r]\ar[d,tail]&\Map_{\Top}(|\sfY|,|\sfX|\setminus\bigcup_i|\sfZ|)\ar[d]\\
\Map_\Stk(\sfY,\sfX)\ar[r]&\Map_{\Top}(|\sfY|,|\sfX|).
\end{tikzcd}\end{equation}
for all $\sfY \in \Stk$. Moreover, this construction satisfies the following properties:
\begin{enumerate}
\item $\sfX\setminus(\sfZ_1\cup\cdots\cup \sfZ_n)\simeq (\sfX\setminus\sfZ_1)\times_\sfX\cdots\times_\sfX(\sfX\setminus\sfZ_n)$;
\item $|\sfX\setminus\bigcup_i \sfZ_i|\simeq |\sfX|\setminus\bigcup_i|\sfZ_i|$;
\item For any morphism $\sfY \to \sfX$, there is an equivalence $\sfY\times_\sfX (\sfX\setminus\sfZ)\simeq \sfY\setminus(\sfY\times_\sfX \sfZ)$.
\end{enumerate}
\end{lemma}
\begin{proof}
If $\sfZ \to \sfX$ is a closed immersion, then as closed immersions of topological spaces are closed under colimits it follows that $|\sfZ| \to |\sfX|$ is a closed immersion. It follows that the indicated Cartesian diagram determines a limit preserving functor
\[
\Map_\Stk(\bs,\sfX\setminus\bigcup_i \sfZ_i)\colon \Stk^\op\to\spaces.
\]
In particular, this functor is determined by its restriction to $\CAlg\subset\Stk^\op$. The claim is that it is representable and satisfies properties (1--3). Property (1) follows just from the fact that
\[
|\sfX|\setminus(|\sfZ_1|\cup\cdots\cup |\sfZ_n|) \simeq (|\sfX|\setminus |\sfZ_1|)\times_{|\sfX|}\cdots\times_{|\sfX|} (|\sfX|\setminus |\sfZ_n|).
\]
In particular, we may reduce to the case of a single closed immersion $\sfZ \to \sfX$ in what follows.

We next verify that the functor
\[
\Map_\Stk(\Spec(\bs),\sfX\setminus\sfZ)\colon \CAlg\to\spaces
\]
is representable by a stack. First, suppose that $\sfX = \Spec \A$ is affine. By assumption, $\sfZ = \Spec \B$ for some $\B$ equipped with a map $\A \to \B$ which induces a surjection $\pi_0 \A \to \pi_0 \B$, and therefore a closed immersion $|\Spec \B|\to |\Spec \A|$. Let $\SpSch^\nc$ denote the category of nonconnective spectral schemes. By \cite[Pr.3.2.2.2]{reconstruction} and \cite[Th.1.6.2.1]{sag}, the Yoneda embedding $\Spec\colon \Aff\to\Stk$ extends to a fully faithful embedding
\[
h_{(\bs)}\colon \SpSch^\nc \to \Stk.
\]
Given an $\einfty$ ring $\A$, write $\zSpec \A$ for the locally ringed space constructed in \cite[\textsection1.1.4]{sag} with underlying space $|\Spec \A|$. Then a complement $\Spec \A \setminus \Spec \B$ is given by the stack represented by the restriction of the locally ringed space $\zSpec \A$ to the open subspace $|\Spec \A|\setminus |\Spec \B|$.

Next, suppose that $\sfX$ is arbitrary. Choose a presentation $\sfX\simeq \colim_j \Spec \A_j$ of $\sfX$ as a colimit of affines, and write $\Spec \B_j = \Spec \A_j \times_\sfX \sfZ$. By the above construction, we have for each $j$ an open complement $\Spec \A_j \setminus \Spec \B_j \subset \Spec \A_j$. Each map $\Spec \A_j \to \Spec \A_{j'}$ in the defining diagram for $\sfX$ sends $\Spec \A_j \setminus \Spec \B_j$ into $\Spec \A_{j'}\setminus \Spec \B_{j'}$, and so we may form $\sfY = \colim_j \Spec \A_j \setminus \Spec \B_j$. We claim $\sfY = \sfX\setminus\sfZ$. In other words, we claim that for any $\einfty$ ring $C$, there is a natural Cartesian diagram
\begin{center}\begin{tikzcd}
\Map_{\Stk}(\Spec C,\sfY)\ar[r]\ar[d]&\Map_{\Top}(|\Spec C|,|\sfX|\setminus|\sfZ|)\ar[d]\\
\Map_\Stk(\Spec C,\sfX)\ar[r]&\Map_{\Top}(|\Spec C|,|\sfX|)
\end{tikzcd}.\end{center}
The existence of such a diagram follows from the fact that
\begin{align*}
|\sfY|\simeq \colim_j |\Spec \A_j \setminus \Spec \B_j|&\simeq \colim_j |\Spec \A_j|\setminus |\Spec \B_j|\\
&\simeq (\colim_j|\Spec \A_j|)\setminus (\colim_j |\Spec \B_j|)\simeq |\sfX|\setminus|\sfZ|
\end{align*}
by construction, so we must show that it is Cartesian. As the vertical morphisms are monic, it suffices to show that if $f\colon \Spec C \to \sfX$ is a map for which $|\Spec C| \to |\sfX|$ factors through $|\sfX|\setminus|\sfZ|$, then $f$ lifts to $\sfY$. 

Fix such a morphism. By the given presentation of $\sfX$, there exists a collection of jointly faithfully flat extensions $C \to D_j$ commutative diagrams
\begin{center}\begin{tikzcd}
\Spec D_j\ar[r]\ar[d]&\Spec \A_j\ar[d]\\
\Spec C\ar[r]&\sfX
\end{tikzcd}.\end{center}
As $|\Spec D_j| \to |\Spec \A_j| \to |\sfX|$ factors through $|\sfX|\setminus|\sfZ|$, it follows that $|\Spec D_j|\to|\Spec \A_j|$ factors through $|\Spec \A_j|\setminus |\Spec \B_j|$, and therefore $\Spec D_j \to \Spec \A_j$ factors through $\Spec \A_j \setminus \Spec \B_j$. In other words, we have a commutative diagram of the form
\begin{center}\begin{tikzcd}
\Spec D_j \ar[r]\ar[dd]&\Spec \A_j\setminus \Spec \B_j\ar[d]\\
&\sfX\setminus\sfZ\ar[d]\\
\Spec C\ar[r]&\sfX
\end{tikzcd}.\end{center}
As $\coprod_j \Spec D_j \to \Spec C$ is an effective epimorphism and $\sfX \setminus \sfZ \to \sfX$ is monic, it follows that $\Spec C \to \sfX$ lifts to $\sfX\setminus\sfZ$ as claimed.

Properties (2) and (3) are now clear from the construction.
\end{proof}

\begin{cor}
    Let $\{\sfZ_i \to \sfX\}$ be a finite family of closed immersions over a base stack $\sfM$. Then considering $ \sfX \setminus \bigcup_i \sfZ_i$ as a stack over $\sfM$ via $\sfX \to \sfM$, it is uniquely characterized by the natural Cartesian diagram of spaces
    \[\begin{tikzcd}
\Map_{\Stk_{/\sfM}}(\sfY,\sfX\setminus\sfZ)\ar[r]\ar[d,tail]&\Map_{\Top_{/|\sfM|}}(|\sfY|,|\sfX|\setminus\bigcup_i|\sfZ|)\ar[d]\\
\Map_{\Stk_{/\sfM}}(\sfY,\sfX)\ar[r]&\Map_{\Top_{/|\sfM|}}(|\sfY|,|\sfX|).
\end{tikzcd}\]
    for all $\sfY\in \Stk_{/\sfM}$.
\end{cor}

\begin{proof}
    The desired square is the fibre of the map of squares from (\ref{eq:UP_open_complement}) to 
    \[\begin{tikzcd}
        {\Map_{\Stk}(\sfY, \sfM)}\ar[r]\ar[d, "="]  &   {\Map_{\Top}(|\sfY|, |\sfM|)}\ar[d, "="]    \\
        {\Map_{\Stk}(\sfY, \sfM)}\ar[r]             &   {\Map_{\Top}(|\sfY|, |\sfM|)}
    \end{tikzcd}\]
    over the structure maps $\sfY \to \sfM$ and $|\sfY| \to |\sfM|$.
\end{proof}

To apply this to our situation, we require the following lemma.

\begin{lemma}\label{lem:injexists}
Let $\bfG\colon \Ab_\fin^\op\to\Stk$ be a functor satisfying the following conditions:
\begin{enumerate}
\item $\bfG$ sends biCartesian squares to Cartesian squares;
\item $\bfG[H] \to \bfG[0]$ is affine for all finite abelian groups $H$.
\end{enumerate}
Then, for any surjection $H \to K$ of finite abelian groups, the induced map $\bfG[K] \to \bfG[H]$ is a closed immersion.
\end{lemma}

\begin{rmk}
The first condition of \cref{lem:injexists} is equivalent to asking that $\G$ defines a \emph{torsion object} in $\Stk_{/\G[0]}$ in the sense of \cite[Df.6.4.2]{ec1}.
\end{rmk} 

\begin{proof}[Proof of \Cref{lem:injexists}]
Let $H \to K$ be a surjection with kernel $L$, and consider the diagram
\begin{center}\begin{tikzcd}
{\bfG[K]}\ar[r]\ar[d]&{\bfG[H]}\ar[d]\\
{\bfG[0]}\ar[r]\ar[dr]&{\bfG[L]}\ar[d]\\
&{\bfG[0]}
\end{tikzcd}.\end{center}
This square is Cartesian, so it suffices to show that $\bfG[0] \to \bfG[L]$ is a closed immersion. This holds as it is a section to the affine morphism $\bfG[L]\to\bfG[0]$.
\end{proof}

\begin{mydef}\label{def:injstack}
Let $\bfG\colon \Ab_\fin^\op\to\Stk$ be a functor satisfying the two conditions of \cref{lem:injexists}. Given a finite abelian group $H$, we write
\[
\Inj(\dual{H},\bfG) = \bfG[\dual{H}]\setminus\bigcup_{K\subsetneq H}\bfG[\dual{K}]
\]
for the \emph{stack of injections} $\dual{H} \to \bfG$.
\end{mydef}

\begin{ex}\label{ex:injlevel}
Let $\sfE \to \sfM$ be an elliptic curve over a stack $\sfM$, and write $\sfE[\bfP^\infty]\subset\sfE$ for the $\bfP$-divisible group of $\sfE$. If the order of $H$ is invertible in $\sfM$, then we may identify
\[
\Inj(\dual{H},\sfE[\bfP^\infty])\simeq\operatorname{Level}(\dual{H},\sfE)
\]
as the stack of level $\dual{H}$-structures on $\sfE$.
\end{ex}

\begin{ex}
Let $\bfG$ be a $\bfP$-divisible group. Then, for any finite abelian group $H$, the inclusion of the summand $\bfG_{(|H|)}\subset \bfG$ induces an equivalence
\[
\Inj(\dual{H},\bfG_{(|H|)})\simeq\Inj(\dual{H},\bfG),
\]
and if $H$ and $K$ are finite abelian groups of coprime order, then the comparison map
\[
\Inj(\dual{H\times K},\bfG)\to\Inj(\dual{H},\bfG)\times_\sfM\Inj(\dual{K},\bfG)
\]
is an equivalence.
\end{ex}

We can now state the following theorem.

\begin{theorem}\label{thm:geofixed}
Let $\bfG$ be an oriented $\bfP$-divisible group over a complex periodic stack $\sfM$. Then the inclusion $\Inj(H,\bfG) \to \bfG[H]$ is affine, and there is a (necessarily unique) equivalence
\[
\Phi^H \bfG \simeq \Inj(\dual{H},\bfG)
\]
of stacks affine over $\bfG(\BH) = \bfG[\dual{H}]$. 
\end{theorem}

\begin{rmk}
The complex periodic hypothesis in \cref{thm:geofixed} is an essentially rational hypothesis, as every $p$-complete $\E_\infty$ ring which supports an orientable $\bfP$-divisible group is necessarily complex periodic \cite[Warn.2.6.18]{ec3}.
\end{rmk}

The proof of \cref{thm:geofixed} is based on the following simple observation.

\begin{lemma}\label{lem:cpgeofixed}
Let $G$ be a compact Lie group and $V$ be a $1$-dimensional complex $G$-representation with kernel $K$. Fix an $\E_\infty$ ring $A\in \CAlg(\Sp_G)$ in $G$-spectra and define ideals
\[
I_K = \ker\left(\res^G_K\colon \pi_\ast \A^G \to \pi_\ast \A^K\right),\qquad I_V = \image\left(a_V\colon \pi_{\ast+V}^G \A \to \pi_\ast^G \A = \pi_\ast \A^G\right).
\]
Then there are inclusions
\[
I_V\subset I_K \subset \sqrt{I_V}.
\]
In particular, $\sqrt{I_K} = \sqrt{I_V}$.
\end{lemma}
\begin{proof}
Let $S(V)$ and $S^V$ denote the unit sphere in and the one-point compactification of $V$, respectively, and consider the diagram
\begin{center}\begin{tikzcd}
G/K_+\ar[dr]\ar[d]\\
S(V)_+\ar[r]\ar[d]&S^0\ar[r, "{a_V}"]&S^V\\
\Sigma G/K_+
\end{tikzcd}\end{center}
in which both the row and column are cofiber sequences. Taking $\A$-cohomology, we obtain a diagram of the form
\begin{center}\begin{tikzcd}
\pi_\ast^K \A\\
\A_G^\ast(S(V))\ar[u]&\pi_\ast^G \A\ar[ul,"\res^G_K"']\ar[l,"i"']&\pi_{\ast+V}^G \A\ar[l,"a_V"]\\
\pi_{\ast+1}^G \A\ar[u,"\partial"]
\end{tikzcd}.\end{center}
Clearly, $I_V \subset I_K$. We claim that $\sqrt{I_K}\subset I_V$. If $x\in I_K$, then $i(x) = \partial(y)$ for some $y$. We claim that $\partial(y)$ is nilpotent. It suffices to prove that $\Phi^H \partial(y)$ is nilpotent for all $H\subset G$. If $H = e$ then this is clear, as $\partial$ is the inclusion of the exterior element of
\[
\A_e^\ast \res_G^e S(V)_+ \cong \A_e^\ast(S^1_+)\cong \A_\ast[\epsilon]/(\epsilon^2).
\]
If $H \neq e$ then $\Phi^H S(V) = \emptyset$ and so $\Phi^H \partial(y) = 0$.

Thus we have $i(x^n) = i(x)^n = \partial(y)^n = 0$ for some $n > 0$. It follows that $i(x^n) = a_V z$ for some $z$. This proves $\sqrt{I_Z} \subset I_V$.
\end{proof}

\begin{proof}[Proof of \cref{thm:geofixed}]
This theorem is local on $\sfM$, so without loss of generality $\bfG$ is an oriented $\bfP$-divisible group over an $\einfty$ ring $\A$. In this case, we claim that there is an equivalence
\[
\Spec(\ul{\A}_\bfG^{\Phi H})\simeq \Inj(\dual{H},\bfG)
\]
of stacks over $\bfG(\BH)$. To that end, it suffices to show first that the map $\Spec(\ul{\A}_\bfG^{\Phi H}) \to \bfG(\BH)$ factors through the subobject $\Inj(\dual{H},\bfG)\subset\bfG(\BH)$, and second that the resulting map $\Spec(\ul{\A}_\bfG^{\Phi H}) \to \Inj(\dual{H},\bfG)$ is an equivalence. Both statements hold if and only if they hold after pulling back along any collection of maps $\{\Spec \B_i \to \Spec \A\}$ for which $\coprod_i |\Spec \B_i| \to |\Spec \A|$ is surjective and the functors $\B_i \otimes_\A(\bs)\colon \Mod_\A\to\Mod_{\B_i}$ are jointly conservative. As
\[
\B \otimes_\A \ul{\A}_\bfG^{\Phi H}\simeq \ul{\B}_\bfG^{\Phi H},\qquad \Spec \B \times_{\Spec \A}\Inj(\dual{H},\bfG)\simeq \Inj(\dual{H},\Spec \B\times_{\Spec \A} \bfG)
\]
for any $\A$-algebra $\B$, we may therefore reduce to the case where $\A$ is $p$-complete or rational and $\pi_0 \A$ is a local ring. After possibly taking a further faithfully flat extension, we may assume that $\bfG = \bfG_p^\circ \oplus \ul{\Lambda}$ for some colattice $\Lambda$, (where $\bfG_p^\circ = 0$ if $\A$ is rational).

Let $V$ be a $1$-dimensional complex $H$-representation with kernel $K$. We claim that there exists an invertible element $\beta_V \in \pi_V^H \ul{\A}_{\bfG,H}$ in the $RU(H)$-graded homotopy groups of $\ul{\A}_{\bfG,H}$. Given this, it follows from \cref{lem:cpgeofixed} that if we set $e_V = a_V \beta_V \in \pi_0 \ul{\A}_\bfG^{\BH}$, then
\[
\Spec(e_V^{-1} \ul{\A}_\bfG^{\BH}) \simeq \bfG[\dual{H}]\setminus \bfG[\dual{K}].
\]
Taking the intersection of these subschemes over $\bfG[\dual{H}]$ as $V$ varies over the nontrivial characters of $H$, we obtain an equivalence
\begin{equation}\label{eq:geofp_as_invert_euler}
\Spec(\ul{\A}_\bfG^{\Phi H})\simeq \Spec(a_{\ol{\rho}_H}^{-1}\ul{\A}_{\bfG}) \simeq \Spec(\{a_{V}^{-1}\}\ul{\A}_{\bfG})\simeq \Spec(\{e_{V}^{-1}\}\ul{\A}_{\bfG})\ \simeq \Inj(\dual{H},\bfG),
\end{equation}
as claimed.

The existence of $\beta_V \in \pi_V^H \ul{\A}_{\bfG,H}$ is established as follows. By the above reductions, $\bfG \simeq \bfG_p^\circ \oplus \ul{\Lambda}$. The tempered theory $\ul{\A}_{\bfG_p^\circ}$ is the Borel theory associated with $\A$, and so the character isomorphism of \cite[Cor.4.3.4]{ec3} takes the form
\[
\pi_V^H \ul{\A}_{\bfG,H} \cong \prod_{\alpha\colon \dual{\Lambda} \to H} \A^0(S^{V^{\image(\alpha)}}_{\h H}).
\]
By assumption, $\A$ is complex periodic. As $\pi_0 \A$ is a local ring, it follows that we may choose a periodic complex orientation of $\A$ yielding, for every complex vector bundle $\xi$ over a space, a natural Bott class $\beta_\xi \in \A^0 \Th(\xi)$. In particular, there are Bott classes $\beta_{V^{\image(H)}} \in \A^0(S^{V^{\image(\alpha)}}_{\h H})$ for each map $\alpha\colon \dual{\Lambda} \to H$ which together yield a class $\beta_V \in \pi_V^H \ul{\A}_{\bfG,H}$. As the construction of $\beta_V$ is natural in $H$, we find that multiplication by $\beta_V$ defines a map
\[
\beta_V\colon \Sigma^{-V}\ul{\A}_{\bfG,H} \to \ul{\A}_{\bfG,H}
\]
which is an isomorphism on integer-graded homotopy groups $\pi_\ast^K$ for all $K \subset H$, and is therefore an equivalence of $H$-spectra. Thus $\beta_V$ is a unit in the $RU(H)$-graded homotopy groups of $\ul{\A}_{\bfG,H}$ as needed.
\end{proof}

The following is an immediate consequence of \Cref{thm:geofixed}.

\begin{cor}
    Let $\bfG$ be an oriented $\bfP$-divisible group over a stack $\sfM$. Then, for any finite abelian group $H$, the natural map
    \[\Inj(\dual{H}, \bfG) \to \sfM\]
    is affine.
\end{cor}
\begin{proof}
This follows from the equivalence $\Inj(\dual{H},\bfG)\simeq\Phi^H\bfG$ as $\Phi^H\bfG \to \sfM$ is affine by construction.
\end{proof}

Given a height function $\vec{h} \in \{h(p)  \in \bfZ_{\geq 0}\cup\{\infty\}: p \in \bfP\}$, say that a $\bfP$-divisible group $\bfG$ has \emph{height $\leq \vec{h}$} if its underlying $p$-divisible group $\bfG_{(p)}$ is of height $\leq h(p)$ for all primes $p$.

\begin{cor}[Blueshift]\label{cor:blueshift}
Let $\bfG$ be an oriented $\bfP$-divisible group over a complex periodic stack of height $\leq \vec{h}$. Fix a finite abelian group $H$, and define $n(p) = \log_p|H/pH|$. Then the complex periodic stack $\Phi^H\bfG$ is of height $\leq \vec{h} - \vec{n}$. In particular, $\Phi^H\bfG = \emptyset$ if $n(p) > h(p)$ for some prime $p$.
\end{cor}
\begin{proof}
This statement is local on the base stack, so we may reduce to the case where $\bfG$ is an oriented $\bfP$-divisible group over an $\einfty$ ring $\A$. As
\[
\Spec \pi_0 \A_\bfG^{\Phi H}\simeq \Inj(H^\vee,\bfG^\heartsuit),
\]
where $\bfG^\heartsuit$ is the $\bfP$-divisible group over $\pi_0\A$ associated with $\bfG$, the corollary amounts to the observation if $\bfG^\heartsuit$ is an ordinary $\bfP$-divisible group over a field, then $\Inj(H^\vee,\bfG^\heartsuit) = \Inj(H^\vee,(\bfG^\heartsuit)^\et) = \emptyset$ whenever $\log_p |H/pH| > \height (\bfG_{(p)}^\heartsuit)^\et$ for some prime $p$.
\end{proof}

\begin{ex}
    For a fixed prime $p$ and any $0< h < \infty$, any complex-periodic $K(h)$-local $\E_\infty$-ring $\A$ comes equipped with a canonical oriented Quillen $p$-divisible group $\G^\calQ_\A$ of height $h$; see \cite[Ths.4.5.3 \& 4.6.16]{ec2}. By \Cref{cor:blueshift}, we see that $\Phi^H \G^\calQ_\A$ is empty if $H$ is a finite abelian $p$-group of $p$-rank greater than $h$. In particular, as the global sections functor
    \[\QCoh(\G^\calQ_\A(\bfB H)) \simeq \QCoh(\Spec \A^{BH}) \xrightarrow{\simeq} \Mod_{\A^{BH}}\]
    is an equivalence, we see that
    \[\Ga(\ul{\O}_{\G^\calQ_\A}^{\Phi H}) \simeq \ul{\A}_{\G^\calQ_\A}^{\Phi H} \simeq \A^{\tau H},\]
    the \emph{proper Tate construction}, is zero for such $H$. For finite abelian $p$-groups with $p$-rank at most $h$, it is classical that $\A^{\tau H}$ is nonzero. Combining these with the other vanishing results above, we recover the \emph{derived defect base} for the Borel equivariant theory defined by $\A$. This is a slight generalization of \cite[Pr.5.26]{derivedindandres} which does not require restrictions on the coefficients of $\A$; also compare with \cite[Pr.5.36]{derivedindandres}.
\end{ex}

The same line of reasoning can similarly be used to compute the derived defect base for $\gKU$ and $\gKO$. Our last example is for global topological modular forms.

\begin{ex}\label{ex:geoFP_TMF}
Let $\calE \to \sfM_\Ell^{\ori}$ be the universal oriented elliptic curve, with associated oriented $\bfP$-divisible group $\calE[\bfP^\infty]$, as in \cref{ex:ambidextroustmf}. \Cref{thm:geofixed} then gives an identification
\[
\Phi^H \calE[\bfP^\infty] \simeq \Inj(\dual{H}, \calE[\bfP^\infty])
\]
of stacks, which, upon taking global sections, yields an identification
\[
\gTMF^{\Phi H} = (\Ga(\ul{\O}_{\calE[\P^\infty]}))^{\Phi H} \simeq \Ga(\ul{\O}_{\calE[\P^\infty]}^{\Phi H}) = \Ga(\Phi^H \calE[\P^\infty]) \simeq \Ga \Inj(\dual{H},\calE)
\]
of $\bfE_\infty$ $\TMF$-algebras by \Cref{rmk:0semiaffine}, as $\M_\Ell^\ori$ is $0$-affine; see \cite{akhilandlennart} or \cite[Th.A]{reconstruction}. After inverting the order of $H$, injections correspond to \emph{level structures}, and in this way the geometric fixed points recover and refine previously considered forms of $\TMF$. For example, 
\[
\gTMF^{\Phi C_n}[1/n]\simeq\TMF_1(n),\qquad \gTMF^{\Phi C_n^2}\simeq \TMF(n),\qquad \TMF^{\Phi C_n^3}\simeq 0.
\]
Here, $n$ is already invertible in $\gTMF^{\Phi C_n^2}$ and $\gTMF^{\Phi C_n^3}\simeq 0$ by \cref{cor:blueshift}. In particular, $\gTMF^{\Phi C_n}$ can be regarded as an integral refinement of $\TMF_1(n)$, which has height $1$ at all primes dividing $n$, in contrast to $\TMF_1(n)$ which has height $0$ at all primes dividing $n$. To see this, it suffices to show that for every prime $p$, there exists a map $\gTMF^{\Phi C_n} \to \A$ to a $K(1)$-local $\bfE_\infty$ ring $\A$. Indeed, the modular interpretation
\[\Phi^{C_n} \calE[\bfP^\infty] = \Inj(\widehat{C_n}, \calE[\P^\infty])\]
implies that there is a map $\gTMF^{\Phi C_n} \to \A$ for any elliptic cohomology theory $\A$ associated with an oriented elliptic curve $E$ which admits a subgroup isomorphic to the constant group scheme $\ul{\dual{C_n}}$. Examples which are $K(1)$-local are then obtained by asking that $\A$ is a complete local ring with residue field $\kappa = \pi_0\A/\mathfrak{m}$ of characteristic $p$ and that $E$ restricts to an ordinary elliptic curve over $\kappa$; for example, one may take an separable closure of $L_{K(1)}E_2$ for $E_2$ the Morava $E$-theory associated with a supersingular elliptic curve.
\end{ex}

\addcontentsline{toc}{section}{References}
\bibliography{reference}
\bibliographystyle{alpha}
\end{document}